\documentclass[11pt,twoside,a4paper]{article}

\usepackage[english]{babel}
\usepackage{a4wide}
\usepackage{graphicx}
\usepackage{amsmath,amssymb,amsgen}
\usepackage{amsthm} 
\ifcsname XeTeXversion\endcsname\else\usepackage[cp1252]{inputenc}\fi
\usepackage{listings}
\usepackage{color}
\usepackage[usenames,dvipsnames]{xcolor}
\usepackage{rotating}
\usepackage{hhline}
\usepackage{multirow}
\usepackage{enumerate}
\usepackage[bookmarks]{}
\usepackage{amsfonts}   
\usepackage{dsfont}
\usepackage{tikz}
\usepackage{booktabs}
\usepackage{stmaryrd} 

\long\def\drop#1{}

\usepackage{thmtools}
\newtheorem{thm}{Theorem}[section]
\newtheorem{prop}[thm]{Proposition}
\newtheorem{lem}[thm]{Lemma}

\newtheorem{cor}[thm]{Corollary}
\newtheorem{defn}[thm]{Definition}

\newtheorem{mthm}{Theorem}   

\newtheorem{mcor}[mthm]{Corollary}   

\newtheorem{mlem}[mthm]{Lemma}   

\declaretheoremstyle[bodyfont=\normalfont,qed=$\square$]{remark}
\declaretheorem[style=remark,numberlike=thm,name=Remark]{rem}
\declaretheorem[style=remark,numberlike=thm,name=Example]{ex}

\def\XXint#1#2#3{{\setbox0=\hbox{$#1{#2#3}{\int}$ }
\vcenter{\hbox{$#2#3$ }}\kern-.6\wd0}}

\newcommand{\e}{\varepsilon}

\newcommand{\Lip}{\operatorname{Lip}}
\newcommand{\loc}{\operatorname{loc}}
\newcommand{\N}{\mathbb N}
\newcommand{\R}{\mathbb R}
\newcommand{\sign}{\operatorname{sign}}

\newcommand{\xto}[1]{\xrightarrow{#1}}

\newcommand{\Z}{\mathbb Z}

\newcommand{\bb}{\mathbf b}
\newcommand{\bM}{\mathbf M}
\newcommand{\bone}{\boldsymbol 1}
\newcommand{\bx}{\mathbf x}
\newcommand{\by}{\mathbf y}

\newcommand{\cF}{\mathcal F}
\newcommand{\cI}{\mathcal I}
\newcommand{\cM}{\mathcal M}
\newcommand{\cZ}{\mathcal Z}

\newcommand{\pv}{\mathrm{pv}}

\DeclareFontFamily{U}{mathx}{\hyphenchar\font45}
\DeclareFontShape{U}{mathx}{m}{n}{
      <5> <6> <7> <8> <9> <10>
      <10.95> <12> <14.4> <17.28> <20.74> <24.88>
      mathx10
      }{}
\DeclareSymbolFont{mathx}{U}{mathx}{m}{n}
\DeclareFontSubstitution{U}{mathx}{m}{n}
\DeclareMathAccent{\widecheck}{0}{mathx}{"71}

\usepackage{ifthen}
\newlength{\leftstackrelawd}
\newlength{\leftstackrelbwd}
\def\leftstackrel#1#2{\settowidth{\leftstackrelawd}%
{${{}^{#1}}$}\settowidth{\leftstackrelbwd}{$#2$}%
\addtolength{\leftstackrelawd}{-\leftstackrelbwd}%
\leavevmode\ifthenelse{\lengthtest{\leftstackrelawd>0pt}}%
{\kern-.5\leftstackrelawd}{}\mathrel{\mathop{#2}\limits^{#1}}}

\ifcsname XeTeXversion\endcsname

\else
\usepackage{bbm} 

\fi

\usepackage[colorlinks,citecolor=magenta,linkcolor=magenta]{hyperref}

\renewcommand{\o}{\overline}

\title{Discrete-to-continuum convergence of charged particles in 1D with annihilation}
\author{Patrick van Meurs, Mark Peletier, Norbert Po\v{z}\'ar}

\begin{document}

\maketitle

\begin{abstract}
We consider a system of charged particles moving on the real line driven by electrostatic interactions. Since we consider charges of both signs, collisions might occur in finite time. Upon collision, some of the colliding particles are effectively removed from the system (annihilation). The two applications we have in mind are vortices and dislocations in metals. 

In this paper we reach two goals. First, we develop a rigorous solution concept for the interacting particle system with annihilation. The main innovation here is to provide a careful management of the annihilation of groups of more than two particles, and we show that the definition is consistent by proving existence, uniqueness, and continuous dependence on initial data. The proof relies on a detailed analysis of ODE trajectories close to collision, and a reparametrization of vectors in terms of the moments of their elements.  

Secondly, we pass to the many-particle limit (discrete-to-continuum), and recover the expected limiting equation for the particle density. Due to the singular interactions and the annihilation rule, standard proof techniques of discrete-to-continuum limits do not apply. In particular, the framework of measures seems unfit. Instead, we use the one-dimensional feature that both the particle system and the limiting PDE can be characterized in terms of Hamilton--Jacobi equations. While our proof follows a standard limit procedure for such equations, the novelty with respect to existing results lies in allowing for stronger singularities in the particle system by exploiting the freedom of choice in the definition of viscosity solutions. 
\end{abstract}

\tableofcontents

\section{Introduction}

Our starting point is the interacting particle system formally given by
   \begin{equation} \label{Pn} \tag{$P_n$}
   \left\{ \begin{aligned}
     &\frac{dx_i}{dt} = \frac{1}{n} \sum_{ \substack{ j = 1 \\ j \neq i} }^n \frac{b_i b_j}{x_i - x_j}
     && t \in (0,T), \ i = 1,\ldots, n \\
     &+ \text{annihilation upon collision,}
     &&
   \end{aligned} \right.
\end{equation}
where $n \geq 2$ is the number of particles, $\bx := (x_1, \ldots, x_n)$ are the particle positions in $\R$, and $\bb := (b_1, \ldots, b_n)$ are the charges of the particles, which are initially set as $+1$ or $-1$. One can think of this particle system as charged particles in a viscous fluid. Indeed, the right-hand side of the ODE shows that the interaction forces are nonlocal and singular. Moreover, equal-sign charges repel and opposite-sign charges attract each other. 

Due to these attractive forces, particles of opposite charge may collide in finite time. Since the right-hand side of the ODE does not vanish prior to the collision (in fact, it blows up),  the ODE is ill-defined at the collision time, and a modeling choice is to be made on how to continue it. Here we make the choice that opposite signs \emph{annihilate} each other: when two particles of opposite charge collide, they are removed from the system, and the remaining particles continue to evolve by \eqref{Pn}. Technically we encode this annihilation by setting the charges $b_i$ of the colliding particles to $0$. This is equivalent to removing them from the system because the ODE assigns zero velocity to particles with zero charge and vice versa, particles with zero charge do not exert a force on any other particle. In analogy with electric charges we call particles with zero charge \emph{neutral}, and particles with nonzero charge \emph{charged}.

A rigorous description of \eqref{Pn} is given in Definition \ref{defn:Pn}, and Figure \ref{fig:trajs} illustrates an example set of trajectories. To give some insight into the properties of \eqref{Pn}, we list here several observations: 
\begin{itemize}
  \item \textit{Gradient flow}. In between collisions, the ODE is the gradient flow of the energy
  \begin{equation} \label{E}
    E(\bx; \bb) = \frac1{2n^2} \sum_{i=1}^n \sum_{ \substack{ j = 1 \\ j \neq i} }^n b_i b_j V (x_i - x_j),
    \qquad V(x) := \log \frac1{|x_i - x_j|}, 
    \end{equation}  
  where $V$ is the particle interaction potential.
  \item \textit{Particle exchangeability}. \eqref{Pn} is invariant under relabeling the indices of particles with the same charge.
  
  \item \textit{Multiple-particle collisions}. The most likely collision is that between two charged particles of opposite charge. Yet, there are possible situations in which three or more charged particles collide; see Figure~\ref{fig:trajs}. 
  
  \item \textit{Translation invariance}. If $(\bx, \bb)$ is a solution to $(P_n)$, then $(\bx + a \bone, \bb)$ also is a solution, for any $a \in \R$ (here $\bone = (1,\ldots,1)$).
  
  \item \textit{Scale invariance}. If $t\mapsto (\bx(t), \bb(t))$ is a solution to $(P_n)$, then $t\mapsto (\alpha \bx(\alpha^{-2} t), \bb(\alpha^{-2} t))$ also is a solution for any $\alpha > 0$.
  
  \item \textit{Conserved quantities}. The first moment $M_1 (x_i) := \sum_{i=1}^n x_i$ and the net charge $\sum_{i=1}^n b_i$ are conserved under \eqref{Pn}.
\end{itemize}

\begin{figure}[ht]
\centering
\begin{tikzpicture}[scale=1.5, >= latex]
\def \sqtwo {1.414}
\def \rr {0.03}

\draw[->] (0,0) -- (0,3.5) node[above] {$t$};
\draw[->] (-2,0) -- (4.8,0) node[right] {$x$};

\draw[dotted] (0,.5) node[left]{$\tau_1$} -- (3.581,.5);
\draw[dotted] (0,1.5) node[right]{$\tau_2$} -- (-1,1.5);
\draw[dotted] (0,3) node[left]{$\tau_3$} -- (2,3);

\begin{scope}[shift={(2,3)},scale=1] 
    \draw[thick, red] (0,0) -- (0,.5);
    \draw[thick, blue] (0,-1.5) -- (0,0);
    \draw[domain=-1.732:1.581, smooth, thick, red] plot (\x,{-\x*\x});           
    \draw[domain=1.581:1.732, smooth, thick, blue] plot (\x,{-\x*\x});           
    \fill[black] (0,0) circle (\rr); 
\end{scope}

\begin{scope}[shift={(2,1.5)},scale=1, rotate = 270] 
    \draw[domain=0:1, smooth, thick, blue] plot (\x,{.08*\x*sqrt(\x)}); 
    \draw[domain=1:1.5, smooth, thick, blue] plot (\x,{.08*\x*sqrt(\x) - .1*(\x - 1)*sqrt(\x - 1)});         
\end{scope}

\begin{scope}[shift={(3.581,.5)},scale=1] 
    \draw[domain=-.707:.707, smooth, thick, red] plot ({2*exp(-\x/2)-2},{-\x*\x});
    \fill[black] (0,0) circle (\rr);  
\end{scope}

\begin{scope}[shift={(1.2,.5)},scale=1] 
    \draw[domain=0:.408, smooth, thick, red] plot (\x,{-3*\x*\x});
    \draw[domain=-.408:0, smooth, thick, blue] plot (\x,{-3*\x*\x});     
    \fill[black] (0,0) circle (\rr);     
\end{scope}

\begin{scope}[shift={(-1,1.5)},scale=1] 
    \draw[domain=-.707:0, smooth, thick, blue] plot (\x,{-3*\x*\x});  
    \draw[domain=0:.707, smooth, thick, red] plot (\x,{-3*\x*\x}); 
    \fill[black] (0,0) circle (\rr);         
\end{scope}
\end{tikzpicture} \\
\caption{A sketch of solution trajectories to \eqref{Pn}. Trajectories of particles with positive charge are coloured red; those with negative charge blue. The black dots indicate the annihilation points $(\tau, y)$.}
\label{fig:trajs}
\end{figure}
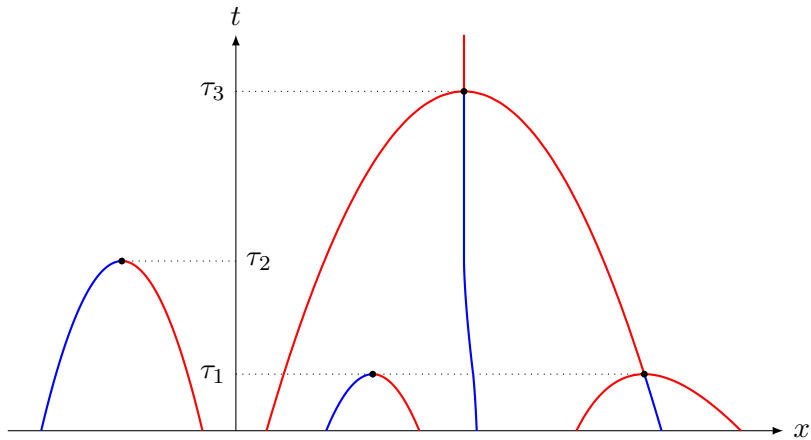

The main aim of this paper is to establish the `continuum limit' $n\to\infty$. This is a classical question in the study of interacting particle systems, with a long and distinguished history (see below). What makes the results of this paper special is the combination of opposite signs (and the ensuing annihilation) with singular interaction potentials. 

Along the way we also develop a theory for the annihilating finite-$n$ particle system \eqref{Pn}. For instance, it turns out that formulating the right solution concept for \eqref{Pn} is subtle; with the choices made in Definition~\ref{defn:Pn} we are able to prove   existence, uniqueness, stability, and various other properties  of its solutions.

\subsection{Motivation}
\label{s:intro:mot}

Our main motivation for studying the limit  $n \to \infty$  of \eqref{Pn} comes from a key problem in understanding plastic deformation of metals. Plastic deformation of metals is the macroscopic behavior of many dislocations interacting on microscopic time  and length scales. Dislocations are curve-like defects in the atomic lattice of the metal. We refer to \cite{HirthLothe82,HullBacon01} for textbook descriptions of dislocations and their relation to plasticity. Because of the immense complexity of the dynamics of groups of dislocations, it is a long-standing problem to derive plastic deformation as the micro-to-macro limit of interacting dislocations.

To address this problem, several simplifying assumptions are common in the literature:
\begin{enumerate}
  \item The assumption that the dislocation lines are straight and parallel, such that their position can be identified by points in a 2D cross section 
\cite{AlicandroDeLucaGarroniPonsiglione14,
Berdichevsky06,
CermelliLeoni06,
GarroniVanMeursPeletierScardia19,
GromaBalogh99,
GromaCsikorZaiser03,
GromaGyorgyiKocsis06,
HudsonOrtner14,
KooimanHuetterGeers15b,
MoraPeletierScardia17}.  
  \item The further assumption that these dislocation points move along the same line 
\cite{FocardiGarroni07,
GeersPeerlingsPeletierScardia13,
Hall11,
HeadLouat55,
VanMeursMuntean14}%
. This assumption is inspired by the fact that for many types of dislocations the movement is confined to a \emph{slip plane}.
  \item The assumption that the velocity of dislocations is proportional to the stress acting on them. Most of the literature on micro-to-macro limits operates under this assumption. A few exceptions are \cite{AlicandroDeLucaGarroniPonsiglione16,
HudsonVanMeursPeletier20ArXiv,
MoraPeletierScardia17}.
\end{enumerate}
Under these assumptions, the resulting dynamics of dislocations are given by \eqref{Pn}. The charges $b_i$ correspond to the orientation of the dislocations (given by their Burgers vector). Dislocations with opposite Burgers vector are such that when they collide, their lattice defects cancel out, and a locally perfect lattice emerges after a collision. This motivates the annihilation rule that we adopt here. 

Including annihilation in dynamic dislocation models in one form or another is not new. For instance, the discrete-to-continuum limits of Alicandro, De Luca, Garroni, and Ponsiglione~\cite{AlicandroDeLucaGarroniPonsiglione14,AlicandroDeLucaGarroniPonsiglione16}  allow opposite-sign defects to cancel each other, in a regime corresponding to $O(1)$ net dislocations. Forcadel, Imbert, and Monneau~\cite{ForcadelImbertMonneau09} and Van Meurs and Morandotti~\cite{VanMeursMorandotti19} both include annihilation for \emph{regularized} dislocation interactions. Another example is  \cite{MonneauPatrizi12a,
MonneauPatrizi12,
PatriziValdinoci15,
PatriziValdinoci16}, 
where the dislocation positions are described by a phase-field model (the Peierls--Nabarro model) which by construction includes annihilation, and can be considered a form of regularization. 

The work of this paper is different in a number of ways. First, we insist on taking the $n\to\infty$ limit for point particles with \emph{unregularized}, singular interactions. Secondly, the finite-$n$ system has an explicit annihilation rule, which is a consequence of the point nature of the particles and the singularity of the ODE, as we described above. 

\medskip

Our second motivation for studying annihilation is the dynamics of vortices. In \cite{Serfaty07II} and \cite{SmetsBethuelOrlandi07} it is shown that the two-dimensional dissipative Ginzburg-Landau equation converges, as the related phase-field parameter tends to zero, to an evolution of point vortices. Besides the complication of two spatial dimensions, vortices may have an integer-valued degree (charge), which further complicates the annihilation rule. In fact, a self-contained equation for the vortex dynamics including an annihilation rule has not yet been constructed. With the results of this paper we hope to help in clarifying vortex dynamics.

On the other end of the spectrum, a model for the many-vortex limit appears in superconductivity where it is called the Chapman--Rubinstein--Schatzman--E model. In \cite{AmbrosioMaininiSerfaty11} a gradient flow structure of this model is formulated, but the mathematical description is not completely satisfactory due to an implicit term regarding the vortex annihilation. 

As of yet, there appears to be no rigorous result on the many-vortex limit passage for multi-sign vortices. In this paper we aim to contribute  by considering a simpler particle system. 

\medskip

Our third motivation for studying the limit $n \to \infty$ of \eqref{Pn} is to contribute to the general understanding of many-particle limits for interacting particle systems with multiple species; see, e.g., 
\cite{DiFrancescoEspositoSchmidtchen20ArXiv,
DiFrancescoFagioli13,
DiFrancescoFagioli16}.
The particle systems considered in these papers are similar to \eqref{Pn}, but have bounded interaction forces depending on the particle types $b_i$ and $b_j$, for which no particular collision rule needs to be specified. Still, in \cite{DiFrancescoEspositoSchmidtchen20ArXiv}, the interaction force is discontinuous at zero, which results in nontrivial particle collisions that require special attention. A central difference with the current paper is that particles that collide are not removed; instead,  particles are conserved after collision. In the subsequent dynamics they remain completely `joined' to each other, and this joined couple effectively decouples the interaction between other particles to the left and to the right. Since in \cite{DiFrancescoEspositoSchmidtchen20ArXiv} the number of particles in conserved and the interaction forces remain bounded, different techniques such as Wasserstein gradient flows apply, which makes the analysis significantly different from that in this paper.


\subsection{Well-posedness of the ODE}

Our first main result is the well-posedness of \eqref{Pn}. While the classical ODE theory gives existence and uniqueness of solutions up to the first collision time $\tau$, it fails to say what happens at $\tau$. Since the interaction force is singular, it is not obvious why the limit of $x_i(t)$ as $t \nearrow \tau$ exists. And, if it exists, the possibility of multiple particle collisions makes it unclear whether the right-hand side of the ODE is defined after the annihilation rule has been applied. Besides the classical ODE theory, the gradient flow structure is of limited use too, because the energy diverges to $-\infty$ prior to collision.

Our strategy to overcome these two difficulties at a collision event is to consider the relevant quantities of the ODE. First, for existence of the limits $x_i(t)$ as $t \nearrow \tau$, we use the property that the moments 
\begin{equation} 
\label{def:Mk-intro}	
  M_k (\bx) := \frac1k \sum_{i=1}^n x_i^k \qquad
  \text{for } k = 1, \ldots, n
\end{equation}
do not blow up as $t \nearrow \tau$. Indeed, differentiating the moments along solutions yields that the singularity in the right-hand side of the ODE cancels out. This turns out to be sufficient to deduce that the limit of $x_i(t)$ as $t \nearrow \tau$ exists. 

Secondly, we show that at any collision event, the total net charge of the colliding particles has to be either $-1$, $0$, or $+1$.
Hence, after a collision, at most one charge remains active, and therefore the right-hand side of the ODE~\eqref{Pn} is again meaningful. Hence, the ODE can be restarted after collision.

These proof elements lead to our first main result.
\begin{mthm} \label{mthm:A}
(See Theorem \ref{t:Pn}).
The system of equations~\eqref{Pn} with rigorous meaning given by Definition~\ref{defn:Pn} is well posed, i.e.\ solutions exist, are unique, and depend continuously on the initial data.  In addition they are Lipschitz continuous in time with respect to a special metric~$d_\bM$ based on the moments $M_k$~\eqref{def:Mk-intro}.
\end{mthm}

\subsection{The many-particle limit $n \to \infty$}
\label{sec:intro-many-particle-limit}

Our second main result, Theorem \ref{t}, is the passage of \eqref{Pn} to the limit $n\to\infty$, and the characterization of the limit equation. The usual approach in the particle systems literature (see, e.g.,  
\cite{Duerinckx16,
GarroniVanMeursPeletierScardia19,
VanMeursMuntean14,
Schochet96}) 
is to describe the particle positions by an empirical measure. For the two-species particle system \eqref{Pn} a natural choice is
\begin{equation}
\label{def:kappa_n}	
\kappa_n := \frac1n \sum_{i=1}^n b_i \delta_{x_i}.
\end{equation}
Then, one derives from the particle system that $\kappa_n$ satisfies a weak form of a PDE, which for \eqref{Pn} turns out to be
\begin{equation} \label{kappan:PDE}
  \partial_t \kappa_n = \partial_x \big( |\kappa_n| \, (V' * \kappa_n ) \big),
\end{equation}
where $V'(x) = - \frac1x$ as in \eqref{E} and $*$ denotes the convolution over space. Finally, the task is to pass to $n \to \infty$ in this weak form in the framework of measures. The expected limiting equation is
\begin{equation} \label{kappa:PDE}
  \partial_t \kappa = \partial_x \big( |\kappa| \, (V' * \kappa ) \big),
\end{equation}
where positive values of the particle density $\kappa$ correspond to positively charged particles, and negative values represent the density of the negatively charged ones. 
In our setting, it is difficult to make this limit passage rigorous due to the combination of signed charges, singular interactions, and the annihilation rule. 
 
In this paper we follow a different approach. To date, the only rigorous interpretation of \eqref{kappa:PDE} is developed in \cite{ImbertMonneauRouy08} and \cite{BilerKarchMonneau10}. It is given in terms of a Hamilton--Jacobi equation. Before laying out the details, let us formally derive this Hamilton--Jacobi equation. The idea is to {integrate} \eqref{kappa:PDE} in space. Assuming that $\kappa$ is regular enough, we define the cumulative distribution function 
\begin{equation} \label{u:from:kappa}
  u(x) := \int_{-\infty}^x d\kappa.
\end{equation}
Since $\partial_x u = \kappa$, the derivative of $u$ describes the density of particles; in particular, the sign of $\partial_x u(x)$  determines the charge of the particles around $x$. With this definition of $u$, integrating~\eqref{kappa:PDE} in space yields
\begin{equation} 
	\label{HJ:informal:Vp}
  \partial_t u = \partial_x (V' * u ) \, |\partial_x u|.
\end{equation}
The interaction term $\partial_x (V' * u)$ has several alternative representations given by
\begin{multline} \label{Iu}
  \partial_x (V' * u )
      = \cI [u] (x)
      = -(-\Delta)^{\frac12} u (x)
      = -\cF^{-1} \left( |\xi| \widehat u (\xi) \right) \\
      = -\mathcal H [\partial_x u] (x)
      = \text{pv}\int_\R \big( u(x + z) - u(x) \big) \, \frac{dz}{z^2},
\end{multline}
where $-\cI$ is a L\'evy operator of order $1$, $(-\Delta)^{1/2}$ is the half-Laplacian, $\cF$ is the Fourier transform, $\mathcal H$ is the Hilbert transform, and $\text{pv}\int$ is the principal-value integral. 

Using the short notation $\cI[u]$ above we write the Hamilton--Jacobi equation~\eqref{HJ:informal:Vp} as
\begin{equation} \label{HJ:formal}
  \partial_t u = \cI[u] \, |\partial_x u|.
\end{equation}
This is a Hamilton--Jacobi equation with a Hamiltonian $\cI[u]$ which is non-standard for two reasons. First, it is a nonlocal function of $u$; similar nonlocal Hamilton--Jacobi equations are regularly used to describe curvature-driven flows such as mean-curvature flow (see e.g.~\cite{ChambolleMoriniPonsiglione15}). Secondly, the singular kernel in $\cI$ does not permit evaluation of $\cI$ at every bounded, uniformly continuous function, and thus a proper extension has to be constructed. We postpone the related rigorous definition of viscosity solutions to Section \ref{s:HJ}.
 
Coming back to the question of passing to the limit $n \to \infty$ in \eqref{Pn}, it is then natural to seek a Hamilton--Jacobi formulation for it such that the limit passage can be carried out in a Hamilton--Jacobi framework. In~\cite{ForcadelImbertMonneau09} the authors develop such a formulation for a version of \eqref{Pn} with \emph{regularized} interactions. They successfully pass to the limit $n \to \infty$, but the link between the Hamilton--Jacobi equation and \eqref{Pn} (with singular interactions) is only established in the simple case where all particles have the same charge. The difficulty in the signed-charge case stems from the singularity of the interactions and the previously missing definition and well-posedness of \eqref{Pn}. With Theorem \ref{mthm:A} in hand, we overcome this difficulty, and treat the case of singular interactions and annihilation. 

To construct a Hamilton--Jacobi formulation for \eqref{Pn}, we build on the ideas in \cite{ForcadelImbertMonneau09}. Analogously to \eqref{u:from:kappa}, we set
\begin{equation} \label{un}
  u_n(x) 
  = \int\limits_{(-\infty,x]} d\kappa_n 
  = \frac1n \sum_{i=1}^n b_i H(x - x_i)
  \qquad \text{for all } x \in \R
\end{equation}
as the level set function, where $H$ is the Heaviside function with $H(0) = 1$. We illustrate $u_n$ in Figure \ref{fig:level:sets}. The function $u_n$ is piecewise constant, with jumps at the particle positions $x_i$, where the direction of the jump (upward or downward) determines the charge~$b_i$. The precise value of $u_n$ at the jump points turns out to be important; we will often consider the upper and lower semi-continuous envelope of $u_n$.

\def\lwid{4pt}
\begin{figure}[ht]
\centering
\begin{tikzpicture}[scale=1.2, >= latex]
\tikzset{
    grhorline/.style={dotted, line width=\lwid, , green!90!black}
}
\def \rr {0.06} 

\draw[->] (1.5, -2.2) --++ (0,2.7) node[left] {$y$};
\draw[->] (-1, -1) node[left]{$0$} --++ (7.2,0) node[right] {$x$};
\draw (6.2,.5) node[right] {$v(\tau_1 - \delta, x)$};

\draw[thin] (-1, 0) node[left]{$\frac1n$} --++ (7.2,0);
\draw[thin] (-1, -2) node[left]{$-\frac1n$} --++ (7.2,0);

\draw[grhorline] (0,-2) -- (1,-2);

\draw[grhorline] (-1,-1) -- (0,-1);
\draw[grhorline] (1,-1) -- (2,-1);
\draw[grhorline] (2.9,-1) -- (3.1,-1);
\draw[grhorline] (4,-1) -- (4.8,-1);
\draw[grhorline] (5,-1) -- (5.2,-1);

\draw[grhorline] (2,0) -- (2.9,0);
\draw[grhorline] (3.1,0) -- (4,0);
\draw[grhorline] (4.8,0) -- (5,0);
\draw[grhorline] (5.2,0) -- (6.2,0);
\draw[green!70!black] (6.2,0) node[right] {$u_n(\tau_1 - \delta, x)$};

\foreach \x in {0, 1}
{ \draw[dotted, thick, green!70!black] (\x,-1) --++ (0,-1); }

\foreach \x in {2, 3.1, 4.8, 5.2, 2.9, 4, 5}
{ \draw[dotted, thick, green!70!black] (\x,0) --++ (0,-1); }

\foreach \x in {1, 2, 3.1, 4.8, 5.2}
{ \fill[red] (\x,-1) circle (\rr); }

\foreach \x in {0, 2.9, 4, 5}
{ \fill[blue] (\x,-1) circle (\rr); }

\draw [thick] plot [smooth] coordinates {(-1,-.5) (-.5,-.6) (0,-1) (.5,-1.2) (1,-1) (2,0) (2.45, .2) (2.9,0) (3, -.05) (3.1,0) (3.55, .2) (4,0) (4.4, -.2) (4.8,0) (4.9, .05) (5,0) (5.1, -.05) (5.2,0) (5.7, .4) (6.2, .5)};

\begin{scope}[shift={(0,-3.5)}] 
\draw[->] (1.5, -2.2) --++ (0,2.7) node[left] {$y$};
\draw[->] (-1, -1) node[left]{$0$} --++ (7.2,0) node[right] {$x$};
\draw (6.2,.5) node[right] {$v(\tau_1, x)$};

\draw[thin] (-1, 0) node[left]{$\frac1n$} --++ (7.2,0);
\draw[thin] (-1, -2) node[left]{$-\frac1n$} --++ (7.2,0);

\draw[grhorline] (0,-2) -- (1,-2);

\draw[grhorline] (-1,-1) -- (0,-1);
\draw[grhorline] (1,-1) -- (2,-1);
\draw[grhorline] (4,-1) -- (5,-1);

\draw[grhorline] (2,0) -- (4,0);
\draw[grhorline] (5,0) -- (6.2,0);
\draw[green!70!black] (6.2,0) node[right] {$u_n(\tau_1, x)$};

\foreach \x in {0, 1}
{ \draw[dotted, thick, green!70!black] (\x,-1) --++ (0,-1); }

\foreach \x in {2, 4, 5}
{ \draw[dotted, thick, green!70!black] (\x,0) --++ (0,-1); }

\foreach \x in {3}
{ \draw[dotted] (\x,0) --++ (0,-1); }

\foreach \x in {1, 2, 5}
{ \fill[red] (\x,-1) circle (\rr); }

\foreach \x in {3}
{ \draw[fill=white] (\x,-1) circle (\rr); }

\foreach \x in {0, 4}
{ \fill[blue] (\x,-1) circle (\rr); }

\draw [thick] plot [smooth] coordinates {(-1,-.5) (-.5,-.6) (0,-1) (.5,-1.2) (1,-1) (2,0) (2.5, .2) (3, 0) (3.5, .2) (4,0) (4.4, -.2) (4.8,-.025) (5,0) (5.2,.045) (5.7, .4) (6.2, .5)};         
\end{scope}

\begin{scope}[shift={(0,-7)}] 
\draw[->] (1.5, -2.2) --++ (0,2.7) node[left] {$y$};
\draw[->] (-1, -1) node[left]{$0$} --++ (7.2,0) node[right] {$x$};
\draw (6.2,.5) node[right] {$v(\tau_1 + \delta, x)$};

\draw[thin] (-1, 0) node[left]{$\frac1n$} --++ (7.2,0);
\draw[thin] (-1, -2) node[left]{$-\frac1n$} --++ (7.2,0);

\draw[grhorline] (0,-2) -- (1,-2);

\draw[grhorline] (-1,-1) -- (0,-1);
\draw[grhorline] (1,-1) -- (2,-1);
\draw[grhorline] (4,-1) -- (5,-1);

\draw[grhorline] (2,0) -- (4,0);
\draw[grhorline] (5,0) -- (6.2,0);
\draw[green!70!black] (6.2,0) node[right] {$u_n(\tau_1 - \delta, x)$};

\foreach \x in {0, 1}
{ \draw[dotted, thick, green!70!black] (\x,-1) --++ (0,-1); }

\foreach \x in {2, 4, 5}
{ \draw[dotted, thick, green!70!black] (\x,0) --++ (0,-1); }

\foreach \x in {2, 4, 5}
{ \draw[dotted] (\x,0) --++ (0,-1); }

\foreach \x in {1, 2, 5}
{ \fill[red] (\x,-1) circle (\rr); }

\foreach \x in {0, 4}
{ \fill[blue] (\x,-1) circle (\rr); }

\draw [thick] plot [smooth] coordinates {(-1,-.5) (-.5,-.6) (0,-1) (.5,-1.2) (1,-1) (2,0) (2.5, .2) (3, .05) (3.5, .2) (4,0) (4.4, -.2) (4.8,-.05) (5,0) (5.2,.08) (5.7, .4) (6.2, .5)};         
\end{scope}
\end{tikzpicture} \\
\caption{Possible graphs of the level set functions $u_n$ and $v$ for the particle trajectories in Figure \ref{fig:trajs} sliced at three time points; slightly before $\tau_1$, at $\tau_1$, and slightly after $\tau_1$. The change of sign of $v$ at the intersections with the level sets at $\Z/n$ determines the charge of the particles. If there is no change in sign, no particle is associated to it (illustrated by the white circle in the second graph). }
\label{fig:level:sets}
\end{figure}
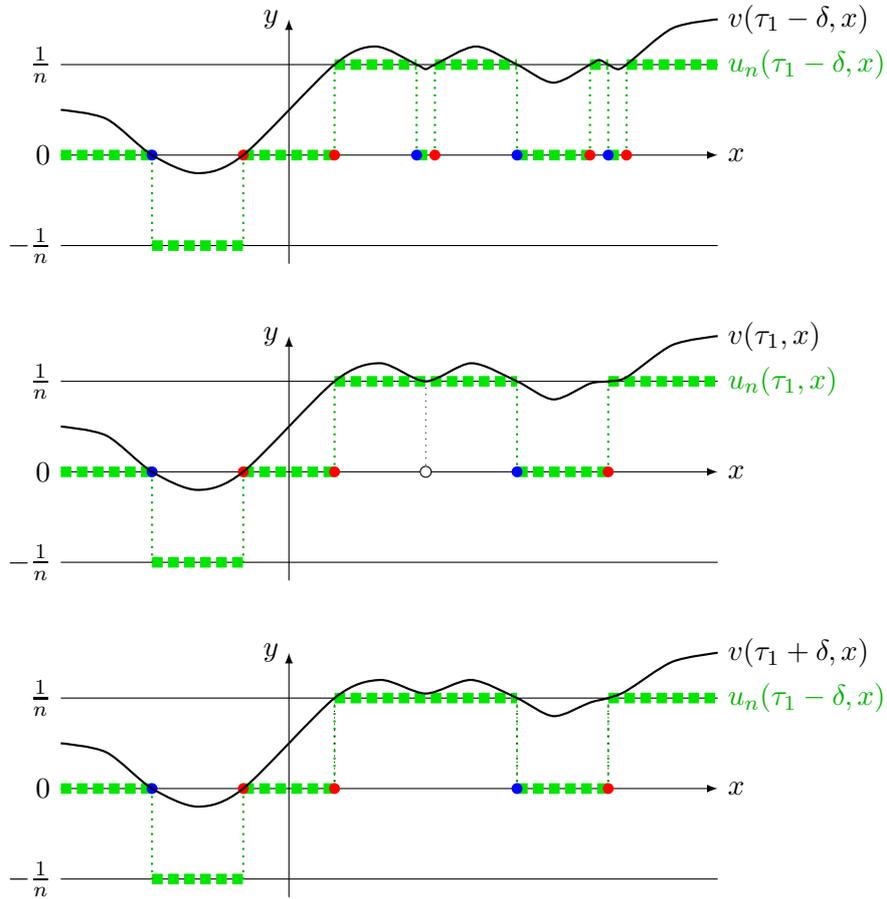

However, to derive formally the Hamilton--Jacobi equation, it is more instructive to replace~$u_n$ by a regular enough level set function $v$. Figure \ref{fig:level:sets} illustrates a possible choice of~$v$. Given $v$, the particle positions $x_i$ can be recovered from the level sets of~$v$ at~$\frac1n \Z$, i.e.\ by solving for~$y$ in $v(y) \in \frac1n \Z$. The charge $b_i$ is then given by 
\begin{equation} \label{bi:ito:dxv}
  b_i = \sign( \partial_x v(x_i) ),
\end{equation}
which is the analogue of the continuum Hamilton--Jacobi equation in which $\partial_x u$ determines the charge of the particle density. Given $u_n$, a sufficient condition for a related continuous function $v$ is
\begin{equation} \label{v:from:un}
 0 < v(x) - u_n(x) < \frac1n \qquad \text{for all } x \in \R \setminus \{x_i : i = 1,\ldots,n\}. 
\end{equation}
Because of the jumps in $u_n$ and the continuity of $v$, this implies that 
\begin{equation} \label{vi:ito:uni}
  v(x_i) = u_n^*(x_i), 
\end{equation}
where $u_n^*$ is the upper semi-continuous envelope of $u_n$. By \eqref{v:from:un} we can recover $u_n$ from $v$ by
\begin{equation} \label{un:ito:v}
  u_n^*(x) = \frac1n \lfloor n v(x) \rfloor,
\end{equation}
where $\lfloor\cdot\rfloor$ denotes the floor function, that is, the largest integer less than or equal to the argument.

To derive formally the Hamilton--Jacobi equation from the solution $(\bx, \bb)$ to \eqref{Pn}, suppose that there exists a regular level set function $v : [0,\infty) \times \R \to \R$ such that $v(t, \cdot)$ satisfies \eqref{v:from:un} at each $t \geq 0$. Then, by \eqref{vi:ito:uni}, $v(t, x_i(t))$ is constant in $t$ for each $i$. Hence, $0 = \frac d{dt} v(t, x_i(t))$, which we rewrite as
\begin{equation} \label{HJn:formal:1}
  \partial_t v (t, x_i(t)) 
  = - \frac{d x_i}{dt}(t) \, \partial_x v (t, x_i(t))
  = - \frac{1}{n} \sum_{ j \neq i }^n \frac{b_i b_j}{x_i - x_j} \, \partial_x v (t, x_i(t)).
\end{equation}
To rewrite the force in terms of $v$, we use for $x_i > x_j$ that
\begin{equation} \label{force:ito:uz}
  \frac{1}{x_i - x_j} 
  = \int_{x_i - x_j}^\infty \frac{dz}{z^2} 
  = \int_{\R} H(z - x_i + x_j) \, \frac{dz}{z^2}
  \qquad \text{and} \qquad
  \text{pv}\int_{\R} \Big( H(z) - \frac12 \Big) \, \frac{dz}{z^2} = 0
\end{equation}
to obtain (see \eqref{pf:Me:RHS} for details)
\begin{align*}
  - \frac{1}{n} \sum_{ j \neq i }^n \frac{b_i b_j}{x_i - x_j}
  &= b_i \, \text{pv}\int_\R \Big( u_n^*(x_i + z) - u_n^*(x_i) + \frac1{2n} \Big) \, \frac{dz}{z^2} \\
  &= b_i \, \text{pv}\int_\R E_{1/n} \big( v(x_i + z) - v(x_i) \big) \, \frac{dz}{z^2},
\end{align*}
where in the second equality we have used \eqref{un:ito:v} with
\begin{equation} \label{E1n}
    E(\alpha) := \lfloor \alpha \rfloor + \frac12,
    \qquad E_{1/n}(\alpha) = \frac1n E ( n \alpha).
\end{equation} 
Inserting this in \eqref{HJn:formal:1} and recalling \eqref{bi:ito:dxv}, we obtain
\begin{equation} \label{HJn:formaller}
  \partial_t v = \cI_n [v] \, |\partial_x v|, \qquad
  \cI_n [v] (x) := \text{pv}\int_\R E_{1/n} \big( v(x + z) - v(x) \big) \, \frac{dz}{z^2}.
\end{equation}
This is the formal shape for the Hamilton--Jacobi equation which describes solutions to \eqref{Pn}.  

\medskip

We take a moment to discuss several features of \eqref{HJn:formaller}. First, the expression for $\cI_n$ resembles the last expression of $\cI$ in \eqref{Iu}. The only difference is the appearance of $E_{1/n}$, which is a staircase approximation of the identity (see Figure \ref{fig:En}). The offset $\frac1{2n}$ is carefully chosen to cancel out the singularity of the kernel $1/z^2$ at $0$; it originates in the subtraction of $\frac12$ in the integrand in \eqref{force:ito:uz}. The role of $E_{1/n}$ is to project the information of $v$ down to the behavior of the level sets of $v$ corresponding to the discrete levels $v(x) =  k/n$ for different values of~$k\in\Z$; this corresponds to the discrete nature of \eqref{Pn}. Because of this nonlocal feature of $\cI_n$ which even includes interactions between different level sets, \eqref{HJn:formaller} is a non-standard Hamilton--Jacobi equation. To further illustrate its behavior, we compute in Example \ref{ex:HJe} an explicit, continuous solution $v$ to \eqref{HJn:formaller}, and show how each choice of $a \in [0,1)$ for the level sets at $\frac1n (\Z + a)$ corresponds to a different solution to \eqref{Pn}. 

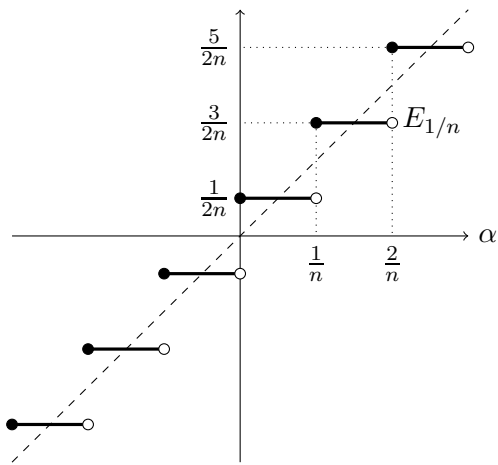
\begin{figure}[ht]
\centering
\begin{tikzpicture}[scale=1]
    \draw[->] (-3,0) -- (3,0) node[right]{$\alpha$};
    \draw[->] (0, -3) -- (0,3);
    
    \draw[dotted] (1,1.5) --++ (0,-1.5) node[below]{$\frac1n$}; 
    \draw[dotted] (2,2.5) --++ (0,-2.5) node[below]{$\frac2n$};
    \draw[dotted] (1,1.5) --++ (-1,0) node[left]{$\frac3{2n}$}; 
    \draw[dotted] (2,2.5) --++ (-2,0) node[left]{$\frac5{2n}$};
    \draw (0,.5) node[left]{$\frac1{2n}$};  
    
    \draw[dashed] (-3, -3) --++ (6,6);
    \foreach \k in {-3,-2,-1,0,1,2}{    
      \draw[very thick] (\k, \k +0.5) --++ (1,0);
      \draw[fill=black] (\k, \k + 0.5) circle[radius=2pt];
      \draw[fill=white] (\k + 1, \k + 0.5) circle[radius=2pt];
    }
    \draw (2,1.5) node[right]{$E_{1/n}$};
\end{tikzpicture} \\
\caption{Plot of the staircase approximation $E_{1/n}$ of the identity.}
\label{fig:En}
\end{figure}

Secondly, the formulation of \eqref{Pn} in terms of the Hamilton--Jacobi equation \eqref{HJn:formaller} has several mathematical advantages:
\begin{enumerate}
   \item No annihilation rule needs to be specified. For example, a two-particle annihilation happens when a local maximum of $v$ crosses a level set downwards, or when a local minimum crosses it upwards (see Figure \ref{fig:level:sets}).
   \item The function $v$ need not develop singularities around annihilation points (see again Figure \ref{fig:level:sets}).
 \end{enumerate} 
 Note that the Hamilton--Jacobi formulation exposes a monotonicity property of the trajectories of~\eqref{Pn}: for $a_1 \neq a_2$, the level sets of $v$ at $\frac1n (\Z + a_1)$ do not cross with any of the level sets of $v$ at $\frac1n (\Z + a_2)$. This property follows after establishing a comparison principle. This monotonicity of trajectories is difficult to obtain from either \eqref{Pn} or its measure formulation \eqref{kappan:PDE}, and is not used in the part of the literature on many-particle limits which relies on a measure-theoretic framework. 

Thirdly, even for smooth functions $\phi$ the operator $\cI_n$ is not defined at local maxima or local minima of $\phi$. However, at these points, $\partial_x \phi$ vanishes, and thus it may be possible to extend the right-hand side in \eqref{HJn:formaller} at these points. We show that such an extension is possible when developing a notion of viscosity solutions to \eqref{HJn:formaller}. In this definition, we build on ideas from \cite{ImbertMonneauRouy08} for the limiting Hamilton--Jacobi equation \eqref{HJ:formal}. It is here that we depart from the framework in \cite{ForcadelImbertMonneau09}; in~\cite{ForcadelImbertMonneau09} this issue was solved by regularizing the singular kernel $1/z^2$ around $0$, which implies modifying the equation~\eqref{Pn}, while we insist on studying the original equation.

Fourthly, our Definition~\ref{def:HJe:VS} of viscosity solutions of \eqref{HJn:formaller} does not only allow for solutions that are continuous, but also for \emph{discontinuous} solutions;  we show in Lemma \ref{l:Pn:to:HJe} and Proposition \ref{p:Pn:to:HJe:un} that the step function $u_n(t,x)$ as constructed from the solution to \eqref{Pn} indeed is the unique viscosity solution to \eqref{HJn:formaller} with initial datum $u_n(0,x)$. Because of the discontinuity, usual proofs of uniqueness do not apply (see e.g.~\cite[Th.~2.4]{BilerKarchMonneau10}); Proposition~\ref{p:Pn:to:HJe:un} uses the special structure of~\eqref{Pn} to deal with this discontinuity. This shows that each solution of~\eqref{Pn} generates a corresponding \emph{unique} solution of the Hamilton--Jacobi equation \eqref{HJn:formaller}. This is interesting, because \eqref{Pn} has a clear physical interpretation with possible extension to higher dimensions, whereas \eqref{HJn:formaller} has an advantageous mathematical structure.

\medskip

We now return to the aim of this paper to pass to the limit $n \to \infty$. For this limit passage we follow the usual approach in viscosity theory; this is our second main result.  

\begin{mthm}
\label{mthm:B}
	(See Theorem~\ref{t}.) Assume that the initial data satisfy $u_n^\circ\to u^\circ$ uniformly, and that $u^\circ$ is bounded and uniformly continuous. 
Then $u_n \to u$ locally uniformly in time-space as $n \to \infty$, the function $u$ is continuous, and $u$ satisfies the limiting Hamilton--Jacobi equation given formally by \eqref{HJ:formal}.
\end{mthm}

\medskip

 Turning to  the limiting Hamilton--Jacobi equation \eqref{HJ:formal}, we construct a different---but equivalent---notion of viscosity solutions than that in \cite{ImbertMonneauRouy08,BilerKarchMonneau10}. The reason for this is that Theorem \ref{mthm:B} is easier to prove if the notions of viscosity solutions of \eqref{HJ:formal} and \eqref{HJn:formaller} are similar.  More precisely, in contrast to the viscosity-solution approach in \cite{ImbertMonneauRouy08}, we restrict the class of test functions at $\partial_x u = 0$ to functions of the specific form $cx^4 + g(t)$. This idea of reducing the class of test functions is inspired by the approach of Ishii and Souganidis \cite{IshiiSouganidis95} and \cite{OhnumaSato97}; see also \cite{ChambolleMoriniPonsiglione15} and the discussion of $\mathcal F$-solutions in \cite{Giga06}. This choice interacts well with the observation that  annihilating particles  meet with quadratic rate (see Theorem~\ref{t:Pn}(\ref{t:Pn:C12}--\ref{t:Pn:LB:col}) and Figure~\ref{fig:trajs}). Even with this restricted class of test functions, the standard comparison principle holds (Theorem~\ref{t:CPe} and Theorem~\ref{t:CP:HJ}) and yields uniqueness even for discontinuous initial data  (Proposition \ref{p:Pn:to:HJe:un}). 
 
\subsection{Measure-theoretic version of Theorem~\ref{mthm:B}}

As we mentioned above, a common approach in many-particle limits is to consider the evolution equation~\eqref{kappan:PDE} for the empirical measure $\kappa_n$, and pass to the limit in a weak formulation of that equation. The typical type of convergence that one obtains in this way is narrow convergence at each time $t$, i.e.
\[
\int_\R \varphi(x) \kappa_n(t,dx) \longrightarrow\int_\R \varphi(x)\kappa(t,dx) 
\qquad\text{for each }\varphi\in C_b(\R).
\]
For comparison with this body of literature we now restate Theorem~\ref{mthm:B} in terms of $\kappa_n$ and $\kappa$,  using measure terminology.

Theorem~\ref{mthm:B} requires the initial datum $u_n^\circ$ to converge uniformly to a continuous limit $u^\circ$. 
If $\kappa_n^\circ$ were non-negative, then the pair $(\kappa_n^\circ,u_n^\circ)$ could be interpreted as a probability distribution and its cumulative distribution function; in this case $\kappa_n^\circ$ converges narrowly if and only if $u_n^\circ$ converges at every~$x$ at which the limit $u^\circ$ is continuous. When the limit $u^\circ$ is continuous, as is the case here, the convergence strengthens to uniform convergence.

However, the measures $\kappa_n$ have both signs, and then (locally) uniform convergence to a continuous limit is significantly stronger than narrow convergence alone (a counterexample is $\kappa_n = \delta_{1/n} - \delta_0$; see Section~8.1 in~\cite{Bogachev07.II} for a discussion). This explains the appearance of the third condition on $\kappa_n$ in the following Lemma.
\begin{mlem}
\label{ml:conv-circ}
(See Lemma~\ref{l:conv-equiv-R}.)
Let $u_n^\circ(x)=\kappa_n^\circ((-\infty,x])$ as in~\eqref{un}, and assume that the sequence $\kappa_n^\circ$ is bounded in total variation and tight. Then the following are equivalent:
\begin{enumerate} 
\item $u_n^	\circ$ converges uniformly to $u^\circ$, and $u^\circ$ is continuous;
\item $u_n^	\circ$ converges locally uniformly to $u^\circ$, and $u^\circ$ is continuous;
\item \begin{enumerate}
	\item $\kappa_n^\circ $ converges narrowly to $\kappa^\circ$,
	\item \label{mlem:i:no-atoms} $\kappa^\circ$ has no atoms, and
	\item there exist a sequence $s_n\xrightarrow{n\to\infty}0$ and a modulus of continuity $\omega$ such that 
	\begin{equation}
	\label{cond:AEC-kappa-intro}
		\text{for all }-\infty< x\leq y< \infty, \qquad 
		  |\kappa^\circ_n((x,y])|\leq s_n + \omega(|x-y|).
	\end{equation}
	\end{enumerate}
\end{enumerate}
The limits $u^\circ$ and $\kappa^\circ$ are connected by $u^\circ(x) = \kappa^\circ((-\infty,x])$.
\end{mlem}

\begin{mcor}
\label{mcor}
(See Corollaries~\ref{c:conv-kappa} and~\ref{c:kappa-limit-ct}.)
Assume that $\kappa_n^\circ$ converges to  $\kappa^\circ$ in the sense of Lemma~\ref{ml:conv-circ}.
Then  for any sequence $t_n\to t$ in $[0,T]$, $\kappa_n(t_n)$ converges to $\kappa(t)$ in the sense of Lemma~\ref{ml:conv-circ}. The sequence $(s_n)_n$ and the modulus of continuity $\omega$ can be chosen to be independent of the sequence $t_n\to t$.

In addition the map $t\mapsto \kappa(t)$ is narrowly continuous. 
\end{mcor}

\noindent
The connection between Theorem~\ref{mthm:B} and Corollary~\ref{mcor} is discussed in Section~\ref{s:kappa}.
\subsection{Discussion}
\label{ss:Discussion}

To summarize the above, our main two results are the combination of Definition~\ref{defn:Pn} and  Theorem \ref{t:Pn} on well-posedness of the particle systems described by \eqref{Pn}, and Theorem \ref{t} on the convergence of \eqref{Pn} as $n \to \infty$ with \eqref{kappa:PDE} as the PDE for the signed limiting particle distribution. We conclude by discussing several features of these results.

\paragraph{Properties of the limiting equation} In \cite{BilerKarchMonneau10} regularity of viscosity solutions to \eqref{HJ:formal} is proven. In particular, if $u^\circ \in \Lip (\R)$, then $\| u(t, \cdot) \|_\infty$ and $\| \partial_x u(t, \cdot) \|_\infty$ are non-increasing in time.
Hence, if the signed particle distribution $\kappa^\circ$ is an absolutely continuous measure with bounded density, then $\kappa(t, \cdot)$ is also absolutely continuous with density in $L^\infty(\R)$ for each $t \geq 0$.

\paragraph{H\"older-continuous trajectories} The properties of the solution $(\bx, \bb)$ to \eqref{Pn} listed in Theorem \ref{t:Pn} suggest that $t \mapsto \bx (t)$ is in $C^{1/2}([0,T])$. Since our proof methods do not rule out oscillatory behavior of the trajectories prior to annihilation, we were only able to prove that $t \mapsto \bx (t)$ is continuous, and $C^\infty$ away from collision times.

\paragraph{Other interaction potentials} In \eqref{Pn} we chose to take as the particle interaction potential the function $V(x) = -\log |x|$. This choice is most relevant to the applications mentioned in Section \ref{s:intro:mot}. In the literature on particle systems other potentials such as $V(x) = |x|^{-s}$ with $0 < s < 1$ or smooth perturbations thereof also appear 
\cite{Duerinckx16,
GeersPeerlingsPeletierScardia13,
Hauray09,
PetracheSerfaty14}. 
Due to lack of clear applications for annihilation and to avoid clutter, we have not investigated whether our method works for such potentials. 

\paragraph{Two dimensions} A future goal is to define and prove well-posedness of \eqref{Pn} in two dimensions, and then also to pass to the limit $n \to \infty$. Our current proof methods strongly rely on the ordering of the particles. Yet, already in 1D our proof method for the well-posedness of \eqref{Pn} contains unconventional ideas, which may inspire a new approach to treat the higher dimensional case.

\paragraph{Organization of the paper} In Section \ref{s:Pn} we state and prove our first main result on the well-posedness of \eqref{Pn}. In Section \ref{s:HJ} we give a precise meaning to the Hamilton--Jacobi equations \eqref{HJn:formaller} and \eqref{HJ:formal}, prove that they satisfy a comparison principle, and establish the convergence of \eqref{HJn:formaller} to \eqref{HJ:formal}. In Section \ref{s:Pn:to:HJ} we apply this convergence to pass to $n \to \infty$ in \eqref{Pn}. In Section \ref{s:kappa} we reformulate the convergence result in a measure theoretic framework.

\section{Well-posedness and properties of $(P_n)$}
\label{s:Pn}

To give a rigorous meaning to the particle system in \eqref{Pn} with $n \geq 2$, we start with defining a state space $\cZ_n$ for the pair $(\bx, \bb)$. It will be convenient to have a unique description of a set of particles by numbering them from left to right; in the case of charged particles the dynamics preserves such numbering, since same-sign neighbors repel each other and opposite-sign neighbors are removed upon collision. Neutral particles, however, have no reason to preserve the left-to-right numbering of the charged particles, since neutral and charged particles do not interact. 

These observations lead to a definition of the state space $\cZ_n$ that imposes ordering for charged particles only:
\begin{equation*}
  \cZ_n := \{ (\bx, \bb) \in \R^n \times \{-1,0,+1\}^n : \text{if } i>j \text{ and } b_i b_j \neq 0, \text{ then } x_i > x_j \},
\end{equation*}
Since we are interested in the charged particles, we call two particles $j < i$ \emph{neighbors}  if they are charged and any particle in between them is neutral, i.e.\ 
\[ b_i(t) b_j(t) \neq 0
\quad \text{and} \quad 
\{k: b_k(t) \neq 0,\ j < k < i\} = \emptyset. \]

(As an alternative to $\cZ_n$, one could remove the neutral particles from the state, and consider the union of sets of varying dimension
\[
\bigcup_{k=1}^n \big\{ (\bx, \bb) \in \R^k \times \{-1,+1\}^k : x_{i+1}>x_i \big\}
\]
as the state space. This state space can naturally be embedded into $\cZ_n$ by relabeling the particle indices. We prefer to keep the number of particles the same, and `remove' particles by setting their charge to zero.)

The solution concept to \eqref{Pn} is as follows:

\begin{defn}[Solution to $(P_n)$] \label{defn:Pn}
Let $n \geq 2$, $\bb^\circ \in \{-1,0,1\}^n$ and $(\bx^\circ, \bb^\circ) \in \cZ_n$. A map $(\bx, \bb) : [0,T] \to \cZ_n$ is a solution of $(P_n)$ if there exists a finite subset $S \subset (0,T]$ such that
\begin{enumerate}[(i)]
  \item (Regularity) For each $i\in\{1,n\}$, $x_i \in C([0,T]) \cap C^1([0,T] \setminus S)$, and $b_1, \ldots, b_n : [0,T] \to \{-1, 0, 1\}$ are right-continuous; 
  \item (Initial condition) $(\bx(0), \bb(0)) = (\bx^\circ, \bb^\circ)$;
  \item (Annihilation rule) \label{defn:Pn:ann:rule} Each $b_i$ jumps at most once. If $b_i$ jumps at $t \in [0,T]$, then $t \in S$, $|b_i(t-)| = 1$ and $b_i(t) = 0$. Moreover, for all $(\tau, y) \in S \times \R$,
  \begin{equation} \label{annih:rule}
    \sum_{i: x_i(\tau) = y} \llbracket b_i \rrbracket (\tau) = 0,
  \end{equation}
  where the bracket $\llbracket f \rrbracket(t)$ is the difference between the right and left limits of $f$ at $t$;
  \item (ODE of $\bx$) On $(0,T) \setminus S$, $\bx$ satisfies the ODE in \eqref{Pn}.
\end{enumerate}
\end{defn}

Definition \ref{defn:Pn} calls for some terminology. For a function $f$ of one variable, we set $f(t-) := \lim_{s \nearrow t} f(s)$ as the left limit.
We call a point $(\tau, y) \in S \times \R$ an  \emph{annihilation point} if the sum in \eqref{annih:rule} contains at least one non-zero summand. We call the time $\tau$ of an annihilation point a \emph{collision time}. The set of all collision times $\{ \tau_1, \ldots, \tau_K \}$ is finite, where $K \leq n^+ \wedge n^-$ and $n^\pm$ are the numbers of positively/negatively charged particles at time $0$. From Theorem \ref{t:Pn} it turns out that the minimal choice for $S$ is $\{ \tau_1, \ldots, \tau_K \}$.

In Definition \ref{defn:Pn}, annihilation is encoded by the combination of the annihilation rule in \eqref{defn:Pn:ann:rule} and the requirement that $(\bx(t), \bb(t)) \in \cZ_n$. Indeed, Definition \ref{defn:Pn}\eqref{defn:Pn:ann:rule} limits the choice of jump points for $b_i$, while the separation of particles implied by $(\bx(t), \bb(t)) \in \cZ_n$ requires particles to annihilate upon collision. 

\begin{rem}
\label{r:unique-modulo-renumbering}
As we shall see in Theorem~\ref{t:Pn} below, solutions according to Definition~\ref{defn:Pn} are unique, but only  up to relabeling. One can recognize the possibility of relabeling as follows: if three particles (say numbered $i=1,2,3$, with charges $+,-,+$) collide at some point $(\tau,y)$, then according to~\eqref{annih:rule} one positive particle should continue, while the two other particles should become neutral. Therefore either $i=1$ could remain positive, with $i=2,3$ becoming neutral, or $i=3$ could remain positive with $i=1,2$ becoming neutral. Both lead to the same evolution of \emph{points} and their charges, and therefore the same physical interpretation, but the numbers attached to the points are different. 
\end{rem}

\medskip
  
To state the main result of this section, Theorem \ref{t:Pn}, on the well-posedness of \eqref{Pn} and properties of the solutions, we introduce several objects. First, given $(\bx, \bb) \in \cZ_n$, we set $d^+$ as the smallest distance between any two neighboring particles with positive charge. Analogously, we define $d^-$ for the negatively charged particles. More precisely, we set $m := \sum_{i=1}^n |b_i|$ as the number of charged particles, and take a permutation $\sigma \in S_n$ such that $(x_{\sigma(1)}, \dots, x_{\sigma(m)})$ is the ordered list of all charged particles. Then, 
\begin{equation} \label{dpm}
  d^\pm := \inf \big\{ x_{\sigma(i+1)} - x_{\sigma(i)} : i \text{ is such that } b_{\sigma(i+1)} = b_{\sigma(i)} = \pm 1 \big\} \in (0,\infty].
\end{equation} 

Secondly, we recall from the introduction the (scaled and signed) moments of $\bx$ given by
\begin{equation*} 
  M_k (\bx) := \frac1k \sum_{i=1}^n x_i^k \qquad
  \text{for } k = 1, \ldots, n.
\end{equation*}
Then, using the map
\begin{equation*}
  \bM : \R^n / S_n \to \R^n,
  \qquad \bM \big( \bx) := (M_1(\bx), \ldots, M_n(\bx) \big)^T,
\end{equation*}
we define the moment-distance
\begin{equation*}
  d_\bM (\bx, \by) := \| \bM(\bx) - \bM(\by) \|_2.
\end{equation*}

\begin{lem}
\label{le:dM-metric}
The distance $d_\bM$ is a metric on $\R^n / S_n$ and $d_\bM$-bounded sets are relatively compact. Moreover, if $\bx_m \to \bx$ in $(\R^n/S_n, d_\bM)$ then $\bx_m \to \bx$ on $\R^n /S_n$ with the Euclidean norm.
\end{lem}

\begin{proof}
Positivity and the triangle inequality are immediate. The fact that $d_\bM (\bx, \by) = 0$ implies $ \bx = \by$ follows from Newton's identities. To see this, we write 
$$\prod_{i=1}^n (z - x_i) = \sum_{k=0}^n (-1)^k e_k z^{n - k}$$ 
using the symmetric polynomials 
\begin{align*}
  e_0 &= 1 \\
  e_1 &= x_1 + \cdots + x_n \\
  e_2 &= x_1 x_2 + \cdots + x_{n-1} x_n \\
  &\vdots \\
  e_n &= x_1 x_2 \cdots x_n.
\end{align*}
By Newton's identities, the symmetric polynomials satisfy
\[
me_m  = \sum_{k=1}^m (-1)^{k-1} e_{m-k} k M_k \quad \text{for } m = 1, \ldots, n.
\]
In particular, $d_\bM(\bx, \by) = 0$ implies $\bM(\bx) = \bM(\by)$ and hence $\prod_{i=1}^n (z - x_i) = \prod_{i=1}^n (z - y_i)$ for all $z\in \R$. Therefore  $\bx = \by$ up to a reordering of the indices, i.e.\ $\bx=\by$ as elements of $\R^n/S_n$. It follows that $d_\bM$ is a metric on $\R^n / S_n$.

If $d_\bM(\bx_m, \bx) \to 0$ then by Newton's identities above we deduce that $e_k(\bx_m) \to e_k(\bx)$ for all $1 \leq k \leq n$ and hence the polynomials $\prod_{i=1}^n(z - \bx_{m,i})$ converge locally uniformly to $\prod_{i=1}^n(z - \bx_i)$. We deduce that $\bx_m \to \bx$ in the Euclidean norm as elements of $\R^n/S_n$.

Finally, take a sequence $\bx_m$ bounded in $d_\bM$, that is, for some $R >0$ we have $d_\bM(\bx_m, 0) < R$ for all $m \geq 1$. In particular, $\| \bx_m \|_2^2 = 2 M_2(\bx_m) < 2R$, and therefore $\bx_m$ converges along a subsequence in $\R^n$ to some $\bx$ in the Euclidean norm. Then, along the same subsequence, $M_k(\bx_m) \to M_k(\bx)$ for all $k \geq 1$, and thus $d_\bM(\bx_m, \bx) \to 0$.
\end{proof}

\begin{thm}[Properties of $(P_n)$] \label{t:Pn}
Let $n \geq 2$, $T > 0$ and $(\bx^\circ, \bb^\circ) \in \cZ_n$. 
Then, $(P_n)$ has a solution $(\bx, \bb)$ with initial datum $(\bx^\circ, \bb^\circ)$, according to Definition \ref{defn:Pn}, that is unique modulo relabeling (see Remark~\ref{r:unique-modulo-renumbering}). 
Moreover, setting $S$ as the minimal set from Definition~\ref{defn:Pn} for the solution $(\bx, \bb)$, the following properties hold: 
\begin{enumerate}[(i)]
  \item \emph{($d_\bM$-Lipschitz regularity)}. \label{t:Pn:Lip}
   There exists a constant $C_n > 0$ depending only on $n$ and $\bM(\bx^\circ)$ such that
\begin{equation*}
  d_\bM (\bx(t), \bx(s)) \leq C_n|t-s|
  \qquad \text{for all } 0 \leq s < t \leq T;
\end{equation*}
  \item \emph{(Lower bound on minimal distance between neighbors of equal sign)}. \label{t:Pn:LB:d}
   \[ 
     d^\pm(t) \geq \sqrt{ \tfrac8{n^2 - 1} t +  d^\pm(0)^2} \qquad  \text{for all } t \in [0,T];
     \]
   \item \emph{(Lower bound on distance between any two neighbors)}.  \label{t:Pn:LB:pm}
   Let $i < j$ be neighboring particles at time $t_0 \geq 0$. Then
\begin{align}
\label{opposite-sign-lower-bound} 
x_j(t) - x_i(t) \geq \sqrt{c_0^2 - 8 \tfrac{\log n + 1}n (t - t_0)},
\qquad c_0 := \min(d^+, d^-, x_j - x_i) (t_0)
\end{align}
for all $t \geq t_0$ for which the square root exists;
   \item \emph{(Upper bound at collision)}. \label{t:Pn:C12}
   For any $\tau \in S$ and any $i$, there exists a $C \geq 0$ such that 
   $$ |x_i(t) - x_i(\tau)| \leq C \sqrt{\tau - t} 
      \quad \text{for all } t \in [0, \tau]; $$
   \item \emph{(Lower bound at collision)}. \label{t:Pn:LB:col}
   For each annihilation point $(\tau, y)$, there exists a $c > 0$ and indices $i,j$ such that $x_i(\tau) = x_j(\tau) = y$, $b_i(\tau-) \neq 0$, $b_j(\tau-) \neq 0$ and 
   \begin{align*}
     x_i(t) - y &> c \sqrt{\tau - t} \\
     x_j(t) - y &< -c \sqrt{\tau - t}
   \end{align*}
   for all $t < \tau$ large enough;
  \item \emph{(Stability with respect to $\bx^\circ$)}. \label{t:Pn:stab}
   Let $(\bx_m^\circ, \bb^\circ) \in \cZ_n$ be such that $\bx_m^\circ \to \bx^\circ$ as $m \to \infty$. Let $(\bx_m, \bb_m)$ be the solution of $(P_n)$ with initial data $(\bx_m^\circ, \bb^\circ)$. Then, $\bx_m \to \bx$ in $C([0,T]; (\R^n / S_n, d_\bM))$ and $\bb_m \to \bb$ locally uniformly on $[0, T] \setminus S$ as $m \to \infty$.
\end{enumerate}
\end{thm}

Remark that in Property \eqref{t:Pn:LB:d}, equality is reached when $n$ is odd, $b_i^\circ = 1$ and $x_i^\circ = i$ for all~$i$. Moreover, as a direct consequence of Theorem \ref{t:Pn}\eqref{t:Pn:LB:d} we have the following result.
\begin{cor}[Multiple-particle collisions]\label{c:Pn}
Let $(\tau, y) \in S \times \R$ be an annihilation point, and let $I$  be the corresponding indices:
  \begin{equation} \label{I:coll:parts}
    I := \{ i : x_i(\tau) = y, \, b_i(\tau-) \neq 0 \}.
  \end{equation}
Then prior to annihilation, the particles with index in $I$ have charges of alternating sign. In particular, 
\begin{equation*}
\Big| \sum_{i \in I} b_i(\tau-) \Big| \leq 1.
\end{equation*}
\end{cor}

\begin{proof}[Proof of Theorem \ref{t:Pn}]
\emph{Uniqueness}. Let $(\bx, \bb)$ and $(\hat \bx, \hat \bb)$ be two solutions with minimal sets of annihilation times $S = \{\tau_1, \ldots, \tau_K\}$ and $\hat S = \{\hat \tau_1, \ldots, \hat \tau_{\hat K} \}$ respectively. By standard ODE theory and the minimality of $S$ and $\hat S$, we obtain $\hat \tau_1 = \tau_1$ and $\hat \bx |_{[0, \tau_1)} = \bx |_{[0, \tau_1)}$. By continuity, $\hat \bx(\tau_1) = \bx(\tau_1)$. Hence, any annihilation point $(\tau_1, y)$ of $\bx$ is also an annihilation point of $\hat \bx$. Let $(\tau_1, y)$ be such an annihilation point, and let $I$ be the related index set of the colliding particles (see \eqref{I:coll:parts}). Definition \ref{defn:Pn}(iii) implies that 
\begin{equation*}
  \sum_{i \in I} \hat b_i(\tau)
  = \sum_{i \in I} \hat b_i(\tau-)
  = \sum_{i \in I} b_i(\tau-)
  = \sum_{i \in I} b_i(\tau).
\end{equation*}
In addition, since $(\bx, \bb), \, (\hat \bx, \hat \bb) \in \cZ_n$,
\begin{equation*}
  \sum_{i \in I} | b_i(\tau) | \leq 1
  \quad \text{and} \quad
  \sum_{i \in I} | \hat b_i(\tau) | \leq 1.
\end{equation*}
Hence, $\{ \hat b_i(\tau) \}_{i \in I}$ equals $\{ b_i(\tau) \}_{i \in I}$ up to a possible relabeling. Thus, the ODEs for $\bx$ and $\hat \bx$ are restarted with the same right-hand side. Iterating the argument above over all annihilation times in $S$, we obtain $\hat S = S$ and $(\hat \bx, \hat \bb) = (\bx, \bb)$, modulo relabeling of the particles.
\medskip

\emph{Existence and Properties \eqref{t:Pn:Lip},\eqref{t:Pn:LB:d}}. Standard ODE theory provides the existence of $\bx$ up to the first time $\tau_1$ at which either $\bx(\tau_1-)$ does not exist, or $(\bx(\tau_1-), \bb^\circ) \notin \cZ_n$. It is sufficient to show that $\bx(\tau_1-)$ exists, and that under Definition \ref{defn:Pn}\eqref{defn:Pn:ann:rule} $\bb(\tau_1)$ can be chosen such that $(\bx, \bb)(\tau_1) \in \cZ_n$. Indeed, if these two conditions are met, then $(\bx, \bb)(\tau_1)$ is an admissible initial condition for $(P_n)$, and further annihilation times $\tau_k$ are found and treated by induction. 

To prove these two conditions, we set $\tau := \tau_1$, and note that it is sufficient to prove Properties \eqref{t:Pn:Lip} and \eqref{t:Pn:LB:d} both with $T$ replaced by $\tau$. Indeed, Property \eqref{t:Pn:Lip} implies that $\bx(\tau_1-) \in \R^n$ exists. Then, Property \eqref{t:Pn:LB:d} implies Corollary \ref{c:Pn}, which gives enough information to construct $\bb(\tau)$ such that $(\bx, \bb)(\tau_1) \in \cZ_n$.
\smallskip

\emph{Property \eqref{t:Pn:Lip} with $T$ replaced by $\tau$}. For any integer $k \geq 0$ we compute
\begin{equation} \label{pf:ddt:Mk1}
  \frac d{dt} M_{k+1} (\bx) 
  = \sum_{i=1}^n x_i^k \frac{d x_i}{dt}
  = \sum_{i=1}^n \frac{x_i^k}n \sum_{j \neq i}^n \frac{b_i b_j}{x_i - x_j}
  = \frac1n \sum_{i=1}^n \sum_{j = 1}^{i-1} b_i b_j \frac{x_i^k - x_j^k}{x_i - x_j}
\end{equation}
on $(0, \tau)$. For $k = 0$, the right-hand side vanishes, and thus $M_1(\bx(t)) = M_1(\bx^\circ)$ is constant in $t$. For $k = 1$, we observe that
\begin{equation} \label{pf:ddt:Mk2}
  \frac d{dt} M_2 (\bx) 
  = \frac1n\sum_{i=1}^n \sum_{j = 1}^{i-1} b_i b_j
  = \frac1{2n} \Big( \sum_{i=1}^n \sum_{j = 1}^n b_i b_j - \sum_{i=1}^n b_i^2 \Big),
\end{equation} 
which is constant and bounded from above by $(n-1)/2$. Hence, 
\begin{equation}
\label{ineq:M2-bound}	
M_2(\bx(t)) \leq M_2(\bx^\circ) + \frac{n-1}2 t.
\end{equation}
For $k \geq 2$,
\begin{equation*}
  \bigg| \frac d{dt} M_{k+1} (\bx) \bigg|
  = \frac1n \bigg| \sum_{i=1}^n \sum_{j = 1}^{i-1} b_i b_j \sum_{\ell = 0}^{k-1} x_i^\ell x_j^{k - 1 - \ell} \bigg|
  \leq \frac1n \sum_{\ell = 0}^{k-1} \Big( \sum_{i=1}^n |x_i|^\ell \Big) \Big( \sum_{j=1}^n |x_j|^{k - 1 - \ell} \Big).
\end{equation*}
To bound the right-hand side, note that for even $\ell$,
\begin{equation*}
  \sum_{i=1}^n |x_i|^\ell = \ell M_\ell(\bx),
\end{equation*}
and for odd $\ell$
\begin{equation*}
  \Big( \sum_{i=1}^n |x_i|^{\tfrac{\ell+1}2} |x_i|^{\tfrac{\ell-1}2} \Big)^2
  \leq \Big( \sum_{i=1}^n |x_i|^{\ell + 1 } \Big) \Big( \sum_{i=1}^n |x_i|^{\ell - 1 } \Big)
  = \left\{ \begin{aligned}
    &{2n M_2(\bx)}
    &&\text{if } \ell = 1  \\
    &{ (\ell^2-1) M_{\ell+1}(\bx) M_{\ell-1}(\bx) }
    &&\text{otherwise. }
  \end{aligned} \right.
\end{equation*}
Hence, $| \frac d{dt} M_{k+1} (\bx) |$ is bounded on $(0,\tau)$ in terms of $f_k(M_1 (\bx), \ldots, M_k (\bx) )$ for some $f_k \in C(\R^k)$.
The result follows from induction over $k$ by integrating $\frac d{dt} M_{k+1} (\bx)$ from $0$ to $\tau$. 
\smallskip

\emph{Property \eqref{t:Pn:LB:d} with $T$ replaced by $\tau:= \tau_1$}. We prove Property \eqref{t:Pn:LB:d} for $d^+$; the proof for $d^-$ is analogous. For convenience, we assume that there are no neutral particles. Since $\bx$ is a solution to $(P_n)$, we obtain from the definition of $d^+$ in \eqref{dpm} that $d^+$ as a function on $[0, \tau)$ is positive, locally Lipschitz continuous and hence differentiable almost everywhere. 

Let $t \in (0, \tau)$ be a point of differentiability of $d^+$, and let $x_i(t)$ and $x_{i+1}(t)$ be particles for which the minimum in \eqref{dpm} is attained. Then, $(x_{i+1} - x_i)(t) = d^+(t)$, and at time $t$,
\begin{align} \notag
  \frac d{dt} d^+
  &= \frac{ dx_{i+1} }{dt} - \frac{ dx_i}{dt} \\\notag
  &= \frac2n \frac1{x_{i+1} - x_i} + \frac{1}{n} \sum_{ j \notin \{i, i + 1 \} } b_j \Big( \frac1{x_{i+1} - x_j} - \frac1{x_i - x_j} \Big) \\\label{pf:ddt:dp}
  &= \frac2{n d^+} + \frac{1}{n} \sum_{ j \notin \{i, i + 1 \} } (-b_j) \underbrace{ \frac{d^+}{(x_{i+1} - x_j)(x_i - x_j)} }_{=: \gamma_j}.
\end{align}

Next we bound the sum in \eqref{pf:ddt:dp} from below. For convenience, we focus on the part corresponding to $j > i+1$. The idea is to remove certain positive terms from the summation such that the remaining indices in the sum correspond to negative contributions of positively charged particles which are all separated by a distance no smaller than $d^+$.  

We remove indices (particles) in two consecutive steps. In the first step, we apply the following rule for all $j = i+2, \ldots, n-1$. If $b_j = -1$ and $b_{j+1} = 1$, then we remove both $j$ and $j+1$ from the summation. Note that the joint contribution of $j$ and $j+1$ to the sum is $-b_j \gamma_j + b_{j+1} \gamma_{j+1} = \gamma_j - \gamma_{j+1}$, which is positive since $\gamma_k$ is decreasing in $k$. The second step is simply to remove all remaining negatively changed particles $x_j$. Since $\gamma_k > 0$, each such particle provides a positive contribution $\gamma_j > 0$ to the summation.

After applying this rule for removing indices from the summation (but keeping the original labeling of the particles), all remaining indices $j$ satisfy $b_j = b_{j-1} = 1$, and thus all the corresponding particles are separated by a distance no smaller than $d^+$. This yields the following lower bound (assuming $n$ is even for convenience):
\begin{multline*}
  \sum_{ j \notin \{i, i + 1 \} } (-b_j) \frac{d^+}{(x_{i+1} - x_j)(x_i - x_j)}
  \geq - \sum_{ j \notin \{i, i + 1 \} } \frac{d^+}{(i+1 - j)d^+ \, (i - j)d^+} \\
  \geq - \frac2{d^+} \sum_{ k=1 }^{n/2-1} \frac{1}{k(k+1)}
  = - \frac2{d^+} \sum_{ k=1 }^{n/2-1} \Big( \frac1k - \frac1{k+1} \Big)
  = - 2\Big(1 - \frac2n\Big) \frac1{d^+}.
\end{multline*}
Inserting this lower bound in \eqref{pf:ddt:dp}, we obtain
\begin{equation} \label{pf:ddt:dp:2}
  \frac d{dt} d^+
  \geq \frac2{n d^+} - \frac2n \Big(1 - \frac2n\Big) \frac1{d^+}
  = \frac{4}{n^2} \frac1{d^+}.
\end{equation}
When $n$ is odd, a similar computation yields
\begin{equation*}
  \frac d{dt} d^+
  \geq \frac{4}{n^2 - 1} \frac1{d^+}
\end{equation*}
and thus
\begin{equation} \label{pf:ddt:dp:3}
  \frac d{dt} (d^+)^2
  \geq \frac{8}{n^2 - 1}.
\end{equation}

Since \eqref{pf:ddt:dp:3} holds for any $n \geq 2$ and at any point of differentiability $t$, and since $d^+$ is locally Lipschitz, we obtain by integrating from $0$ to any $t \in (0,\tau)$ that
\begin{equation*}
  d^+(t)^2 
  \geq d^+(0)^2 + \frac{8}{n^2 - 1} t.
\end{equation*}
This proves Property \eqref{t:Pn:LB:d} for $t < \tau = \tau_1$. 
\smallskip

A similar argument establishes Property~\eqref{t:Pn:LB:d} for $t\in [\tau_1,\tau_2)$, provided we show that $d^+(\tau_1) \geq d^+(\tau_1-)$, which we do now. Let the indices $i < k$ be such that $x_{k}(\tau_1) - x_i(\tau_1)$ is a minimizer for the set in \eqref{dpm} at $\tau_1$, i.e.\ 
\begin{itemize}
    \item $d^+(\tau_1) = x_k(\tau_1)-x_i(\tau_1)$;
    \item $b_k(\tau_1) = b_i(\tau_1) = 1$;
	\item $x_i(\tau_1)$ and $x_k(\tau_1)$ are neighbors.
\end{itemize}
Since $x_i(\tau_1)$ and $x_k(\tau_1)$ are neighbors, $\sum_{j=i+1}^{k-1} b_j(\tau_1) = 0$, and \eqref{annih:rule} implies that $\sum_{j=i+1}^{k-1} b_j(\tau_1-) = 0$. Therefore $\sum_{j=i}^{k} b_j(\tau_1-) = 2$, and since all $b_j(\tau_1-)$ in this sum are either $+1$ or $-1$, there exists $j \in \{i, \ldots, k-1\}$ such that $b_j(\tau_1-) = b_{j+1}(\tau_1-) = 1$. Hence,
\begin{equation*}
  d^+(\tau_1-) 
  \leq x_{j+1}(\tau_1-) - x_j(\tau_1-)
  = x_{j+1}(\tau_1) - x_j(\tau_1)
  \leq x_k(\tau_1)-x_i(\tau_1)
  = d^+(\tau_1).
\end{equation*} 

This completes both the proof for the existence of the solution $(\bx, \bb)$ to $(P_n)$ up to time~$T$, and the proof of Properties \eqref{t:Pn:Lip},\eqref{t:Pn:LB:d} up to time $T$.
\medskip

\emph{Property \eqref{t:Pn:LB:pm}}. 
For convenience, we assume that at $t_0$ all particles are charged (this implies $j = i+1$) and that $b_i(t_0) = 1$. Setting $d := x_{i+1} - x_i$, we write
\begin{align*}
\frac{d x_i}{dt} = -\frac{b_{i+1}}{nd} + \frac1n \sum_{j=1}^{i-1} \frac{b_j}{x_i - x_j} - \frac1n \sum_{j=i+2}^{n} \frac{b_j}{x_j - x_i}.
\end{align*}
To estimate the right-hand side from above, we use the technique in the proof of  Property \eqref{t:Pn:LB:d} to remove from the first sum a certain number of particles such that all remaining particles have positive charge and are separated by a distance $d^+$. Using the same technique also for the second sum, we obtain
\begin{align*}
\frac{d x_i}{dt} \leq \frac 1{nd} + \frac1n \sum_{j=1}^{i-1} \frac{1}{(i-j) d^+} + \frac1n \sum_{j=i+2}^{n} \frac{1}{(j - i) d^-}
\leq \frac{\log n + 1}n \Big( \frac1d + \frac1{d^=} \Big),
\end{align*}
where $d^= := \min(d^+, d^-)$. Since this upper bound is positive, it includes the scenario in which $x_i$ annihilates with $x_{i-1}$. A symmetric argument yields
\begin{align*}
\frac{d x_{i+1}}{dt} \geq -\frac{\log n + 1}n \Big( \frac1d + \frac1{d^=} \Big). 
\end{align*}
Therefore
\begin{align*}
\frac{d}{dt} d \geq - 2\frac{\log n + 1}n \Big( \frac1d + \frac1{d^=} \Big),
\end{align*}
As $d^=$ is nondecreasing due to Property \eqref{t:Pn:LB:d}, by comparison with $\frac{d e}{dt} = -4/e$ with initial datum $e(t_0) = \min(d^=(t_0), d(t_0))$, we deduce \eqref{opposite-sign-lower-bound}. 
\medskip

\emph{Property \eqref{t:Pn:C12}}. Note that it is sufficient to prove Property \eqref{t:Pn:C12} only for all $t$ from the last annihilation time $\tau_0$ prior to~$\tau$ (we set $\tau_0 = 0$ if $\tau$ is the first annihilation time) up to~$\tau$. Indeed, we can otherwise iterate backwards in time over the finitely many annihilation times, and use the continuity of $x_i$ and the bound in Property \eqref{t:Pn:C12} at each annihilation time to capture the resulting curve in a new parabola.

Next, we prove Property \eqref{t:Pn:C12} for all $t \in (\tau_0, \tau)$ and any $i$. We note that on this interval, $\bb$ is constant, and $\bx$ satisfies the ODE in \eqref{Pn}. If $b_i = 0$, then Property \eqref{t:Pn:C12} is satisfied with $C = 0$. If $b_i \neq 0$ and $x_i$ does not collide at $t = \tau$, then the right-hand side in \eqref{Pn} is bounded at $\tau$. By the continuity of $\bx$, it is also bounded in a neighborhood around $\tau$. Hence, $x_i$ is Lipschitz continuous in this neighborhood, which is sufficient to construct a $C$ for which Property \eqref{t:Pn:C12} is satisfied.

The delicate case is when $x_i$ collides with other particles at $\tau$. To avoid relabeling, we assume that there are no neutral particles up to time $\tau$. Let $I$ as in \eqref{I:coll:parts} be the index set of all particles that collide with $x_i$ at $\tau$, including $i$ itself. We use the translation invariance to assume that $x_i(\tau) = 0$. 
We can split the right-hand side of the ODE in \eqref{Pn} as
\begin{equation*}
 \frac{dx_i}{dt} 
  = \frac1n \sum_{ j \in I \setminus \{i\} } \frac{b_i b_j}{x_i - x_j} + F_i,
\end{equation*}
where
\begin{equation} \label{pf:Fi}
  F_i := \frac1n \sum_{ j \in I^c } \frac{b_i b_j}{x_i - x_j}.
\end{equation}
By the definition of $I$, the continuity of $\bx$ and $(\bx, \bb) \in \cZ_n$, we have
\begin{equation*}
  c := \inf_{\tau_0 < t < \tau} \min_{j \in I} \min_{k \in I^c} |x_j - x_k|(t) > 0.
\end{equation*}
Using $c$, we obtain the bound
\begin{equation*}
  \max_{j \in I} |F_j(t)| \leq 1/c,
\end{equation*}
which is independent of $t$.

Next we inspect the second moment of the colliding particles, which we define by
\begin{equation*}
  M (t) := \frac12 \sum_{ j \in I } x_j(t)^2.
\end{equation*}
By definition $M(\tau) = 0$.
A computation similar to \eqref{pf:ddt:Mk1} and \eqref{pf:ddt:Mk2} yields
\begin{equation*}
  \frac{dM}{dt}
  = \underbrace{ \frac1{2n} \sum_{ j \in I } \sum_{ \substack{ k \in I \\ k \neq j } } b_j b_k }_{=: -B} + \underbrace{ \sum_{ j \in I } x_j F_j }_{=: R},
\end{equation*}
where $B \in \R$ is a constant and $R = R(t)$ is a remainder term. A similar computation as in \eqref{pf:ddt:Mk2} shows that 
\[ B 
   = - \frac1{2n} \Big( \Big( \sum_{j=1}^n b_j \Big)^2 - \sum_{j=1}^n b_j^2 \Big)
   \geq \frac{|I| - 1}{2n} 
   > 0. \] 
To bound $R(t)$, we use that $M(t) \to 0$ as $t \nearrow \tau$ to get
\begin{equation*}
  |R(t)|^2
  \leq \Big(\sum_{ j \in I } x_j(t)^2 \Big) \Big(\sum_{ j \in I } F_j(t)^2 \Big)
  \leq 2 M(t) \frac{|I| }{c^2}  
  \to 0\qquad \text{as } t \nearrow \tau.
\end{equation*} 
Therefore for all $t < \tau$ sufficiently large we have
\begin{align}
\label{dMdt-bound}
-\frac{3B}2 \leq \frac{dM}{dt}(t) \leq -\frac B2.
\end{align}
From this and $M(\tau) = 0$ we deduce that
\begin{equation*}
  x_i(t)^2 \leq 2 M(t) \leq 3 B (\tau - t),
\end{equation*}
which completes the proof of Property \eqref{t:Pn:C12}.
\medskip

\emph{Property \eqref{t:Pn:LB:col}}. We translate coordinates such that $y=0$, and consider the computation and notation in the proof of Property \eqref{t:Pn:C12} for the colliding particles. In addition, we may assume that $\tau$ is the first collision time. Set $k := \min I$ and $\ell := \max I$; we will construct $c' > 0$ such that for all $t < \tau$ large enough
\begin{equation*}
  c' \sqrt{\tau - t} 
  < \min \big\{ x_\ell(t), - x_k(t) \big\}.
\end{equation*}

From \eqref{dMdt-bound} we have for $t < \tau$ large enough
\[
\sum_{ i \in I } x_i(t)^2 = 2M(t) \geq B(\tau - t).
\]
We conclude that there exists $c > 0$ such that for all $t < \tau$ large enough
\begin{equation} \label{pf:xkxl}
  x_k(t) \leq - c \sqrt{ \tau-t }
  \quad \text{or} \quad 
  x_\ell(t) \geq c \sqrt{ \tau-t }.
\end{equation}

To show that both cases have to hold, we inspect the first moment $m(t) := \sum_{ i \in I } x_i(t)$. Similar to \eqref{pf:ddt:Mk1} we compute
\begin{equation*}
  \frac{dm}{dt} 
  = \sum_{ i \in I } \frac1n \sum_{ j \in I^c } \frac{b_i b_j}{x_i - x_j}
  = \sum_{ i \in I } F_i,
\end{equation*}
where $F_i \in C([0, \tau])$. Hence, $m \in C^1([0,\tau])$. Since $m(\tau) = 0$, there exists $C > 0$ such that 
$$ |m(t)| \leq C (\tau - t). $$ 

Next we show that each of the two inequalities in \eqref{pf:xkxl} implies the other. Suppose that $x_\ell(t) \geq c \sqrt{ \tau-t }$ holds for some $t < \tau$ large enough (to be specified later). Then,
\begin{equation*}
  C (\tau - t) 
  \geq m(t)
  \geq x_\ell(t) + (|I|-1) x_k(t)
  \geq c \sqrt{ \tau-t } + (|I|-1) x_k(t).
\end{equation*}
Rearranging terms and changing constants, 
\begin{equation*}
  x_k(t) \leq c \sqrt{ \tau-t } \big( C \sqrt{ \tau-t } - 1 \big).
\end{equation*} 
Hence, for $t$ large enough, the upper bound on $x_k$ in \eqref{pf:xkxl} holds. Similarly, it follows that the upper bound on $x_k$ implies the lower bound on $x_\ell$. 
\medskip

\emph{Property \eqref{t:Pn:stab}}. Using Property \eqref{t:Pn:Lip} and Lemma~\ref{le:dM-metric}, Ascoli--Arzel\`a gives a subsequence $m$ (not relabeled) and an $\tilde \bx : [0,T] \to \R^n$ for which $\bx_{m} \to \tilde \bx$ in $C([0,T]; (\R^n / S_n, d_\bM))$ as $m \to \infty$, and in particular $\bx_m(t) \to \tilde\bx (t)$ pointwise in the Euclidean norm as elements of $\R^n/S_n$. By uniqueness of solutions to \eqref{Pn}, it is then sufficient to show that $(\tilde \bx, \bb)$ is a solution to \eqref{Pn}.

We start by proving that $(\tilde \bx, \bb)$ is a solution to \eqref{Pn} up to the first collision time $\tau$ of the limit $\tilde \bx$. Let $\delta \in (0, \tau)$. Passing to the limit $m \to \infty$ in the weak version of the ODE (testing against $\varphi \in C^1([0,\tau - \delta])$), we obtain that $\tilde \bx$ satisfies the weak version of the ODE on $[0,\tau - \delta]$. Since $\delta$ is arbitrary, $(\tilde \bx, \bb)$ satisfies the ODE on $(0, \tau)$. Moreover, by the continuity of $\tilde \bx$,
\begin{equation} \label{pf:bbe:upto:tau}
  \tilde \bx |_{[0,\tau]} = \bx |_{[0,\tau]}.
\end{equation}

Next we claim that for all $\delta$ small enough there exists $m_0 > 0$ such that for all $m \geq m_0$
\begin{equation} \label{pf:bbe:claim}
  \bb_m (\tau + \delta) = \bb (\tau + \delta) \quad \text{(modulo relabeling)}.
\end{equation}
From this claim, the argument above applies again to pass to the limit $m \to \infty$ in the weak form of the ODE on any compact subinterval of $(\tau, \tau_2)$, where $\tau_2$ is the second collision time of $\bx$. This yields that $(\tilde \bx, \bb)$ satisfies the ODE on $(\tau, \tau_2)$, and by the continuity of $\tilde \bx$ we get
$$\tilde \bx |_{[\tau, \tau_2]} = \bx |_{[\tau, \tau_2]}.$$
Property \eqref{t:Pn:stab} follows by iterating over the annihilation times of $\bx$.

It is left to prove the claim \eqref{pf:bbe:claim}. The idea of the argument is to localize around any annihilation point at $\tau$. With this aim, we fix any $i \in \{1,\ldots,n\}$, and take $I$ as the index set of particles $\tilde x_j$ which collide with $\tilde x_i$ at $\tau$, including $i$ itself. We allow for $I = \{i\}$, in which case $\tilde x_i$ does not collide with any other particle at $\tau$. From \eqref{pf:bbe:upto:tau} we infer that $(\tilde \bx(\tau), \bb(\tau)) \in \cZ_n$, and thus any two particles $\tilde x_j$ and $\tilde x_k$ with $j \in I$ and $k \notin I$ at time $\tau$ are separated by a distance of at least
\[ \rho 
:= \min_{\substack{ j \in I \\ k \notin I }} \big| x_j(\tau) - x_k(\tau) \big|
  > 0. \]
Then, since $\tilde \bx$ is continuous, a similar separation distance remains in effect over the time interval $[\tau - \delta, \tau + \delta]$ for all $\delta$ small enough, i.e.,
\begin{equation} \label{pf:tbx:through:cyl}
  \min_{ t \in [\tau - \delta, \tau + \delta] } 
  \min_{\substack{ j \in I \\ k \notin I }} 
  \big| \tilde x_j(t) - \tilde x_k(t) \big| > \frac23 \rho.
\end{equation}
We illustrate the geometric interpretation of $\rho$ and $\delta$ in Figure \ref{fig:pf:vi}. For later use, we will take $\delta$ small enough so that
\begin{equation} \label{pf:bbe:apriori:2}
  0 < \tau-\delta  < \tau+\delta < \tau_2   
  \quad \text{and} \quad  
  \frac6\rho - \frac1{n \delta} \leq -1.
\end{equation}

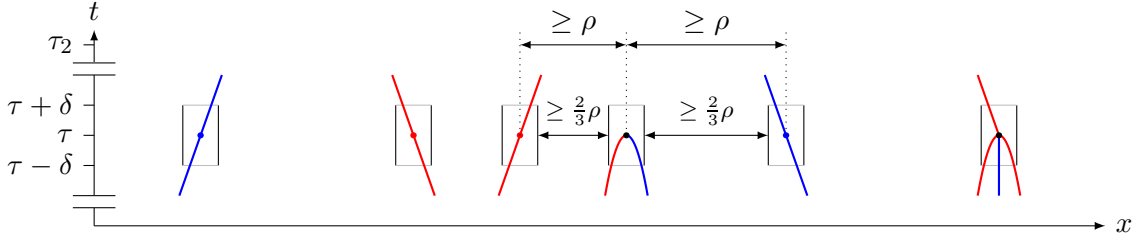
\begin{figure}[ht]
\centering
\begin{tikzpicture}[xscale=1.4, yscale=.4, >= latex]
\def \sqtwo {1.414}
\def \rr {0.03}

\draw[->] (0,0) -- (0,6.5) node[above] {$t$};
\draw[->] (0,0) -- (9.5,0) node[right] {$x$};

\foreach \point in {(0,.8),(0,5.2)}{
  \begin{scope}[shift=\point] 
    \fill[white] (-.2,-.2) rectangle (.2,.2);
    \draw (-.2,-.2) --++ (.4,0);
    \draw (-.2,.2) --++ (.4,0);
  \end{scope}
}

\draw (0,2) --++ (-.1,0) node[left]{$\tau - \delta$};
\draw (0,3) --++ (-.1,0) node[left]{$\tau$};
\draw (0,4) --++ (-.1,0) node[left]{$\tau + \delta$};
\draw (0,6) --++ (-.1,0) node[left]{$\tau_2$}; 

\foreach \point in {(1,3),(3,3),(4,3),(5,3),(6.5,3),(8.5,3)}{
  \begin{scope}[shift=\point] 
    \draw (-.167,-1) --++ (0,2);
    \draw (.167,-1) --++ (0,2);
    \draw[lightgray] (-.167,-1) -- (.167,-1);
    \draw[lightgray] (-.167,1) -- (.167,1);
  \end{scope}
}

\begin{scope}[shift={(1,3)},scale=1] 
    \draw[thick, blue] (-.2,-2) -- (.2,2);          
    \fill[blue] (0,0) ellipse (0.03 and 0.1);
\end{scope}

\begin{scope}[shift={(3,3)},scale=1] 
    \draw[thick, red] (.2,-2) -- (-.2,2);          
    \fill[red] (0,0) ellipse (0.03 and 0.1);
\end{scope}

\begin{scope}[shift={(4,3)},scale=1] 
    \draw[thick, red] (-.2,-2) -- (.2,2);          
    \fill[red] (0,0) ellipse (0.03 and 0.1);
\end{scope}

\begin{scope}[shift={(6.5,3)},scale=1] 
    \draw[thick, blue] (.2,-2) -- (-.2,2);          
    \fill[blue] (0,0) ellipse (0.03 and 0.1);
\end{scope}

\begin{scope}[shift={(5,3)},scale=1] 
    \draw[domain=-.2:0, smooth, thick, red] plot (\x,{-2*\x*\x*25});           
    \draw[domain=0:.2, smooth, thick, blue] plot (\x,{-2*\x*\x*25});           
    \fill[black] (0,0) ellipse (0.03 and 0.1);
\end{scope}

\begin{scope}[shift={(8.5,3)},scale=1] 
    \draw[thick, red] (0,0) -- (-.2,2);
    \draw[thick, blue] (0,0) --++ (0,-2);
    \draw[domain=-.2:.2, smooth, thick, red] plot (\x,{-2*\x*\x*25});                     
    \fill[black] (0,0) ellipse (0.03 and 0.1);
\end{scope}

\foreach \point in {(4,3),(5,3),(6.5,3)}{
  \begin{scope}[shift=\point] 
    \draw[dotted] (0,0) --++ (0,3.5);
  \end{scope}
}

\draw[<->] (4.167,3) -- (4.833,3) node[midway, above]{\footnotesize $\geq \frac23 \rho$};
\draw[<->] (5.167,3) -- (6.333,3) node[midway, above]{\footnotesize $\geq \frac23 \rho$};
\draw[<->] (4,6) --++ (1,0) node[midway, above]{$\geq \rho$};
\draw[<->] (5,6) --++ (1.5,0) node[midway, above]{$\geq \rho$};
\end{tikzpicture} \\
\caption{Sketch of the geometric interpretation of $\rho$ and $\delta$ for the localization of the trajectories of $x_i$ in Figure \ref{fig:trajs} zoomed in around $\tau$.}
\label{fig:pf:vi}
\end{figure}

Next we construct $m_0$. First, a separation condition similar to  \eqref{pf:tbx:through:cyl} remains in effect for the particles $\bx_m$ when $m$ is large enough. Indeed, by the pointwise convergence of $\bx_m$ to $\tilde \bx$ in $\R^n / S_n$ as $m \to \infty$, it follows from \eqref{pf:tbx:through:cyl} by the triangle inequality that for all $t \in [\tau - \delta, \tau + \delta]$ there exists $m_0 > 0$ such that for all $m \geq m_0$
\begin{equation*}
  \min_{\substack{ j \in I \\ k \notin I }} 
  \big| x_{m,j}(t) - x_{m,k}(t) \big| > \frac\rho2.
\end{equation*}
Then, by Property \eqref{t:Pn:LB:pm}, this separation condition is uniform in time, i.e., there exists $m_0 > 0$ such that for all $m \geq m_0$
\begin{equation} \label{pf:tbx:through:cyl:e}
  \min_{ t \in [\tau - \delta, \tau + \delta] } \min_{\substack{ j \in I \\ k \notin I }} 
  \big| x_{m,j}(t) - x_{m,k}(t) \big| > \frac\rho3.
\end{equation}
Second, by \eqref{pf:bbe:upto:tau} and the fact that the trajectories of $\bx$ over $[0, \tau - \delta]$ do not intersect, it follows again from the pointwise convergence of $\bx_m$ to $\tilde \bx$ and Property \eqref{t:Pn:LB:pm} that for $m$ large enough no particles $x_{m,j}$ with $j \in I$ collide before time $\tau - \delta$, i.e., by taking $m_0$ larger if necessary,
\begin{equation} \label{pf:tbx:b:equal}
  b_{m,j}(\tau - \delta) = b_j (\tau - \delta)
\end{equation}
for all $j \in I$ and all $m \geq m_0$. Third, for later use, we take $m_0$ larger if necessary to ensure that for all $m \geq m_0$
\begin{equation} \label{pf:tbx:xe:close:at:tau}
  \max_{j,k \in I} \big| x_{m,j}(\tau) - x_{m,k}(\tau) \big| < \delta.
\end{equation}

Take $m \geq m_0$ arbitrary. To prove \eqref{pf:bbe:claim} it is enough to show that at most one particle $x_{m,j}$ with $j \in I$ is charged at time $\tau + \delta$, i.e.,
\begin{equation} \label{pf:bbe:claim:reduced}
  \sum_{j \in I} \big| b_{m,j}(\tau + \delta) \big| \leq 1.
\end{equation}
Indeed, from Corollary \ref{c:Pn} it follows that also at most one particle $x_j$ with $j \in I$ is charged at time $\tau + \delta$. Then, by the conservation of charge at collisions (see Definition \ref{defn:Pn}\eqref{defn:Pn:ann:rule}) and \eqref{pf:tbx:b:equal} we obtain
\[
  \sum_{j \in I} b_{m,j}(\tau + \delta)
  = \sum_{j \in I} b_{m,j}(\tau - \delta)
  = \sum_{j \in I} b_j(\tau - \delta)
  = \sum_{j \in I} b_j(\tau + \delta),
\]
and the claim in \eqref{pf:bbe:claim} follows.

To prove \eqref{pf:bbe:claim:reduced}, we set
\[
  I_m (t) := \{ j \in I : b_{m,j}(t) \neq 0 \}
\]
as the index set of charged particles $x_{m,j}$ at time $t$ with $j \in I$, we define 
\[ 
  D_m(t) := \left\{ \begin{aligned}
    &\max_{j,k \in I_m(t)} x_{m,j}(t) - x_{m,k}(t)
    &&\text{if } |I_m(t)| \geq 2  \\
    &0
    &&\text{otherwise}
  \end{aligned} \right.
\]
as the maximal distance between any two particles with indices in $I_m(t)$,
and we prove that 
\begin{equation} \label{pf:bbe:Ie}
  D_m(t) \leq (\tau + \delta) - t
  \qquad \text{for all } t \in [\tau, \tau + \delta],
\end{equation} 
so that in particular $D_m(\tau + \delta) = 0$, which implies \eqref{pf:bbe:claim:reduced}.

By \eqref{pf:tbx:xe:close:at:tau}, the claim \eqref{pf:bbe:Ie} holds at $t = \tau$. To prove \eqref{pf:bbe:Ie} beyond $\tau$, we may assume that $|I_m (\tau)| \geq 2$. For convenience, we relabel the particles so that $I_m (\tau) = \{k, k+1, \ldots, \ell\}$. We treat the case where $b_{m,k}(\tau) = 1$; the other case can be treated analogously. Let $S_m$ be the set of annihilation times of $\{ x_{m,j} : j \in I_m (\tau) \}$, and set $\tau_m = \min S_m$. It is sufficient to show that
\begin{subequations} \label{pf:bbe:De}
\begin{align} \label{pf:bbe:De:jump}
  \llbracket D_m (t) \rrbracket &\leq 0
  &&\text{for all } t \in S_m, \\\label{pf:bbe:De:der}
  \frac{d D_m}{dt} (t) &\leq -1
  &&\text{for all } t \in (\tau, \tau + \delta) \setminus S_m \text{ with } D_m(t) > 0.
\end{align}
\end{subequations}

Since \eqref{pf:bbe:De:jump} is obvious, we focus on proving \eqref{pf:bbe:De:der}. We first consider the case where $|I|$ is even. In this case, $|I_m (\tau)|$ is also even, and we find on $(\tau, \tau_m)$ that
\begin{equation} \label{dxe1dt}
  \frac{dx_{m,k}}{dt} 
  = \frac1n \sum_{ j = k+1 }^\ell \frac{b_{m,k} b_{m,j}}{x_{m,k} - x_{m,j}} + F_{m,k},
\end{equation} 
where $F_{m,k}$ is as in \eqref{pf:Fi}. We bound $|F_{m,k}|$  as in the proof of Property~\eqref{t:Pn:C12}. Since any particle $x_{m,j}(t)$ with index $j < k$ or $j > \ell$ satisfies either $b_{m,j}(t) = 0$ or $j \notin I$, 
we obtain from \eqref{pf:tbx:through:cyl:e} that
\begin{equation*}
  |F_{m,k}|
  = \frac1n \bigg| \sum_{j=1}^{k-1} \frac{b_k b_j}{x_{m,k} - x_{m,j}} + \sum_{j=\ell+1}^{n} \frac{b_k b_j}{x_{m,k} - x_{m,j}} \bigg|
  \leq \frac1n \Big( (k-1) \frac3\rho + (n - \ell - 1) \frac3\rho \Big)
  \leq \frac3\rho.
\end{equation*}
For the first term on the right-hand side of \eqref{dxe1dt}, we infer from Corollary~\ref{c:Pn} that $b_{m,j}(\tau) = (-1)^{j-k}$ for all $j = k, \ldots, \ell$. Hence, on $(\tau, \tau_m)$,
\begin{multline} \label{pf:xe1:ita:force}
  \frac1n \sum_{ j = k+1 }^\ell \frac{b_{m,k} b_{m,j}}{x_{m,k} - x_{m,j}} \\
  = \frac1n \sum_{ j = 1 }^{(\ell - k - 1)/2} \Big( \frac{1}{x_{m,k + 2j - 1} - x_{m,k}} - \frac{1}{x_{m,k + 2j} - x_{m,k}} \Big) + \frac1n \frac{1}{x_{m,\ell} - x_{m,k}}
  \geq \frac1n \frac{1}{D_m},
\end{multline}
where in the last inequality we have used the ordering of the particles $\{ x_{m,j} : j \in I_m (\tau) \}$. Collecting these findings in \eqref{dxe1dt}, we obtain
\begin{equation*} 
  \frac{dx_{m,k}}{dt} 
  \geq \frac1n \frac{1}{D_m} - \frac3\rho.
\end{equation*}
Similarly, one can derive that $dx_{m,\ell} / dt \leq 3/\rho - 1/(n D_m)$ on $(\tau, \tau_m)$. Hence, $d D_m /dt \leq 6/\rho - 2/(n D_m)$, which by \eqref{pf:bbe:apriori:2} and $D_m(\tau) < \delta$ implies \eqref{pf:bbe:De:der} on $(\tau, \tau_m)$.  In fact, the estimates above show that outside of $S_m$, $d D_m /dt \leq -1$ as long as $|I_m (t)| \geq 2$, i.e., $D_m(t) > 0$. This completes the proof of \eqref{pf:bbe:De} for when $|I|$ is even.

If $|I|$ is odd, a similar argument applies. The only difference is that in \eqref{pf:xe1:ita:force} we need to be more precise in the estimates:
\begin{align*}
  \frac{dD_m}{dt} 
  &= \frac1n \sum_{ j = 1 }^{(\ell - k)/2} \Big( \frac{1}{x_{m,\ell} - x_{m,k + 2j - 2}} - \frac{1}{x_{m,\ell} - x_{m,k + 2j - 1}} \Big) \\
  &\qquad
       - \frac1n \sum_{ j = 1 }^{(\ell - k)/2} \Big( \frac{1}{x_{m,k + 2j - 1} - x_{m,k}} - \frac{1}{x_{m,k + 2j} - x_{m,k}} \Big)
       + F_{m,\ell} - F_{m,k} \\
  &= -\frac1n \sum_{ j = 1 }^{(\ell - k)/2} \bigg( \frac{ x_{m,k + 2j - 1} - x_{m,k + 2j - 2} }{(x_{m,\ell} - x_{m,k + 2j - 2} )( x_{m,\ell} - x_{m,k + 2j - 1} )}  \\
  & \qquad \qquad \qquad \qquad + \frac{ x_{m,k + 2j} - x_{m,k + 2j - 1} }{ ( x_{m,k + 2j - 1} - x_{m,k} )( x_{m,k + 2j} - x_{m,k} ) } \bigg)
       + F_{m,\ell} - F_{m,k} \\
  &\leq -\frac1{n D_m^2} \sum_{ j = 1 }^{(\ell - k)/2} ( x_{m,k+2j} -  x_{m,k+2j-2}) + \frac6\rho
  = -\frac1{n D_m} + \frac6\rho.
\end{align*} 
This concludes the proof of \eqref{pf:bbe:Ie}. Finally, the proof of \eqref{pf:bbe:claim} follows by repeating the construction of $\delta$ small enough and $m_0$ for each $i \in \{1,\ldots,n\}$, and then taking the minimal and maximum value respectively over~$i$.
\end{proof}

\section{The Hamilton--Jacobi equations}
\label{s:HJ}

In this section we introduce the notion of viscosity solutions for the Hamilton--Jacobi equation~\eqref{HJn:formaller} describing the particle system, as well as the limit equation~\eqref{HJ:formal}. 

Let us expand on the brief introduction of viscosity solutions at the end of Section~\ref{sec:intro-many-particle-limit}. Viscosity solutions are the natural generalized notion of solutions for this type of nonlocal Hamilton--Jacobi equations due to their comparison principle structure. The classical theory for local Hamilton--Jacobi equations goes back to the work of Crandall and Lions~\cite{CrandallLions83}; see \cite{CrandallIshiiLions92,Giga06} for the standard treatment of the theory and references. The general idea of viscosity solutions is to use the comparison principle as the defining property. We first identify a sufficiently large class of functions, called test functions, for which the property of locally being (strict) subsolutions or supersolutions has a classical meaning. We then require that a candidate viscosity solution satisfies a comparison principle with all such classical strict subsolutions and supersolutions. For the standard first-order and second-order partial differential equations one can choose smooth functions or even second order polynomials as the class of test functions; there is no unique choice. For more singular equations, it might be necessary to restrict the class of the test functions to be able to give the operator a classical meaning, but the choice of test functions may be subtle; choosing too few could make the comparison principle fail, whereas choosing too many could make the proof of existence or stability of solutions more challenging. This idea of restricting the class of test functions depending on the operator first appeared in \cite{IshiiSouganidis95} in the context of level set equations and in \cite{OhnumaSato97} for more general singular equations.

\subsection{Notation}

Throughout this section, we set $Q := (0,\infty)\times \R$ and $Q_T := (0,T) \times \R$. For any $\e > 0$, the staircase function $E_\e$ was already defined in \eqref{E1n}. For a function $f : \R^d \to \R$, we set $f_*, f^*$ to be its lower/upper semi-continuous envelope over all its variables. For example,
  \begin{equation*}
    E_\e^* = E_\e, \qquad
    E_{\e,*} (\alpha) = \e \Big\lceil \frac \alpha \e \Big\rceil - \frac\e2.
  \end{equation*}
Finally, $BUC$ is the space of bounded uniformly continuous functions, $C_b$ is the space of bounded continuous functions and $C_b^2$ is the space of bounded twice continuously differentiable functions: $C_b = C \cap L^\infty$ and $C_b^2 = C^2 \cap L^\infty$. 

\subsection{Hamilton--Jacobi equation at $\e>0$}

With $\e := \frac1n$ we rewrite the Hamilton--Jacobi equation in \eqref{HJn:formaller} as
\begin{align}
\label{HJeps-p}
\tag{HJ$_\e$}
u_t &= \mathcal M_\e[u] |u_x|, & \text{on } \R \times (0, \infty),
\intertext{with initial condition}
u(0, \cdot) &= u^\circ,\nonumber
\end{align}
where the nonlocal operator $\mathcal M_\e$ is formally defined as
\begin{align*}
\mathcal M_\e[w](x) = \int_\R E_\e[w(x+z) - w(x)] \frac{dz}{z^2}.
\end{align*}
In this section we switch the parameters $n$ and $\e$. Given any $\e > 0$, the level sets of $u^\circ$ at $\e \Z$ determine $(\bx^\circ, \bb^\circ)$, which in particular prescribes the initial number $n_\e$ of charged particles. Then, by the same formal arguments as in the introduction, it readily follows that the related ODE is
\[
  \frac{dx_i}{dt} = \e \sum_{ j \neq i }^{n_\e} \frac{b_i b_j}{x_i - x_j}
     \qquad t \in (0,T), \ i = 1,\ldots, n_\e.
\] 

Before constructing a rigorous definition of \eqref{HJeps-p}, we provide an explicit example of the expected solution to \eqref{HJeps-p} for simple choices of the initial condition $u^\circ$. 

\begin{ex} \label{ex:HJe}
For any $\e > 0$ and smooth even initial data $u^\circ$ which is strictly decreasing in $[0, \infty)$ and satisfies $\sup u^\circ - \inf u^\circ \leq \e$, one can check that the function $u(t,x) = u^\circ(\sqrt{x^2 + \e t})$ satisfies \eqref{HJeps-p} at all points $(t,x)$ where $\partial_x u(t,x) \neq 0$, i.e.\ at points $x \neq 0$. As a concrete example, consider $u^\circ(x) := \e/(x^2 + 1)$ that yields the solution $u(t,x) = \e/(x^2 + \e t + 1)$. After giving a rigorous definition to \eqref{HJeps-p}, it turns out to be  the unique viscosity solution with the initial data $u^\circ$. While it is not difficult to verify this, we do not provide the details, and refer instead to the proof of Lemma \ref{l:Pn:to:HJe} for a possible procedure. 

Instead, we focus on how the formula for $u(t,x)$ can be deduced from the particle system \eqref{Pn} with time-rescaling factor $\e n$ by using the method of characteristics. The initial data $u^\circ$ describes a continuum of two-particle systems parametrized by their initial position $0 < a < \infty$: a particle of charge $+1$ located at $x_1^{a\circ} = -a$ and a particle of charge $-1$ located at $x_2^{a\circ} = a$. The unique solution of the rescaled \eqref{Pn} up to the annihilation time $\tau_1 = a^2/\e$ is $x_1^a(t) = -\sqrt{a^2 - \e t}$, with $x_2^a(t) = -x_1^a(t)$. Note that that the parameter $a$ generates a foliation by trajectories of the half-plane $\{(t, x) : t > 0\}$; see Figure \ref{fig:ex:HJe}.  As $u$ is constant along the trajectories of the particles, $u(t, x_i^a(t)) = u^\circ(x_i^{a\circ})$, where $a = a(t,x) = \sqrt{x^2 + \e t}$ is the unique parameter so that the point $(t,x)$ lies on either of the trajectories $x_1^a$ or $x_2^a$. 
\end{ex}

\begin{figure}[h]
\centering
\begin{tikzpicture}[scale=0.5, >= latex]
\draw[->] (0,0) -- (0,5) node[above] {$t$};
\draw[->] (-5,0) -- (5,0) node[right] {$x$};
\foreach \a in {0.5, 1, 1.5, 2} {
    \draw[domain=-\a:0, smooth, thick, red] plot (\x,{\a*\a-\x*\x});  
    \draw[domain=0:\a, smooth, thick, blue] plot (\x,{\a*\a-\x*\x}); 
    }
\foreach \a in {2.5, 3, 3.5, 4} {
    \draw[domain=-\a:-sqrt(\a*\a - 4.5), smooth, thick, red] plot (\x,{\a*\a-\x*\x});  
    \draw[domain=sqrt(\a*\a - 4.5):\a, smooth, thick, blue] plot (\x,{\a*\a-\x*\x}); 
    }
\end{tikzpicture}
\caption{The foliation of the half-plane $t > 0$ by the trajectories of pairs in Example~\ref{ex:HJe}. Trajectories of particles with positive charge are coloured red; those with negative charge blue. These trajectories are also the level sets of the solution $u$ of \eqref{HJeps-p} constructed in Example~\ref{ex:HJe}.}
\label{fig:ex:HJe}
\end{figure}
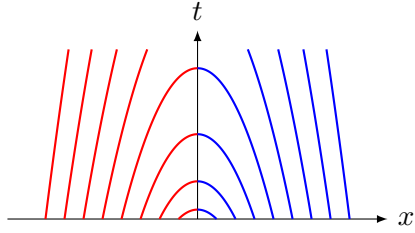

We continue with constructing a rigorous definition for \eqref{HJeps-p}. The difference with the setting in \cite{ForcadelImbertMonneau09} is that we consider the kernel $z^{-2}$ in the definition of $\mathcal M_\e$, whereas \cite{ForcadelImbertMonneau09} considers an integrable one. As a result, the integral in the definition of $\cM_\e [w] (x)$ does not converge as a Lebesgue integral if $\partial_x w (x) = 0$ due to the singularity of  $z^{-2}$ at $0$.

We construct a proper replacement for $\mathcal M_\e[u] |u_x|$ in two steps. In the first step, we follow the idea in the works of Sayah~\cite{Sayah91}, Imbert, Monneau and Rouy~\cite{ImbertMonneauRouy08} and Jakobsen and Karlsen~\cite{JakobsenKarlsen05} to replace $u$ in the integral of $\mathcal M_\e[u]$ by a test function $\phi$ on the range $-\rho < z < \rho$, where $\rho > 0$ is a (small) parameter. The test function $\phi$ will later be taken as the same test functions used in the definition of viscosity solutions. It will also turn out that the notion of viscosity solutions which follows is independent of the choice of $\rho$.

\begin{defn}[Hamiltonians at $\e>0$]
\label{d:eps-ham}
Fix $\rho>0$, $u \in L^\infty(\R \times [0,\infty))$, $t > 0$, $x\in \R$ and $\phi\in C^2_b(B_\delta(t) \times B_\rho(x))$ for some $\delta > 0$. If $\phi_x(t,x) \neq 0$, we 
define 
\begin{align*}
\o H_{\rho,\e} [\phi,u](t, x) 
&:= \o M_{\rho,\e}[\phi(t,\cdot),u(t, \cdot)](x) \,|\phi_x(t, x)| 
 \\
\underline H_{\rho,\e} [\phi,u](t,x) 
&:= \underline M_{\rho,\e}[\phi(t,\cdot),u(t, \cdot)](x) \,|\phi_x(t, x)|,
\end{align*}
where
\begin{align} \label{oMrhoe}
\o M_{\rho,\e}[\phi,u](x) &:= \pv \int_{B_\rho} E_\e^*\big[ \phi(x+z)-\phi(x)\big] \frac{dz}{z^2}
+ \int_{B_\rho^c} E_\e^*\big[ u(x+z)-u(x)\big] \frac{dz}{z^2}, \\
\underline M_{\rho,\e}[\phi,u](x) &:= \pv \int_{B_\rho} E_{\e,*}\big[ \phi(x+z)-\phi(x)\big] \frac{dz}{z^2}
+ \int_{B_\rho^c} E_{\e,*}\big[ u(x+z)-u(x)\big] \frac{dz}{z^2}.
\end{align}
\end{defn}

The Hamiltonians $\o H_{\rho,\e}$ and $\underline H_{\rho,\e}$ will replace the right-hand side of \eqref{HJeps-p} in the definition of the viscosity subsolution and supersolution. Since the Hamiltonians are discontinuous, we need to choose the one with the correct semi-continuity to be able to pass to various limits.

The next lemma shows that the expressions $\o H_{\rho,\e}$ and $\underline H_{\rho,\e}$ above are well-defined.

\begin{lem} \label{l:Hpe}
Let $\e$, $\rho$, $u$, $t$, $x$ and $\phi$ be as in Definition~\ref{d:eps-ham}. Then, dropping the dependence on $t$,
\begin{enumerate}[(i)]
  \item \label{l:Hpe:Bp} $\exists \, \rho_0 > 0 \: \forall \, \tilde \rho \in (0, \rho_0] :  \displaystyle \pv \int_{B_{\tilde \rho}} E_\e^*\big[ \phi(x+z)-\phi(x)\big] \frac{dz}{z^2} = 0$;
  \item \label{l:Hpe:Bpc} $\displaystyle \bigg| \int_{B_\rho^c} E_\e^*\big[ u(x+z)-u(x)\big] \frac{dz}{z^2} \bigg| \leq \frac{ 4 \|u\|_\infty + \e}\rho$.
\end{enumerate}
The above two properties also hold when $E_\e^*$ is replaced with $E_{\e,*}$. 
\end{lem}

\begin{proof}
Let $\e$, $\rho$, $u$, $t$, $x$, $\phi$ be given. Property \eqref{l:Hpe:Bpc} follows by simply using $| E_\e^*[\alpha] | \leq |\alpha| + \e/2$. For Property \eqref{l:Hpe:Bp}, we assume for convenience that $\phi'(x) > 0$. Since $\phi \in C^1(\R)$, there exists $\rho_0 > 0$ such that 
\begin{equation*} 
  \phi(x+z)-\phi(x)
  \in \left\{ \begin{array}{ll}
    (-\e, 0)
    &\text{if } z \in (-\rho_0, 0)  \\
    (0, \e)
    &\text{if } z \in (0, \rho_0).
  \end{array} \right.
\end{equation*}
Hence,
\begin{equation*}
  2 E_\e^* \big[ \phi(x+z)-\phi(x)\big]
  = \left\{ \begin{array}{ll}
    -\e
    &\text{if } z \in (-\rho_0, 0)  \\
    \e
    &\text{if } z \in (0, \rho_0);
  \end{array} \right.
\end{equation*}
Property \eqref{l:Hpe:Bp} follows.

The proof for $E_{\e,*}$ is analogous.
\end{proof} 

Even after replacing $\mathcal M_\e[u] |u_x|$ by either $\o H_{\rho,\e} [\phi,u]$ or $\underline H_{\rho,\e} [\phi,u]$ for a regular test function $\phi$, we still require $\phi_x(t,x) \neq 0$ in Definition \ref{d:eps-ham} for the Hamiltonians. This is the key difference with \cite{ImbertMonneauRouy08}, where the corresponding integral is defined for any smooth $\phi$. Here, the requirement $\phi_x(t,x) \neq 0$ is more than a technical issue; at annihilation points of the particle system we necessarily have $\partial_x v = 0$, and thus we require at least some test functions with $\phi_x(t,x) = 0$. This issue is avoided in Slep\v{c}ev~\cite{Slepcev03} and Forcadel, Imbert and Monneau~\cite{ForcadelImbertMonneau09} by replacing the singular kernel $1/z^2$ in the operator $\cM_\e$ in \eqref{HJn:formaller} by a smooth one. In that case, the parameter $\rho$ need not be introduced. However, in our case, regularizing the kernel breaks the connection with the particle system \eqref{Pn}; we therefore take a different approach.

This brings us to step 2 of the construction of a rigorous definition to \eqref{HJeps-p}. In this step we reduce the class of all regular enough test functions. This idea is briefly addressed in the introduction and at the start of Section \ref{s:HJ}. We recall that we may remove as many test functions as necessary as long as we can still prove a comparison principle. While Definition \ref{d:eps-ham} suggests to remove those with $\phi_x(\o t, \o x) = 0$ at the test point $(\o t, \o x)$, the discussion in the previous paragraph demonstrates that this would remove too many test functions. Hence, we need to appropriately extend the Hamiltonian $\o H_{\rho,\e}$ (and $\underline H_{\rho,\e}$) at least for some functions $\phi$ with $\partial_x \phi(\o t, \o x) = 0$. Since $|\partial_x \phi(\o t, \o x)| = 0$, a natural extension is to define $\o H_{\rho,\e} [\phi, u] (\o t, \o x) = 0$. Since pairs of annihilating particles move along parabolas (see Theorem~\ref{t:Pn}(\ref{t:Pn:C12}),~(\ref{t:Pn:LB:col})), there are smooth functions $\phi$ for which $\partial_t \phi(\o t, \o x)$ can have either sign (see Example~\ref{ex:HJe}), which does not fit \eqref{HJeps-p} with $\o H_{\rho,\e} [\phi, u] (\o t, \o x) = 0$. To remove such functions from the class of test functions, we require $\partial_{xx} \phi(\o t, \o x) = 0$. Our proof of the comparison principle allows us to go even further; we restrict the subclass of regular test functions with $\phi_x(\o t, \o x) = 0$ to functions of fourth-order growth with the specific form $c (x - \o x)^4 + g(t)$. This restriction is helpful later on (see Lemma \ref{l:Pn:to:HJe}) in the proof of Theorem \ref{mthm:B} where we show that the solution of the system of ODEs given by \eqref{Pn} translates to a solution of \eqref{HJeps-p}.

\begin{defn}[$\rho$-sub- and $\rho$-supersolutions for $\e>0$] \label{def:HJe:VS}
Let $\rho, \e > 0$.
\begin{itemize}
\item 
Let $u:Q \to \R$ be upper semi-continuous and bounded. The function $u$ is a $\rho$-subsolution of \eqref{HJeps-p} in $Q$ if the following holds: whenever $\phi\in C^2(Q)$ is such that $u-\phi$ has a global maximum at $(\o t, \o x)$, we have
\[
\phi_t(\o t, \o x) \leq 
\begin{cases}
\o H_{\rho,\e}[\phi,u](\o t, \o x)& \text{if $\phi_x(\o t, \o x) \neq 0$,}\\
0 & \text{if $\phi( t,  x)$ is of the form $c |x - \o x|^4 + g(t)$,}\\
+\infty & \text{otherwise}.
\end{cases}
\]
\item Let $v:Q \to \R$ be lower semi-continuous and bounded. The function $v$ is a $\rho$-supersolution of \eqref{HJeps-p} in $Q$ if the following holds: whenever $\psi \in C^2(Q)$ is such that $v-\psi$ has a global minimum at $(\o t, \o x)$, we have
\[
\psi_t(\o t, \o x) \geq 
\begin{cases}
\underline H_{\rho, \e}[\psi,v](\o t, \o x)& \text{if $\phi_x(\o t, \o x) \neq 0$,}\\
0 & \text{if $\phi( t,  x)$ is of the form $c |x - \o x|^4 + g(t)$,}\\
-\infty & \text{otherwise}.
\end{cases}
\]
\end{itemize}
A function $u:Q \to \R$ is a $\rho$-solution of \eqref{HJeps-p} in $Q$ if $u^*$ is a $\rho$-subsolution and $u_*$ is a $\rho$-supersolution.
\end{defn}

We remark that, as usual, we extend subsolutions $u$ and a supersolutions $v$ to $t=0$ by
\[
  u(0,x) := u^*(0,x)
  \quad \text{and} \quad
  v(0,x) := v_*(0,x)
  \quad \text{for all } x \in \R. 
\]
In addition, without loss of generality, we may assume in Definition \ref{def:HJe:VS} that the maximum of $u-\phi$ (and the minimum of $v-\psi$) is strict, that $(u - \phi)(\o t, \o x) = 0$ and that $\lim_{|t| + |x| \to \infty} (u - \phi)(t,x) = -\infty$. Indeed, if the maximum at $(\o t, \o x)$ is not strict, then we can approximate $\phi$ by 
\[
  \phi_\delta(t,x) := \phi(t,x) +  \delta |x - \o x|^4 + \delta |t - \o t|^2
\]
as $\delta \searrow 0$. Indeed, the maximum of $u-\phi_\delta$ is strict, and by Lemma \ref{l:Hp}\eqref{l:Hp:Bp} the right-hand side in Definition \ref{def:HJ:VS} converges as $\delta \searrow 0$ to that of $\phi$.

In the following lemma we show that this Definition \ref{def:HJe:VS} does not depend on $\rho$. Therefore we can simply talk about subsolutions, supersolutions and viscosity solutions of \eqref{HJeps-p}.

\begin{lem}[Independence of $\rho$] \label{l:HJe:rho}
If $u$ is a $\rho$-sub- or $\rho$-supersolution of \eqref{HJeps-p} for some $\rho > 0$, then it is a $\tilde\rho$-sub- or $\tilde\rho$-supersolution of \eqref{HJeps-p}, respectively, for any $\tilde \rho > 0$.
\end{lem}

\begin{proof}
This is a modification of the proof of \cite[Prop.~II.1]{Sayah91}.
We prove the lemma for subsolutions; the proof for supersolutions is analogous. To prove that $u$ is a $\tilde \rho$-subsolution of \eqref{HJeps-p}, let $\tilde \phi$ be any corresponding test function such that $u - \tilde \phi$ has a global maximum at $( t,  x)$. We assume for convenience that $u( t,  x) = \tilde \phi ( t,  x)$. If $\tilde \phi_x(t,x) = 0$, then $u$ satisfies the condition for being a $\tilde \rho$-subsolution of \eqref{HJeps-p} for any $\tilde \rho > 0$. Hence, in the remainder we may assume that $\tilde \phi_x(t,x) \neq 0$.

We start with the case $\tilde \rho > \rho$. Since $u$ is a $\rho$-subsolution, we have
\[
\tilde \phi_t( t,  x) \leq \o H_{\rho,\e}[\tilde\phi, u]( t,  x).
\]
Since $\tilde \phi(t, x+z) - \tilde \phi(t, x) \geq  u(t, x+z)-u(t, x)$ for all $z \in \R$ and since $E_\e^*$ is non-decreasing, we obtain from the definition of $\o H_{\rho,\e}$ that $\o H_{\rho,\e}[\tilde\phi, u]( t,  x) \leq \o H_{\tilde \rho,\e}[\tilde\phi, u]( t,  x)$. This shows that $u$ is a $\tilde \rho$-subsolution.

It is left to treat the case $\tilde \rho < \rho$. Let $(\phi^k) \subset C^2(Q)$ be a sequence of test functions satisfying
\begin{equation*}
  \phi^k \left\{ \begin{array}{ll}
    = \tilde \phi
    &\text{on } B_{\tilde \rho} (t,x) \\
    \geq u
    &\text{on } B_{\tilde \rho} (t,x)^c
  \end{array} \right.
\end{equation*}
such that $\phi^k \searrow u$ pointwise on the interior of $B_{\tilde \rho} (t,x)^c$ as $k \to \infty$. Since $u$ is a $\rho$-subsolution and $u - \phi^k$ has a global maximum at $(t,x)$, we have
\[
\tilde \phi_t( t,  x) = \phi_t^k( t,  x) \leq \o H_{\rho,\e}[\phi^k, u]( t,  x).
\]
Next we prepare for passing to the limit $k \to \infty$ in the right-hand side. To avoid clutter, we remove the time variable. By construction of $\phi^k$,
\[
  \limsup_{k \to \infty} \phi^k(x+z) - \phi^k(x)
  \leq \left\{ \begin{array}{ll}
    \tilde \phi(x+z) - \tilde \phi(x)
    &\text{if } |z| < \tilde \rho  \\
    u(x+z) - u(x)
    &\text{if } |z| > \tilde \rho 
  \end{array} \right.
\]
for a.e.\ $z \in \R$. Then, since $E_\e^*$ is non-decreasing and upper semi-continuous, 
\[
  \limsup_{k \to \infty} E_\e^* \big[ \phi^k(x+z) - \phi^k(x) \big]
  \leq \left\{ \begin{aligned}
    &E_\e^* \big[ \tilde \phi(x+z) - \tilde \phi(x) \big]
    &&\text{if } |z| < \tilde \rho  \\
    &E_\e^* \big[ u(x+z) - u(x) \big]
    &&\text{if } |z| > \tilde \rho 
  \end{aligned} \right.
\]
for a.e.\ $z \in \R$. Hence, from the definition of $\o H_{\rho,\e}$ and Fatou's lemma, we obtain that
\[
  \limsup_{k \to \infty} \o H_{\rho,\e}[\phi^k, u]( t,  x)
  \leq \o H_{\tilde \rho,\e}[\tilde \phi, u]( t,  x).
\]
This shows that $u$ is a $\tilde \rho$-subsolution.
\end{proof}

\begin{thm}[Comparison principle for \eqref{HJeps-p}]
\label{t:CPe}
Let $u$ be a subsolution and $v$ be a supersolution of \eqref{HJeps-p}. Assume that for each $T>0$, 
\begin{equation} \label{CPe:far-field}
  \lim_{R\to\infty} \sup \Big\{ \;u(t,x)-v(t,x): \ |x|\geq R, \ t\in (0,T]\Big\} \leq 0.  
\end{equation}
Then $u(0,\cdot) \leq v(0,\cdot)$ on $\R$ implies $u\leq v$ on $Q$.
\end{thm}

\begin{proof}
Suppose that the inequality $u\leq v$ does not hold on $Q$; then there exists $T>0$ such that $\theta := \sup_{Q_T} (u-v)  >0$.
For $\eta>0$ and $\gamma>0$ define the function
\begin{align} 
\label{comp-Phi}
\Phi(t,x,s,y) := u(t,x) - v(s,y) - \frac{(t-s)^2}{2\gamma} - \frac{(x-y)^4}{4\gamma} - \frac\eta{T-t} - \frac\eta{T-s}.
\end{align}
For sufficiently small $\eta$ it follows that $\sup_{Q_T\times Q_T} \Phi \geq \theta/2$, independently of $\gamma$.
By \eqref{CPe:far-field} and the semi-continuity of $u$ and $v$, the supremum is achieved at some point $(\o t, \o x, \o s, \o y)\in \overline {Q_T}\times \overline{Q_T}$.
By the divergence of $\Phi$ as $s\nearrow T$ or $t\nearrow T$  we have $\o t, \o s< T$. 
 
We now show by the usual arguments that for sufficiently small $\gamma>0$ we have  $\o t, \o s>0$. Assume, to force a contradiction, that there exists a sequence $\gamma_n\searrow 0$ such that for all $n$ the corresponding maxima $(\o t_n, \o x_n, \o s_n, \o y_n)$ satisfy $\o t_n=0$. As $u - v$ is bounded, by the structure of $\Phi$ we  have $\o s_n\to 0$ and $\o x_n-\o y_n\to0$, and by \eqref{CPe:far-field} we can assume by taking a subsequence if necessary that there exists $\o x\in \R$ such that $\o x_n,\o y_n\to \o x$ as $n\to\infty$. We then estimate, using the semi-continuity of $u$ and $v$,
\[
0< \frac \theta 2 \leq \limsup_{n\to\infty} \Phi_{\gamma_n}(0, \o x_n, \o s_n, \o y_n) 
 \leq \limsup_{n\to\infty} u(0,\o x_n) - v(\o s_n, \o y_n) 
 \leq u(0,\o x) - v(0,\o x) \leq 0,
\]
which is a contradiction. Therefore we can fix $\gamma$ and assume that $\o t, \o s>0$.

We therefore have $(\o t, \o x, \o s, \o y)\in Q_T\times Q_T$. Next we will obtain the contradiction to $\theta>0$ by constructing test functions $\phi$ and $\psi$ for the subsolution $u$ and the supersolution $v$, respectively, for which either $u$ or $v$ does not satisfy Definition \ref{def:HJe:VS}. With this aim, we set
\begin{align*}
\phi(t,x) &:=  \frac{(t-\o s)^2}{2\gamma } + \frac{(x-\o y)^4}{4\gamma} + \frac\eta{T-t}\\
\psi(s,y) &:= - \frac{(\o t-s)^2}{2\gamma } - \frac{(\o x-y)^4}{4\gamma } - \frac\eta{T-s}.
\end{align*}
Since $u-\phi = \Phi(\cdot, \cdot, \o s, \o y) + C$ has a global maximum at $(\o t,\o x)$ and $v-\psi = -\Phi(\o t, \o x, \cdot, \cdot) + C$ a global minimum at $(\o s,\o y)$, $\phi$ and $\psi$ are admissible test functions\footnote{As usual, we replace $\eta/(T - t)$ in $\phi(t,x)$ far enough away from $\o t$ by a regular extension beyond $t = T$, without changing notation. We do the same for $\psi$.}. We compute
\begin{align} \label{pf:phit-psit}
\phi_t(\o t,\o x) - \psi_s(\o s,\o y)
= \frac\eta{(T-\o t)^2} + \frac\eta{(T-\o s)^2} > 0.
\end{align}

Next we separate two cases. If $\o x = \o y$, then $\phi(t,x) = \frac{|x - \o x|^4}{4\gamma} + g(t)$ and $ \psi(s, y) = -\frac{|y - \o y|^4}{4\gamma} + h(s)$, and Definition \ref{def:HJe:VS} yields
\[
\phi_t(\o t,\o x) \leq 0
\quad \text{and} \quad 
\psi_s(\o s,\o y) \geq 0.
\]
This contradicts with \eqref{pf:phit-psit}.

In the second case, $\o x \not= \o y$, we note that $\phi_x(\o t, \o x) = \psi_y(\o s, \o y) = \frac1\gamma (\o x - \o y)^3 \neq 0$. We claim that
\begin{equation}
\label{ineq:CPe:2}
\int_{B_\rho^c} \biggl\{ E_\e^*\big[u(\o t,\o x+z) - u(\o t,\o x)\big] - 
  E_{\e,*} \big[v(\o s,\o y+z) - v(\o s, \o y)\big]\biggr\} \frac{dz}{z^2}
  \leq 0
  \quad \text{for any } \rho>0.
\end{equation}
Then, as $\phi_x(\o t, \o x) = \psi_y(\o s, \o y)$, 
\begin{multline*}
|\phi_x(\o t, \o x)|^{-1}(\o H_{\rho,\e}[\phi,u](\o t,\o x) - \underline H_{\rho,\e}[\psi,v](\o s,\o y))\\
\leq \pv \int_{B_\rho} E_\e^*\big[\phi(\o t, \o x+ z) - \phi(\o t, \o x)\big] \frac{dz}{z^2}
 - \pv \int_{B_\rho} E_{\e,*}\big[\psi(\o s, \o y+ z) - \psi(\o s, \o y)\big] \frac{dz}{z^2},
\end{multline*}
and this expression equals zero for $\rho$ small enough by Lemma \ref{l:Hpe}\eqref{l:Hpe:Bp}. This together with \eqref{pf:phit-psit} contradicts with Definition \ref{def:HJe:VS}.

It is left to prove the claim \eqref{ineq:CPe:2}.
Assume for convenience that $\o x -\o y >0$. 
For any $\delta>0$ and $z\in \R$, we have
\[
\Phi\big( \o t, \o x+ z, \o s, \o y+ z +\delta |z|\big)
\leq \Phi(\o t, \o x, \o s, \o y).
\]
We rewrite this inequality to find
\[
u(\o t, \o x+ z) - u(\o t, \o x) \leq v(\o s, \o y + z + \delta |z|) - v(\o s, \o y)
+ \frac1{4\gamma} \Big( \big |\o x - \o y - \delta |z|\,\big|^4 - |\o x - \o y |^4\Big).
\]
Since the final term is strictly negative if $|z|< (\o x - \o y)/\delta$, we have 
\begin{equation}
\label{ineq:l:CPe}
E^*_\e\big[ u(\o t, \o x+ z) - u(\o t, \o x)\big]
\leq E_{\e,*} \big[v(\o s, \o y + z + \delta |z|) - v(\o s, \o y)\big]
\qquad \text{for }|z|< (\o x - \o y)/\delta.
\end{equation}

We now split the left-hand side of~\eqref{ineq:CPe:2} into two parts:
\begin{align*}
\eqref{ineq:CPe:2} 
&= \int_{B_\rho^c} \biggl\{ E_\e^*\big[u(\o t,\o x+z) - u(\o t,\o x)\big] - 
  E_{\e,*} \big[v(\o s,\o y+z + \delta |z|) - v(\o s, \o y)\big]\biggr\} \frac{dz}{z^2}\\
  &\qquad + \int_{B_\rho^c} 
    \biggl\{ E_{\e,*} \big[v(\o s,\o y+z+ \delta |z|) - v(\o s, \o y)\big] - 
  E_{\e,*} \big[v(\o s,\o y+z) - v(\o s, \o y)\big]\biggr\} \frac{dz}{z^2} =: T_1 + T_2.
\end{align*}
Assuming that $\delta$ is small enough such that $\rho < (\o x - \o y)/\delta$, we further split the first integral $T_1$ into near-field and far-field parts:
\begin{align*}
T_1 &:= T_{1,\mathrm{near}} + T_{1,\mathrm {far}} \\
&:= \int_{B_{(\o x-\o y)/\delta}\setminus B_\rho} 
  \biggl\{ E_\e^*\big[u(\o t,\o x+z) - u(\o t,\o x)\big] - 
  E_{\e,*} \big[v(\o s,\o y+z + \delta |z|) - v(\o s, \o y)\big]\biggr\} \frac{dz}{z^2}\\
  &\qquad + \int_{B_{(\o x-\o y)/\delta}^c} 
  \biggl\{ E_\e^*\big[u(\o t,\o x+z) - u(\o t,\o x)\big] - 
  E_{\e,*} \big[v(\o s,\o y+z + \delta |z|) - v(\o s, \o y)\big]\biggr\} \frac{dz}{z^2}
\end{align*}
The inequality~\eqref{ineq:l:CPe} implies that $T_{1,\mathrm{near}}\leq 0$. For the integral $T_{1,\mathrm{far}}$ we use Lemma \ref{l:Hpe}\eqref{l:Hpe:Bpc} to estimate
\[
T_{1,\mathrm {far}}\leq
\frac{ 4 \|u\|_\infty + 4 \|v\|_\infty + 2\e}{|\o x-\o y|} \delta
\]
which converges to zero as $\delta\to0$. 

Finally, to estimate the term $T_2$ we set 
\[
J(z) := \begin{cases}
\displaystyle \frac 1{z^2} & \text{if }|z|> \rho\\
0 & \text{otherwise}
\end{cases}
\qquad \text{and}\qquad
J^\delta(z) := \begin{cases}
\displaystyle \frac{1-\delta}{z^2} & \text{if } z < -(1-\delta)\rho\\
0 &\text{if }-(1-\delta)\rho \leq  z\leq (1+\delta)\rho\\
\displaystyle \frac{1+\delta}{z^2} & \text{if } z>(1+\delta)\rho.
\end{cases}
\]
Note that $J^\delta$ is the push-forward of $J$ under the map $z\mapsto z + \delta |z|$, and that $J^\delta$ converges to $J$ in $L^1(\R)$ as $\delta \to0$. Since $z\mapsto E_{\e,*} \big[v(\o s,\o y+z) - v(\o s, \o y)\big]$ is an element of $L^\infty(\R)$, it follows that 
\[
T_2= \int_\R   E_{\e,*} \big[v(\o s,\o y+z) - v(\o s, \o y)\big]\Big\{ J^\delta(z) - J(z)\Big\}
\, dz \xrightarrow {\delta\to0}  0.
\]
This concludes the proof of the theorem. 
\end{proof}

\subsection{Limiting Hamilton--Jacobi equation}

The limiting Hamilton--Jacobi equation is studied in 
\cite{ImbertMonneauRouy08,BilerKarchMonneau10}. It reads
\begin{align}
\label{HJ-p}
\tag{HJ}
u_t &= \mathcal I[u] |u_x|, & \text{on } \R \times (0, \infty),
\intertext{with initial condition}
u(0, \cdot) &= u^\circ,\nonumber
\end{align}
where the nonlocal operator $\mathcal I$ is defined on $C_b^2(\R)$ as
\begin{align*}
\mathcal I[w](x) := \int_\R (w(x+z) - w(x) - z w'(x)) \frac{dz}{z^2}.
\end{align*}
It is easy to see from $\text{pv}\int \frac{dz}z = 0$ that this definition is equivalent to that in \eqref{Iu}. Here, however, we do not need a principle-value integral. Indeed, by using a Taylor expansion of $w$ around $x$, it follows that the integrand is bounded around $z = 0$.

Next, we briefly recall several results from \cite{Sayah91,ImbertMonneauRouy08}. First, to define viscosity solutions for \eqref{HJ-p}, we need to introduce a $\rho$-Hamiltonian for any $\rho > 0$ similar to Definition \ref{d:eps-ham}. Due to the absence of the staircase approximation $E_\e$, there is no need to develop different Hamiltonians for sub- and supersolutions.

\begin{defn}[Limiting Hamiltonian] 
\label{d:lim-h}
For $\rho>0$ fixed, $u\in L^\infty(\R \times [0, \infty))$, $t > 0$, $x \in \R$ and $\phi\in C^2_b(B_\delta(t) \times B_\rho(x))$ for some $\delta > 0$, we define
\[
H_\rho[\phi,u](t,x) := \mathcal I_\rho[\phi(t,\cdot),u(t,\cdot)](x) |\phi_x(t,x)|
\]
where
\begin{equation} \label{Irho}
\mathcal I_\rho[\phi,u](x) 
:= \int_{B_\rho} \big( \phi(x+z)-\phi(x) - \phi'(x)z\big) \frac{dz}{z^2}
   + \int_{B_\rho^c} \big( u(x+z)-u(x)\big) \frac{dz}{z^2}.
\end{equation}
\end{defn}

\begin{lem} \label{l:Hp} 
Let $\rho$, $u$, $t$, $x$ and $\phi$ be as in Definition~\ref{d:lim-h}. Then, dropping the dependence on $t$, $\mathcal I_\rho[\phi,u](x)$ is finite and
\begin{enumerate}[(i)]
  \item \label{l:Hp:Bp} if $\phi(x) = (x - y)^4$, then $\displaystyle \int_{B_\rho} \big( \phi(x+z)-\phi(x) - \phi'(x)z\big) \frac{dz}{z^2} = 12 (x-y)^2 \rho + \frac23 \rho^3$;
  \item \label{l:Hp:Bpc} $\displaystyle \bigg| \int_{B_\rho^c} \big( u(x+z)-u(x)\big) \frac{dz}{z^2} \bigg| \leq \frac{ 4 \|u\|_\infty }\rho$.
\end{enumerate} 
\end{lem}

\begin{proof}
Property \eqref{l:Hp:Bpc} is immediate. As already argued above, the first integral in \eqref{Irho} is finite; hence $\mathcal I_\rho[\phi,u](x)$ is well defined. In the special case $\phi(x) = (x - y)^4$, we compute
$$
  \phi(x +z) - \phi(x) - \phi'(x) z = 6 (x-y)^2 z^2 - 4 (x-y) z^3 + z^4;
$$
Property \eqref{l:Hp:Bp} follows.
\end{proof}

We define viscosity solutions with the same restricted class of test functions for the $\e$-problem \eqref{HJeps-p}, because this choice simplifies the proof of the convergence of solutions as $n\to \infty$. We then prove a comparison principle (Theorem~\ref{t:CP:HJ}), which shows the uniqueness of viscosity solutions. In Corollary~\ref{co:solutions-match-imr} we demonstrate that our notion of viscosity solutions coincides with that of \cite{ImbertMonneauRouy08}, in which the complete class of smooth test functions is considered.

\begin{defn}[$\rho$-sub- and $\rho$-supersolutions for the limiting problem] \label{def:HJ:VS}
Let $\rho > 0$.
\begin{itemize}
\item 
Let $u:Q \to \R$ be upper semi-continuous and bounded. The function $u$ is a $\rho$-subsolution of \eqref{HJ-p} in $Q$ if the following holds: whenever $\phi\in C^2(Q)$ is such that $u-\phi$ has a global maximum at $(\o t, \o x)$, we have
\begin{align}
\label{visc-subsol-cond}
\phi_t(\o t, \o x) \leq 
\begin{cases}
H_\rho[\phi,u](\o t, \o x)& \text{if $\phi_x(\o t, \o x) \neq 0$,}\\
0 & \text{if $\phi( t,  x)$ is of the form $c |x - \o x|^4 + g(t)$,}\\
+\infty & \text{otherwise}.
\end{cases}
\end{align}
\item Let $v:Q \to \R$ be lower semi-continuous and bounded. The function $v$ is a $\rho$-supersolution of \eqref{HJ-p} in $Q$ if the following holds: whenever $\psi\in C^2(Q)$ is such that $u-\psi$ has a global minimum at $(\o t, \o x)$, we have
\[
\psi_t(\o t, \o x) \geq 
\begin{cases}
H_\rho[\psi,v](\o t, \o x)& \text{if $\phi_x(\o t, \o x) \neq 0$,}\\
0 & \text{if $\phi( t,  x)$ is of the form $c |x - \o x|^4 + g(t)$,}\\
-\infty & \text{otherwise}.
\end{cases}
\]
\end{itemize}
A function $u:Q \to \R$ is a $\rho$-solution of \eqref{HJ-p} in $Q$ if $u^*$ is a $\rho$-subsolution and $u_*$ is a $\rho$-supersolution.
\end{defn}

As for Definition \ref{def:HJe:VS}, we may assume also in Definition \ref{def:HJ:VS} that the maximum of $u-\phi$ is strict, that $(u - \phi)(\o t, \o x) = 0$ and that $\lim_{|t| + |x| \to \infty} (u - \phi)(t,x) = -\infty$. In addition, Definition \ref{def:HJ:VS} does not depend on $\rho$ in the sense of the following lemma, and therefore we will not emphasize the dependence on $\rho$ in what follows.

\begin{lem}[Independence of $\rho$; {\cite[Prop.~II.1]{Sayah91}}] \label{l:HJ:rho}
If $u$ is a $\rho$-sub- or $\rho$-supersolution of \eqref{HJ-p} for some $\rho > 0$, then it is a $\tilde\rho$-sub- or $\tilde\rho$-supersolution of \eqref{HJ-p} for any $\tilde \rho > 0$, respectively.
\end{lem}

\begin{thm}[Comparison principle {\cite[Thm.~5]{ImbertMonneauRouy08} and \cite[Thm.~4.3]{BilerKarchMonneau10} }] \label{t:CP:HJ}
Let $u^\circ \in BUC(\R)$, and let $u$ be a subsolution and $v$ be a supersolution of \eqref{HJ-p}. If $u(0,\cdot) \leq u^\circ \leq v(0,\cdot)$ on $\R$, then  $u\leq v$ on $Q$.
\end{thm}

\begin{proof}
Since we made the viscosity solution test weaker when $\phi_x(\o t, \o x) = 0$, we need to check that the comparison principle proved in \cite{ImbertMonneauRouy08,BilerKarchMonneau10} still applies. We follow the proof of Theorem~\ref{t:CPe} with modifications similar to the proof of Theorem~3.1.4 in \cite{Giga06}.  

We first show that
\begin{align}
\label{no-init-layer}
\lim_{\delta \searrow 0} \sup \{u(t, x) - v(s, y) : |x - y| \leq \delta,\ 0 \leq t, s \leq \delta\} \leq 0.
\end{align}
Since $u^\circ$ is uniformly continuous, for every $\alpha > 0$ there exists $c > 0$ such that 
\begin{equation} \label{pf:u0:est}
  u^\circ(x) \leq c |x - y|^4 + \alpha + u^\circ(y) 
  \quad \text{for all } x, y.
\end{equation}
Then, we can fix $\sigma >0$ sufficiently large depending only on $c$ and $\|u\|_\infty$ such that 
\begin{equation} \label{pf:u:est}
  u(t,x) \leq c|x - y|^4 + \sigma t + \alpha + u^\circ(y)
  \quad \text{for any } t \geq 0, \: x, y \in \R.
\end{equation}
To see this, suppose that this order fails for some $t, x, y$ and denote the right-hand side for the fixed $y$ as $\phi(t, x)$. Then, since $u$ is bounded, $u - \phi$ has a positive global maximum at a point $(\o t, \o x)$. Since by \eqref{pf:u0:est} there holds $u(0, \cdot) < \phi(0, \cdot)$, we have $\o t > 0$. We cannot have $\o x = y$, because otherwise $\phi_x(\o t, \o x) = 4c(\o x - y)^3 = 0$, which contradicts with $\phi_t = \sigma > 0$ and $u$ being a subsolution. Thus, $\phi_x(\o t, \o x) \neq 0$. Since $u$ is a subsolution, we have
\begin{equation} \label{pf:u:1VS:est}
  \sigma
  = \phi_t(\o t, \o x)
  \leq H_1[\phi(\o t, \cdot), u(\o t, \cdot)](\o x)
  = \cI_1[\phi(\o t, \cdot), u(\o t, \cdot)](\o x) \ 4c | \o x - y |^3.
\end{equation}
Using Lemma \ref{l:Hp}, we get
\begin{align} \label{pf:I1:est}
\cI_1[\phi(\o t, \cdot), u(\o t, \cdot)](\o x) 
\leq c (12 |\o x - y|^2 + 2/3) + 4 \|u\|_\infty. 
\end{align}
Moreover, since $(u - \phi) (\o t, \o x) > 0$, we have that $|\o x - y| < (2 \|u\|_\infty/ c)^{1/4}$. Substituting these estimates in \eqref{pf:u:1VS:est} we obtain
\begin{equation*}
  \sigma \leq C(c, \|u\|_\infty)
\end{equation*}
for some $C(c, \|u\|_\infty) > 0$ which only depends on $c$ and $\|u\|_\infty$. Hence, if we choose $\sigma > C(c, \|u\|_\infty)$, we conclude that \eqref{pf:u:est} holds. In particular, choosing $y = x$,
\begin{equation} \label{pf:u:est:II}
  u(t,x) \leq u^\circ(x) + \sigma t + \alpha
  \quad \text{for any } t \geq 0, \: x \in \R.
\end{equation}
Analogously, we can find a similar bound as in \eqref{pf:u:est:II} for $v$ from below. This together with the uniform continuity of $u^\circ$ and the arbitrariness of $\alpha > 0$ allows us to deduce \eqref{no-init-layer}.
\smallskip

As in the proof of Theorem~\ref{t:CPe}, suppose that the comparison fails and therefore $\theta := \sup (u  - v) > 0$. Considering the function $\Phi$ in \eqref{comp-Phi} with related constants $\gamma, \eta, T > 0$, we can choose $\eta, T^{-1} > 0$ small enough so that $\Theta := \sup_{Q_T \times Q_T} \Phi \geq \theta /2$ uniformly in $\gamma$. By \eqref{no-init-layer} we can take $\gamma$ small enough so that the super-level set $\{ \Phi \geq \theta/4 \}$ is separated by a positive distance from $\{ (0,x,0,y) : x,y \in \R \}$. Due to the lack of compactness, a maximum of $\Phi$ is not necessarily achieved. We choose a maximizing sequence $(t_m, x_m, s_m, y_m)$. As $u - v$ is bounded, the sequences $(x_m - y_m)$, $(t_m)$ and $(s_m)$ are bounded, and thus by selecting a subsequence we can assume that $x_m - y_m \to \o w$, $t_m \to \o t$ and $s_m \to \o s$ as $m \to \infty$. Note that $\gamma$ is chosen so that $\o t, \o s > 0$. 

We consider the test functions
\begin{align*}
\phi_m(t,x) &:=  \frac{(t-s_m)^2}{2\gamma } + \frac{(x-y_m)^4}{4\gamma} + \frac\eta{T-t} + (x - y_m - \o w)^4 + (t - \o t)^2\\
\psi_m(s,y) &:= - \frac{(t_m-s)^2}{2\gamma } - \frac{(x_m-y)^4}{4\gamma } - \frac\eta{T-s} - (x_m - y - \o w)^4 - (s - \o s)^2.
\end{align*}

Let $(\tau_m, \xi_m)$ be a point of maximum of $u - \phi_m$ and $(\sigma_m, \eta_m)$ be a point of minimum of $v - \psi_m$. We have
\begin{align*}
u(t_m, x_m) - \phi_m(t_m, x_m) \leq u(\tau_m, \xi_m) - \phi_m(\tau_m, \xi_m),
\end{align*}
and subtracting $v(s_m, y_m) + \eta/(T-s_m)$ yields
\begin{multline}
\label{Phi-xi-tau}
\Phi(t_m, x_m, s_m, y_m) - (x_m - y_m- \o w)^4 - (t_m - \o t)^2 \\
\leq 
\Phi(\tau_m, \xi_m, s_m, y_m) - (\xi_m - y_m- \o w)^4 - (\tau_m - \o t)^2.
\end{multline}
Then, since $\Phi \leq \Theta$,
\begin{align*}
(\xi_m - y_m- \o w)^4 + (\tau_m - \o t)^2 \leq \Theta - \Phi(t_m, x_m, s_m, y_m) + (x_m - y_m- \o w)^4 + (t_m - \o t)^2.
\end{align*}
The right-hand side converges to 0 as $m\to\infty$ and therefore $\xi_m - y_m \to \o w$ and $\tau_m \to \o t$. Combining this with \eqref{Phi-xi-tau}, we deduce that $\Phi(\tau_m, \xi_m, s_m, y_m) \to \Theta$. A parallel argument for $v - \psi_m$ with $(\sigma_m, \eta_m)$ as a point of maximum yields $x_m - \eta_m \to \o w$ and $\sigma_m \to \o s$, and that $\Phi(t_m, x_m, \sigma_m, \eta_m) \to \Theta$.

We claim that also $\Phi(\tau_m, \xi_m, \sigma_m, \eta_m) \to \Theta$. Indeed, a bit of algebra shows that
\begin{align*}
\Phi(\tau_m, \xi_m, \sigma_m, \eta_m) &= \Phi(\tau_m, \xi_m, s_m, y_m) + \Phi(t_m, x_m, \sigma_m, \eta_m) - \Phi(t_m, x_m, s_m, y_m) \\
&\quad + \frac{(\xi_m - \eta_m)^4}{4\gamma} - \frac{(\xi_m - y_m)^4}{4\gamma} - \frac{(x_m - \eta_m)^4}{4\gamma} + \frac{(x_m - y_m)^4}{4\gamma} \\
&\quad + \frac{(\tau_m - \sigma_m)^2}{2\gamma} - \frac{(\tau_m - s_m)^2}{2\gamma} - \frac{(t_m - \sigma_m)^2}{2\gamma} + \frac{(t_m - s_m)^2}{2\gamma},
\end{align*}
where the right-hand side converges to $\Theta$ as $m \to \infty$.
Therefore, for any $\e > 0$ there exists an $m_0 > 0$ such that $\Phi(\tau_m, \xi_m, \sigma_m, \eta_m) > \Theta - \e$ for all $m \geq m_0$, and thus
\begin{align*}
\Phi(\tau_m, \xi_m + z, \sigma_m, \eta_m + z) \leq \Phi(\tau_m, \xi_m, \sigma_m, \eta_m) + \e \qquad \text{for all } m \geq m_0.
\end{align*}
This yields
\begin{align*}
u(\tau_m, \xi_m + z) - u(\tau_m, \xi_m) \leq v(\sigma_m, \eta_m + z) - v(\sigma_m, \eta_m) + \e,
\end{align*}
and thus, for any $\rho > 0$,
\begin{align}
\label{u-v-gap}
\int_{B_\rho^c} (u(\tau_m, \xi_m + z) - u(\tau_m, \xi_m))\frac{dz}{z^2} \leq \int_{B_\rho^c} (v(\sigma_m, \eta_m + z) - v(\sigma_m, \eta_m)) \frac{dz}{z^2} + \frac{2\e}{\rho}.
\end{align}

Regarding the first of the two integrals in \eqref{Irho}, we obtain from Lemma \ref{l:Hp}\eqref{l:Hp:Bp} that 
\begin{multline*} 
0 \leq \int_{B_\rho} \Big(\phi_m(\tau_m, \xi_m + z) - \phi_m(\tau_m, \xi_m) - (\phi_m)_x(\tau_m, \xi_m) z\Big) \frac{dz}{z^2} \\
= 12 \bigg[ \frac{ (\xi_m - y_m)^2 }{4 \gamma} + (\xi_m - y_m - \o w)^2 \bigg] \rho + \frac23 \bigg[ \frac{ 1 }{4 \gamma} + 1 \bigg] \rho^3.
\end{multline*}
Since $\xi_m - y_m \to \o w$, we obtain for $\rho$ and $m^{-1}$ small enough that 
\begin{equation}\label{brho-uniform} 
  \int_{B_\rho} \Big(\phi_m(\tau_m, \xi_m + z) - \phi_m(\tau_m, \xi_m) - (\phi_m)_x(\tau_m, \xi_m) z\Big) \frac{dz}{z^2} \leq \Big( 4 \frac{\o w^2}\gamma + o_m(1) + C_\gamma \rho^2 \Big)\rho,
\end{equation}
where the constant $C_\gamma > 0$ only depends on $\gamma$, and $o_m(1) \to 0$ as $m \to \infty$ uniformly in $\rho$ and $\e$.
The analogous estimate to \eqref{brho-uniform} for $\psi_m$ can be obtained in a similar fashion.
\smallskip

Let us first consider $\o w \neq 0$. Then, for $m$ large enough we have  $(\phi_m)_x(\tau_m, \xi_m) \neq 0 \neq(\psi_m)_x(\sigma_m, \eta_m)$, and
the definition of $\rho$-solutions yields
\begin{align} \label{pf:uvVS:diff}
(\phi_m)_t(\tau_m, \xi_m) - (\psi_m)_t (\sigma_m, \eta_m) 
\leq H_\rho[\phi_m, u](\tau_m, \xi_m) - H_\rho[\psi_m, v](\sigma_m, \eta_m).
\end{align} 
For the left-hand side, we compute
\begin{equation} \label{pf:uvVS:LHS:bd}
  (\phi_m)_t(\tau_m, \xi_m) - (\psi_m)_t (\sigma_m, \eta_m) 
  = \frac\eta{(T - \o t)^2} + \frac\eta{(T - \o s)^2} + o_m(1)
  \geq \frac{2\eta}{T^2} + o_m(1),
\end{equation}
where $o_m(1) \to 0$ as $m \to \infty$ uniformly in $\rho$ and $\e$. For the right-hand side in \eqref{pf:uvVS:diff}, we first compute $(\phi_m)_x(\tau_m, \xi_m) = \frac{\o w^3}\gamma + o_m(1)$.
Thus, for $m$ large enough, 
\begin{equation*}
  \big| (\phi_m)_x (\tau_m, \xi_m) \big| = \frac{|\o w|^3}\gamma + o_m(1).
\end{equation*}
Analogously, we find
\begin{equation*}
  \big| (\psi_m)_x (\sigma_m, \eta_m) \big| = \frac{|\o w|^3}\gamma + o_m(1).
\end{equation*}
Secondly, we focus on the two integrals in \eqref{Irho} for both $\cI_\rho[\phi_m, u]$ and $\cI_\rho[\psi_m, v]$. We estimate the integrals over $B_\rho$ simply by \eqref{brho-uniform}. The integrals over $B_\rho^c$ require more care; for the leading order terms in $| (\phi_m)_x (\tau_m, \xi_m) |$ and $| (\psi_m)_x (\sigma_m, \eta_m) |$ we use \eqref{u-v-gap}, and for the $o_m(1)$ parts we simply apply Lemma \ref{l:Hp}\eqref{l:Hp:Bpc}. Putting this all together, we estimate the right-hand side of \eqref{pf:uvVS:diff} as
\begin{align} \label{pf:uvVS:RHS:bd}
  H_\rho[\phi_m, u](\tau_m, \xi_m) - H_\rho[\psi_m, v](\sigma_m, \eta_m) 
  \leq 5 \frac{\o w^2}\gamma \rho \ 4 \frac{|\o w|^3}\gamma
        + \frac{2\e}{\rho} \ \frac{|\o w|^3}\gamma
        + \frac4\rho (\|u\|_\infty + \|v\|_\infty) o_m(1).
\end{align}

Next we reach a contradiction in \eqref{pf:uvVS:diff} by showing that $\rho$, $\e$ and $m$ can be chosen such that the right-hand side in \eqref{pf:uvVS:LHS:bd} is larger than the right-hand side in \eqref{pf:uvVS:RHS:bd}. First, we take $\rho > 0$ small enough so that the first term in \eqref{pf:uvVS:RHS:bd} is smaller than $\eta/(2 T^2)$. Secondly, we take $\e$ and $m^{-1}$ small enough such that also 
\begin{equation*}
  \frac{2\e}{\rho} \ \frac{|\o w|^3}\gamma
        + \frac4\rho (\|u\|_\infty + \|v\|_\infty) o_m(1)
  < \frac\eta{T^2}.
\end{equation*}
Third, we take $m$ even larger if necessary to have that \eqref{pf:uvVS:LHS:bd} is larger than $3\eta/(2 T^2)$. Then, the resulting estimates contradict with \eqref{pf:uvVS:diff}.
\smallskip

It is left to consider the case $\o w = 0$. If $\xi_m \neq y_m$ and $\eta_m \neq x_m$ for all $m$, then the estimates in the case $\o w \neq 0$ apply verbatim. In fact, the estimates can be simplified, because $\o w = 0$ implies that
\begin{equation*}
  \big| (\phi_m)_x (\tau_m, \xi_m) \big| + \big| (\psi_m)_x (\sigma_m, \eta_m) \big| = o_m(1),
\end{equation*}
and then it is enough to bound $H_\rho[\phi_m, u](\tau_m, \xi_m)$ and $H_\rho[\psi_m, v](\sigma_m, \eta_m)$ independently from each other (i.e.\ \eqref{u-v-gap} need not be used). Simplifying the estimates in this manner, we observe that if either $\xi_m = y_m$ or $\eta_m = x_m$ for some $m$, then $H_\rho[\phi_m, u](\tau_m, \xi_m) = 0$ or $H_\rho[\psi_m, v](\sigma_m, \eta_m) = 0$, which yields directly a sufficient bound on $H_\rho[\phi_m, u](\tau_m, \xi_m)$ or $H_\rho[\psi_m, v](\sigma_m, \eta_m)$. This yields again a contradiction. 
\end{proof}

Finally, we show in Corollary \ref{co:solutions-match-imr} that our notion of viscosity solution (Definition \ref{def:HJ:VS}) is equivalent to that given by \cite[Def.\ 1]{ImbertMonneauRouy08}. The notion in \cite[Def.\ 1]{ImbertMonneauRouy08} is obtained from Definition \ref{def:HJ:VS} by simply replacing \eqref{visc-subsol-cond} by $\phi_t(\o t, \o x) \leq H_\rho[\phi, u](\o t, \o x)$, and by performing a similar replacement for the supersolutions. Since \cite{ImbertMonneauRouy08,BilerKarchMonneau10} prove existence and regularity of \cite{ImbertMonneauRouy08}-solutions with $BUC$ initial data, Corollary \ref{co:solutions-match-imr} implies that these results then also apply to our notion of viscosity solution. 

\begin{cor}
\label{co:solutions-match-imr}
Let $u^\circ \in BUC(\R)$. Then $u$ is a viscosity solution with initial data $u^\circ$ in the sense of Definition~\ref{def:HJ:VS} if and only if it is a viscosity solution with initial data $u^\circ$ in the sense of \cite{ImbertMonneauRouy08}. In particular, \eqref{HJ-p} has a unique viscosity solution with initial data $u^\circ$.
\end{cor}

\begin{proof}
Since Definition~\ref{def:HJ:VS} only restricts the class of test functions when compared to \cite{ImbertMonneauRouy08}-solutions, it is clear that any viscosity solution with initial data $u^\circ$ in the \cite{ImbertMonneauRouy08}-sense is a viscosity solution in the sense of Definition~\ref{def:HJ:VS}.

On the other hand, suppose that $u$ is a viscosity solution with initial data $u^\circ$ in the sense of Definition~\ref{def:HJ:VS}. By \cite[Th.~4.7]{BilerKarchMonneau10} there exists a unique \cite{ImbertMonneauRouy08}-solution $v$ with initial data $u^\circ$. Then, by the first part of this proof, $v$ is also a viscosity solution in the sense of Definition~\ref{def:HJ:VS}. By the comparison principle, Theorem~\ref{t:CP:HJ}, we conclude that $u \equiv v$.
\end{proof}

\subsection{Convergence}

As usual, we define for an $\e$-indexed sequence $u_\e : Q \to \R$ of upper semi-continuous functions
\begin{equation*}
  {\limsup}^* \, u_\e (t,x) := \limsup_{ \substack{ \e \to 0 \\ s \to t \\ y \to x }  } u_\e (s,y).
\end{equation*}
Similarly, we use $\liminf_*$ for lower semi-continuous functions.

\begin{thm}[Convergence as $\e\to0$] \label{t:conv}
Let $u_\e$ be a sequence of subsolutions of \eqref{HJeps-p}, and assume that $u_\e$ is uniformly bounded. Set $\o u := \limsup^* u_\e$. Then $\o u$ is a subsolution of \eqref{HJ-p}.

Similarly, if $v_\e$ is a sequence of supersolutions of \eqref{HJeps-p} bounded uniformly, then $\underline v := \liminf_* v_\e$ is a supersolution of \eqref{HJ-p}.
\end{thm}

\begin{proof}
We only prove the subsolution case; the proof for supersolutions is analogous. By Lemmas \ref{l:HJe:rho} and \ref{l:HJ:rho}, we may set $\rho = 1$.

Let $\phi$ be a test function for $\o u$ such that $\o u - \phi$ has a strict maximum at $(\o t,\o x)$. By the usual argument, along a subsequence of $\e$, $u_\e-\phi$ has a global maximum at $(t_\e,x_\e)$ and we have $(t_\e,x_\e)\to (\o t,\o x)$ as $\e\to0$ and $u_\e(t_\e,x_\e) \to \o u(\o t,\o x)$. 
In what follows, $\e$ is taken along such subsequence. We therefore also have for any $z\in \R$
\begin{equation}
\label{ineq:th:conv}
\limsup_{\e\to0} u_\e(t_\e,x_\e+z) - u_\e(t_\e,x_\e) \leq \o u(\o t,\o x+ z) - \o u(\o t,\o x).
\end{equation}

We separate two cases; $\phi_x(\o t, \o x) \not= 0$ and $\phi_x(\o t, \o x) = 0$. When $\phi_x(\o t, \o x) \not= 0$, we need to prove that $\phi_t(\o t, \o x)\leq H_\rho[\phi,\o u](\o t, \o x)$. By Lemma~\ref{l:M-I} below, we have that for all $\e > 0$ small enough
\begin{multline*}
\pv \int_{B_\rho} E_\e^* \big[ \phi(t_\e,x_\e+z) - \phi(t_\e,x_\e)\big] \frac{dz}{z^2} \\
\leq \int_{B_\rho} \Big( \phi(t_\e,x_\e+z) - \phi(t_\e,x_\e) - z\phi_x(t_\e,x_\e)\Big) \frac{dz}{z^2} + R_\e,
\end{multline*}
where $R_\e\to 0$ as $\e\to0$. Then 
\begin{multline*}
  \limsup_{\e\to0} \o M_{\rho,\e}[\phi,u_\e] (t_\e,x_\e)
 \leq  \limsup_{\e\to0} \int_{B_\rho} \Big( \phi(t_\e,x_\e+z) - \phi(t_\e,x_\e) - z\phi_x(t_\e,x_\e)\Big) \frac{dz}{z^2} \\
 + \limsup_{\e\to0}\int_{B_\rho^c} E_\e^*\big[u_\e (t_\e,x_\e+ z) - u_\e (t_\e,x_\e)\big] \frac {dz}{z^2}.
\end{multline*}
Relying further on Fatou's lemma and \eqref{ineq:th:conv}, we obtain 
\begin{align*}
\limsup_{\e\to0} {}&\o M_{\rho,\e}[\phi,u_\e] (t_\e,x_\e)
 \\
&\leq \int_{B_\rho} \Big( \phi(\o t,\o x+z) - \phi(\o t,\o x) - z\phi_x(\o t,\o x)\Big) \frac{dz}{z^2}\\
&\qquad + \limsup_{\e\to0}  \frac\e2 \int_{B_\rho^c} \frac{dz}{z^2} 
\;+ \;\limsup_{\e\to0} \int_{B_\rho^c} \Big(u(t_\e,x_\e+ z) - u(t_\e,x_\e)\Big)\frac{dz}{z^2}\\
&\leq \int_{B_\rho} \Big( \phi(\o t,\o x+z) - \phi(\o t,\o x) - z\phi_x(\o t,\o x)\Big) \frac{dz}{z^2} +\int_{B_\rho^c} \Big(\o u(\o t, \o x+ z) - \o u(\o t, \o x)\Big)\frac{dz}{z^2}\\
&= \cI_\rho[\phi,\o u](\o t, \o x).
\end{align*} 
Combining this with the convergence $|\phi_x(t_\e,x_\e)|\to |\phi_x(\o t, \o x)|$, we find
\begin{align*}
\phi_t(\o t,\o x) &= \limsup_{\e\to0} \phi_t(t_\e,x_\e)
\leq \limsup_{\e\to0} \o M_{\rho,\e}[\phi,u_\e] (t_\e,x_\e)\; |\phi_x(t_\e,x_\e)|\\
&\leq \cI_\rho[\phi,\o u](\o t, \o x)\; |\phi_x(\o t, \o x)| = H_\rho[\phi,\o u](\o t, \o x),
\end{align*}
which concludes the proof  that $\o u$ satisfies the subsolution condition for \eqref{HJ-p} in the case $\phi_x(\o t, \o x)\not=0$. 

\medskip

When $\phi_x(\o t, \o x)=0$, we only need to consider the case in which $\phi(t, x) = c |x - \o x|^4 + g(t)$, and prove that $\phi_t(\o t, \o x) \leq 0$. By the usual approximation argument, we may assume that $c > 0$ and that $g(t)$ tends to $+\infty$ as $t \to +\infty$. If $x_\e = \o x$ for some $\e$, then from $u_\e$ being a subsolution to \eqref{HJeps-p} we obtain immediately $\phi_t(t_\e, \o x) \leq 0$. If $x_\e \neq \o x$, we obtain that $\phi_t(t_\e, x_\e) \leq \o H_{\rho, \e}[\phi, u_\e](t_\e, x_\e)$. To estimate $\o H_{\rho, \e}[\phi, u_\e](t_\e, x_\e)$ from above, let $\e$ be sufficiently small so that $|x_\e - \o x| \leq 1$. Then, by Taylor's Theorem,
$$
 \phi(t_\e, x_\e +z) \leq \phi(t_\e, x_\e) + \phi_x(t_\e, x_\e) z + 24c z^2
 \quad \text{for } |z| \leq 1,
$$ 
from which we obtain by Lemma~\ref{l:parabola} below that
\[
\o H_{\rho, \e}[\phi, u_\e](t_\e, x_\e) 
\leq C \|u\|_\infty |\phi_x(t_\e, x_\e)| +  48c |\phi_x(t_\e, x_\e)| + C\Big(48c \e + |\phi_x(t_\e, x_\e)|^2\Big).
\]
Taking $\e \to 0$, we observe from $\phi_x(t_\e, x_\e) \to 0$ that $\phi_t(\o t, \o x) \leq 0$.
\end{proof}

\begin{lem}
\label{l:M-I}
Let $\rho > 0$ be given. Let $\phi\in C^2(Q)$ and $(t_\e,x_\e)\to (\o t, \o x)$ such that $\phi_x(\o t, \o x)\not=0$.
Then there exists a $C > 0$ such that for all $\e > 0$ small enough
\begin{multline*}
 \pv \int_{B_\rho} E_\e^*\big[ \phi(t_\e, x_\e+z) - \phi(t_\e,  x_\e)\big] \frac{dz}{z^2} \\
\leq \int_{B_\rho} \Big(\phi(t_\e, x_\e+z) - \phi(t_\e,  x_\e) -z \phi_x(t_\e,x_\e)\Big)\frac{dz}{z^2} + C\e |\log \e|. 
\end{multline*}
\end{lem}

\begin{proof}
For convenience, we assume $\phi_x(\o t, \o x) > 0$. We set $\phi_\e (z) := \phi(t_\e, x_\e+z) - \phi(t_\e,  x_\e)$, and start by listing several properties of $\phi_\e$. Since
\[ 
  \phi_\e(0) = 0, \quad
  \phi_\e \xto{\e \to 0} \phi(\o t, \o x + \, \cdot \, ) - \phi(\o t,  \o x) \text{ in } C_{\loc}^2(Q),
\]
we may assume that $\| \phi_\e \|_{C^2([\o t /2, 2 \o t] \times \o B_\rho)}$ is bounded uniformly in $\e$. Moreover,
\[
  \exists \, \delta \in (0, \rho] \ 
  \forall \, z \in B_\delta : 
  \phi_\e'(z) \in \Big( \frac12 \phi_x(\o t, \o x), 2 \phi_x(\o t, \o x) \Big)
\]
for all $\e$ small enough, where we will choose $\delta$ small enough with respect to $\phi$ later on in the proof. In particular, $\phi_\e$ is strictly increasing on $B_\delta$; hence, for any $y$ in its range,
\begin{align*}
  \phi_\e^{-1}(y) 
  = \phi_\e^{-1}(0) + y (\phi_\e^{-1})'(0) + R_1(y)
  = \frac y{\phi_\e'(0)} + R_1(y),
\end{align*}
where $|R_1(y)| \leq C y^2$ for some $C > 0$ independent of $\e$. We further set
\[ 
  r := -\phi_\e^{-1}(-\e) \quad \text{and} \quad
  \tilde r := \phi_\e^{-1}(\e),
\]
note that
\[
  r, \tilde r \in \Big( \frac\e {2 \phi_x(\o t, \o x)}, \frac{2 \e}{\phi_x(\o t, \o x)} \Big),
\]
and assume for convenience that $\tilde r \geq r$. Similarly, we set
\[ 
  s := -\phi_\e^{-1}(-\delta), \quad 
  \tilde s := \phi_\e^{-1}(\delta), \quad
  s, \tilde s \in \Big( \frac\delta {2 \phi_x(\o t, \o x)}, \frac{2 \delta}{\phi_x(\o t, \o x)} \Big).
\]
We assume $\e < \frac14 \delta$ such that
\[
  \max \{ r, \tilde r \} \leq \min \{ s, \tilde s \}.
\]
Finally, we will use the following Taylor expansion on $B_\delta$:
\begin{equation*}
  \phi_\e'(z) = \phi_\e'(0) + R_2(z), 
\end{equation*}
where $|R_2(z)| \leq C |z|$ for some $C > 0$ independent of $\e$.

Using $\phi_\e$, the assertion of Lemma \ref{l:M-I} reads
\[
\pv \int_{B_\rho} F_\e [ \phi_\e(z) ] \frac{dz}{z^2}
\leq C\e |\log \e|,
\]
where $F_\e(u) := E_\e^*[u] - u$ is an odd, $\e$-periodic function. We prove this estimate by splitting the domain of integration in the following four parts:
\begin{equation*}
  (-r,r), \quad 
  (r, \tilde r), \quad
  (-s, \tilde s) \setminus (-r, \tilde r), \quad 
  B_\rho \setminus (-s, \tilde s).
\end{equation*}
On $(-r,r)$, we note that $\phi_\e$ increases from $-\e$ to a value less than $\e$, and thus we find
\begin{align*}
  \pv \int_{-r}^r F_\e [ \phi_\e(z) ] \frac{dz}{z^2}
  = -\int_0^r \frac{\phi_\e(z) + \phi_\e(-z) }{z^2} dz
  \leq \int_0^r \frac{ z^2 \| \phi_\e'' \|_\infty }{z^2} dz
  \leq C r \leq C \e.
\end{align*}
On $(r, \tilde r)$, we simply estimate $|F_\e(u)| \leq \frac\e2$, and use the properties of $r, \tilde r$ to find
\begin{align*}
  \int_r^{\tilde r} F_\e [ \phi_\e(z) ] \frac{dz}{z^2}
  \leq \frac\e2 \frac{\tilde r - r}{r^2}
  = \frac\e2 \frac{\phi_\e^{-1}(\e) + \phi_\e^{-1}(-\e)}{r^2}
  \leq \frac\e2 (2 |R_1(\e)|) \Big( \frac{2 \phi_x(\o t, \o x)}{\e} \Big)^2
  \leq C \e.
\end{align*}
On $(-s, \tilde s) \setminus (-r, \tilde r) = \phi_\e^{-1} (B_\delta \setminus B_\e)$, we change variables, and compute
\begin{multline} \label{pf:u}
  \int_{(-s, \tilde s) \setminus (-r, \tilde r)} F_\e [ \phi_\e(z) ] \frac{dz}{z^2}
  = \int_{B_\delta \setminus B_\e} F_\e [ u ] \frac{du}{\phi_\e^{-1}(u)^2 \, \phi_\e'(\phi_\e^{-1}(u))}  \\
  = \phi_\e'(0) \int_{B_\delta \setminus B_\e} \frac{ F_\e [ u ] }{u^2} \frac{du}{ \big(1 + \phi_\e'(0) R_1(u)/u \big)^2 \, \big(1 + R_2(\phi_\e^{-1}(u))/\phi_\e'(0) \big)}.
\end{multline}
Since $F_\e [ u ]/u^2$ is odd in $u$ and the domain $B_\delta \setminus B_\e$ is symmetric, this terms cancels out with the constant contribution from the second fraction. Writing the denominator of this second fraction as $(1 + R_3(u))^2 (1 + R_4(u))$, we note that $|R_3(u)| + |R_4(u)| \leq C|u|$ for some $C > 0$ which only depends on $\phi$. Hence, choosing $\delta < 1/(2C)$, we obtain some $C' > 0$ such that
\begin{equation*}
  \int_{(-s, \tilde s) \setminus (-r, \tilde r)} F_\e [ \phi_\e(z) ] \frac{dz}{z^2}
  \leq C' \int_{B_\delta \setminus B_\e} \frac{ |F_\e [ u ]| }{u^2} |u| \, du
  \leq C' \frac\e2 \int_{B_\delta \setminus B_\e} \frac{ du }{|u|}
  \leq C'' \e |\log \e|.
\end{equation*}
Finally, on the fourth part given by $B_\rho \setminus (-s, \tilde s)$, we simply bound
\begin{equation*}
  \int_{B_\rho \setminus (-s, \tilde s)} F_\e [ \phi_\e(z) ] \frac{dz}{z^2}
  \leq \frac\e2 2 \int_{\tfrac\delta {2 \phi_x(\o t, \o x)}}^\infty \frac{dz}{z^2}
  \leq C \e.
\end{equation*}

\end{proof}

\begin{lem}
\label{l:parabola}
For $\alpha \neq 0$, $K > 0$, $\e > 0$, we have
\begin{align}
\label{parabola-estimate}
|\alpha|\left( \pv \int_{B_1} E_\e^*[\alpha z + \frac K2 z^2] \frac{dz}{z^2}\right) \leq K|\alpha| + C (K\e + \alpha^2),
\end{align}
where $C$ is independent of $\alpha$, $K$ and $\e$. Moreover, \eqref{parabola-estimate} still holds if we replace $E_\e^*$ by $E_{\e,*}$.
\end{lem}

\begin{proof}
We write $\phi(z) = \alpha z + \frac K2 z^2$.
By changing variables from $z $ to $ -z$ if necessary we can assume that $\alpha > 0$.
Let $(a, b) \ni 0$ be the largest interval containing the origin such that 
\begin{align*}
-\e &<  \phi(z) < 0, & a &< z < 0,\\
0 &< \phi(z) < \e, & 0 &< z < b.\\
\end{align*}
Since $\alpha > 0$, such $a < 0 < b$ exist, $\phi(b) = \e$, and $\phi(a)$ is equal to either $-\e$ or $0$. If $\phi(a) = -\e$, then $a < -b$, and if $\phi(a) = 0$, then $a = -\frac{2\alpha}K$. Inverting $\phi(b) = \e$, we find
\begin{align*}
b = \frac \alpha K \left(\sqrt{1 + \frac{2\e K}{\alpha^2}} - 1\right).
\end{align*}
Employing the elementary inequality
\[
  \sqrt{1 + x} \geq \left\{ \begin{aligned}
    &1 + \big( \sqrt 2 - 1 \big)x
    && \text{ if } 0 < x < 1 \\
    &\sqrt 2
    && \text{ if } 1 < x,
  \end{aligned} \right.
\]
we find that
\begin{align*}
b 
\geq \big( \sqrt 2 - 1 \big) \left\{ \begin{aligned}
    &\frac{2 \e}\alpha
    && \text{ if } \frac{2\e K}{\alpha^2} < 1 \\
    &\frac\alpha{K}
    && \text{ if } 1 < \frac{2\e K}{\alpha^2}
  \end{aligned} \right\} 
\geq \big( \sqrt 2 - 1 \big) \min\left\{\frac{2 \e}\alpha, \frac\alpha{K}\right\}
=: \eta.
\end{align*}
Together with the estimates on $a$, we obtain $(a, b) \supset (-\eta, \eta)$. Finally, we compute
\begin{align*}
\pv\int_{B_1}E_\e^*[\phi(z)] \frac{dz}{z^2} &= 
\int_{B_1 \setminus B_\eta}E_\e^*[\phi(z)] \frac{dz}{z^2} \leq
\int_{B_1 \setminus B_\eta}\left(\alpha z + \frac K2z^2 + \frac\e 2\right) \frac{dz}{z^2} \\
&\leq 0 + K + \frac\e\eta  = K + C \max\left\{\frac {K\e}{\alpha}, \alpha\right\}.
\end{align*}
This concludes the proof.
\end{proof}

\section{Passing to the limit $n \to \infty$}
\label{s:Pn:to:HJ}

Given $(\bx, \bb) \in \cZ_n$, we recall from \eqref{un} the related piecewise constant function given by
\begin{equation} \label{un:repd} 
  u_n(x) = \frac1n \sum_{i=1}^n b_i H(x - x_i).
\end{equation}

\begin{thm}[Main theorem] \label{t}
For each $n \geq 2$, let $(\bx^\circ, \bb^\circ)_n \in \cZ_n$ be such that the related function $u_n^\circ$ defined by \eqref{un:repd} satisfies $u_n^\circ \to u^\circ$ uniformly as $n \to \infty$ for some $u^\circ \in BUC(\R)$. Then, the solution $(\bx, \bb)_n$ to $(P_n)$ (Definition \ref{defn:Pn}) with initial datum $(\bx^\circ, \bb^\circ)_n$, translated to $u_n$, converges locally uniformly as $n \to \infty$ to the viscosity solution $u$ (see Definition \ref{def:HJ:VS}) with initial datum $u^\circ$. Furthermore, $u \in BUC([0, \infty) \times \R)$ and $u(t, \cdot)$ has the same modulus of continuity as $u^\circ$ for every $t > 0$.
\end{thm}

\begin{proof}
Lemma \ref{l:Pn:to:HJe} below shows that $u_n$ is a viscosity solution of \eqref{HJeps-p} with $\e = \frac1n$. Then, since $\| u_n \|_\infty \leq 1$, we obtain from Theorem \ref{t:conv} that 
\begin{equation} \label{pf:t}
  \underline u := {\liminf}_* u_n \leq {\limsup}^* u_n =: \o u
  \quad \text{on } Q,
\end{equation}
where $\underline u$, $\o u$ are respectively super- and subsolutions of \eqref{HJ-p}.

Since $u^\circ \in BUC(\R)$, there exists a sequence $(u^\circ_\delta)_\delta \subset C_b^2$ such that $u^\circ_\delta \geq u^\circ + \delta$ and $u^\circ_\delta \to u^\circ$ for all $\delta > 0$. An example of such a sequence is a mollification of $u^\circ + 2\delta$ with sufficiently small, $\delta$-dependent radius. By Lemma~\ref{l:barrier} below, there exists $\sigma_\delta > 0$ such that $u^\circ_\delta(x) + \sigma_\delta t$ is a super-solution for \eqref{HJeps-p}. Then, the comparison principle (Theorem \ref{t:CPe}) implies that $u_n^* \leq u^\circ_\delta(x) + \sigma_\delta t$ for all $n$ large enough, and therefore $\o u(0, \cdot) \leq u^\circ_\delta$. Using a similar argument for approximation from below, in the limit $\delta \to 0$ we recover
\begin{equation} \label{pf:t:2}
  u^\circ \leq \underline u (0, \cdot) \leq \o u (0, \cdot) \leq u^\circ.
\end{equation}
Then, the comparison principle (Theorem \ref{t:CP:HJ}) yields $\o u \leq \underline u$. Hence, the inequality in \eqref{pf:t} has to be an equality, and thus $u_n \to u$ locally uniformly.

It remains to show the boundedness and uniform continuity of $u$. As constants are solutions of \eqref{HJ-p}, we immediately have $\|u\|_\infty = \|u^\circ\|_\infty$. Let $\omega$ be a modulus of continuity of $u^\circ$. Since $(t,x) \mapsto u(t, x + y) \pm \omega(|y|)$ are viscosity solutions of \eqref{HJ-p} for any fixed $y \in \R$, it follows from the comparison principle that $u(t, \cdot)$ is uniformly continuous with the same modulus of continuity $\omega$. 

The uniform continuity in time can be established by the comparison with functions such as $(u(s, \cdot) * \eta_{\omega^{-1}(\delta)})(x) \pm (\delta + \sigma(t - s))$ for $s \geq 0$, $\delta > 0$, where $\eta_r$ is the standard mollifier with support radius $r$ and $\sigma = \sigma(\delta, \|u^\circ\|_\infty)$ is a sufficiently large constant. 
\end{proof}

\begin{lem}
\label{l:Pn:to:HJe}
Consider the setting of Theorem \ref{t}. Then for any $n \geq 2$, $u_n$ is a viscosity solution to \eqref{HJeps-p} with $\e = \frac1n$.
\end{lem}

\begin{proof}
Let us fix $n \geq 2$ and write $u = u_n^*$ to simplify the notation. To prove that $u$ is a subsolution, we check that it satisfies Definition \ref{def:HJe:VS} for any $\rho > 0$. With this aim, let $\phi\in C^2(Q)$ be such that $u-\phi$ has a $0$ global maximum at $(\o t, \o x)$.  

If $\o x \notin \{x_i(\o t): 1 \leq i \leq n, \, |b_i(\o t-)| = 1 \}$, then $u$ is constant in some neighborhood of $(\o t, \o x)$, and thus 
\[ \phi_t(\o t, \o x) 
   = \phi_x(\o t, \o x) 
   = 0
   \quad \text{for any } \rho > 0, \]
which is consistent with Definition \ref{def:HJe:VS}. In the remainder we may thus assume that there exists an index $i$ such that $\o x = x_i(\o t)$ and $|b_i(\o t-)| = 1$. 

We separate three cases. The first case is $\o t \notin S$, where $S$ is the set of collision times from Definition \ref{defn:Pn}. Then, the curve $t \mapsto (t, x_i(t))$ is of class $C^1$ in a neighborhood around $(\o t, \o x)$, and $u$ is constant along this curve. Therefore, 
\begin{equation} \label{pfzz}
  0 = \frac d{dt} \Big|_{t = \o t} \phi (t, x_i(t))
  = \phi_t(\o t, \o x) + \phi_x(\o t, \o x) \dot x_i (\o t).
\end{equation}
If $\phi_x(\o t, \o x) = 0$, then it follows that $\phi_t(\o t, \o x) = 0$, which is consistent with Definition \ref{def:HJe:VS}. Therefore we may assume $\phi_x(\o t, \o x) \neq 0$, and it is left to check that
\begin{equation} \label{pf:HJe:cond}
  \phi_t(\o t, \o x) 
   \leq \o H_{\rho,\e}[\phi, u](\o t, \o x) \qquad \text{for all } \rho > 0.
\end{equation}
With this aim, we compute (omitting $\o t$, recalling \eqref{force:ito:uz} and using that  $u(x_i) = \sum_{j=1}^{i-1} b_j + \frac12(b_i+1)$) 
\begin{align} \notag  
  n \o M_{0,\e}[u](x_i) 
  &:= n \liminf_{\rho \to 0} \o M_{\rho,\e}[\phi,u](x_i) \\\notag
  &= n \, \pv \int_{\R} E_\e^*\big[ u(x_i + z)-u(x_i)\big] \frac{dz}{z^2} \\\notag
  &= \pv \int_{\R} \Big( \sum_{j=1}^n b_j H(x_i - x_j + z) -\sum_{j=1}^{i-1} b_j - \frac{b_i+1}2 + \frac12 \Big) \frac{dz}{z^2} \\\notag
  &= \sum_{j=1}^{i-1} b_j \int_{\R} \big(  H(x_i - x_j + z) - 1 \big) \frac{dz}{z^2}
     + b_i \, \pv \int_{\R} \Big(  H(z) - \frac12 \Big) \frac{dz}{z^2} \\\notag
  & \qquad + \sum_{j=i+1}^n b_j \int_{\R} H(x_i - x_j + z) \frac{dz}{z^2} \\\label{pf:Me:RHS}
  &= \sum_{j=1}^{i-1} \frac{b_j}{x_j - x_i}
     + 0
     + \sum_{j=i+1}^n \frac{b_j}{x_j - x_i}
  = -\sum_{j \neq i}^n \frac{b_j}{x_i - x_j},
\end{align} 
and observe from \eqref{Pn} that $\dot x_i = -b_i \o M_{0,\e}[u](x_i)$ at time $\o t$. Moreover, from the discontinuity of $x \mapsto u(\o t, x)$ at $x_i(\o t)$ we infer that $b_i = \sign \phi_x(\o t, \o x)$. Then, we obtain from \eqref{pfzz} that
\[ \phi_t(\o t, \o x) 
   = -  \phi_x(\o t, \o x) \dot x_i (\o t)
   = \o M_{0,\e}[u](\o t, \o x) \big| \phi_x(\o t, \o x) \big|
   \leq \o M_{\rho,\e}[\phi, u](\o t, \o x) \big| \phi_x(\o t, \o x) \big|
   \quad \text{for all } \rho > 0, \]
which proves \eqref{pf:HJe:cond}.

The second case is $\o t \in S$ and $(\o t, \o x)$ is not an annihilation point. Then, by the continuity of $\bx$ and $(\bx, \bb) \in \cZ_n$, we observe that the right-hand side of the ODE for $x_i(t)$ is continuous as a function of $t$ at $\o t$. Hence, $x_i$ is differentiable, and the argument above applies.

The final, third case is when $(\o t, \o x)$ is an annihilation point. 
Note from Theorem \ref{t:Pn}\eqref{t:Pn:LB:col} that $\phi_x(\o t, \o x) = 0$, because otherwise $\phi \geq u$ fails in any neighborhood of $(\o t, \o x)$. Then, by Definition \ref{def:HJe:VS} we may assume that $\phi(t, x) = c(x - \o x)^4 + g(t)$, and we are left to prove that $\phi_t(\o t, \o x) \leq 0$. By Corollary \ref{c:Pn}, $u$ attains only two distinct values in any small enough neighborhood of $(\o t, \o x)$. In particular, since $u$ is upper semicontinuous, this implies that $u(t, x_i(t)) = u(\o t, \o x)$ for all $t < \o t$ close enough to $\o t$. Hence, $\phi(\bar t, \bar x) \leq \phi(t, x_i(t))$. Using the special form of $\phi$ and Theorem 2.4(iv), this yields 
$$
  g(\o t) 
  = \phi(\bar t, \bar x) 
  \leq \phi(t, x_i(t)) 
  = c(x_i(t) - \bar x)^4 + g(t) 
  \leq C(t - \bar t)^2 + g(t) $$ 
for all $t \leq \bar t$ close enough to $\o t$. Hence, $\phi_t(\bar t, \bar x) = g'(\bar t) \leq 0$.

This concludes the proof that $u = u_n^*$ is a subsolution of \eqref{HJeps-p}. The proof that $(u_n)_*$ is a supersolution is analogous.
\end{proof}

The following lemma is a consequence of the fact that rising parabolas are supersolutions of \eqref{HJeps-p} with $\e$-independent speed.

\begin{lem}
\label{l:barrier}
Let $v_0 : \R \to \R$ be bounded, $L$-Lipschitz and $K$-semiconcave, that is, $x \mapsto v_0(x) - \frac K2 x^2$ is concave. Then there exists $\sigma = \sigma(L, K, \|v_0\|_\infty) > 0$ such that $v(t,x) = v_0(x) + \sigma t$ is a supersolution of \eqref{HJeps-p} for all $\e \in (0,1)$. An analogous result holds for a semiconvex $u_0$, in which case $u(t,x)= u_0(x) - \sigma t$ is a subsolution.
\end{lem}

\begin{proof}  
For a certain $\sigma > 0$ which we specify later, we check that $v$ satisfies Definition \ref{def:HJe:VS} with $\rho = 1$. Suppose $v - \psi$ has a global minimum at $(\o t, \o x)$. Note that, by semiconcavity,
\begin{gather} \label{pf:sig:phit}
  0 
  < \sigma 
  = v_t (\o t, \o x)
  = \psi_t (\o t, \o x), \\\notag
  \alpha 
  := v_x (\o t, \o x)
  = \psi_x (\o t, \o x).
\end{gather}
In particular, if $\psi_x(\o t, \o x) = 0$, then \eqref{pf:sig:phit} is consistent with Definition \ref{def:HJe:VS}. Hence, we may assume that $\alpha \neq 0$. Note that $|\alpha| \leq L$, and set $\phi(z) := \psi(\o t, \o x + z) - \psi(\o t, \o x)$. By semiconcavity,
\begin{align*}  
\phi(z) 
\leq v(\o t, \o x + z) - v(\o t, \o x) 
= v_0(\o x + z) - v_0(\o x) 
\leq \alpha z + \frac K2 z^2,
\end{align*}
and thus
\begin{align*}
\pv \int_{B_\rho} E_{\e, *}[\phi(z)] \frac{dz}{z^2} 
\leq \pv \int_{B_\rho} E_{\e, *}[\alpha z + \frac K2z^2] \frac{dz}{z^2}.
\end{align*}
Then, using Lemma~\ref{l:parabola} and $\e \leq 1$, $|\alpha| \leq L$, we obtain
\begin{align*}
  \underline H_{\rho, \e}[\psi, v](\o t, \o x)
  = \underline H_{\rho, \e}[\psi, v_0](\o t, \o x)
  \leq K L + C(K + L^2) + 4 \| v_0 \|_\infty L + 1 
  =: \sigma
\end{align*}
for some $C$ independent of the parameters. This together with \eqref{pf:sig:phit} shows that $v$ is a supersolution.
\end{proof}

We finally show that the viscosity solution $u_n$ in Theorem \ref{t} constructed from \eqref{Pn} is the unique viscosity solution of \eqref{HJeps-p}. This will establish the correspondence between \eqref{Pn} and \eqref{HJeps-p}.  This is not obvious from the comparison principle (Theorem \ref{t:CPe}), because the initial condition is not continuous.

We start with a general theorem on uniqueness of viscosity solution with possibly discontinuous initial datum:

\begin{prop}
\label{p:visc-ode-uniq}
Consider \eqref{HJeps-p} with $\e > 0$. Let $u$ be a viscosity solution such that
\begin{align*}
(u^*)_* = u_*, \qquad (u_*)^* = u^* \qquad \text{on }Q.
\end{align*}
If there exists a sequence $(u^\eta)_{\eta \in (-1,1)}$ of viscosity solutions such that $(u^{-\eta})^*(0, \cdot) \leq u_*(0, \cdot)$ and $u^*(0, \cdot) \leq (u^{\eta})_*(0, \cdot)$ for all $\eta \in (0,1)$, and
\begin{align*}
(u^\eta)^* &\nearrow u_* \qquad \text{on }Q \qquad \text{as } \eta \nearrow 0\\
(u^\eta)_* &\searrow u^* \qquad \text{on }Q \qquad \text{as } \eta \searrow 0,
\end{align*}
then any viscosity solution $v$ with 
\begin{align*}
v^*(0, \cdot) = u^*(0, \cdot),\qquad
v_*(0, \cdot) = u_*(0, \cdot)
\end{align*}
satisfies $v^* = u^*$ and $v_* = u_*$ on $Q$.
\end{prop}

\begin{proof}
Since at $t = 0$ it is given that $(u^{-\eta})^* \leq u_* \leq u^* \leq (u^{\eta})_*$ on $\R$ for any $\eta \in (0,1)$, the comparison principle yields
\begin{align*}
(u^{-\eta})^* \leq v_* \leq v^* \leq (u^\eta)_* \qquad \text{on } Q.
\end{align*}
Sending $\eta \to 0$, we recover
\begin{align*}
u_* \leq v_* \leq v^* \leq u^* \qquad \text{on }Q.
\end{align*}
By definition of the upper semi-continuous envelope,
\begin{align*}
u^* = (u_*)^* \leq v^* \leq u^* \qquad \text{on } Q,
\end{align*}
which yields $v^* = u^*$ on $Q$.

We can similarly show $v_* = u_*$ using the lower semi-continuous envelope.
\end{proof}

\begin{prop}
\label{p:Pn:to:HJe:un}
Consider the setting of Theorem \ref{t}. Then the viscosity solution $v$ to \eqref{HJeps-p} with $\e = \frac1n$ which satisfies $v^*(0, \cdot) = (u_n^\circ)^*$ and $v_*(0, \cdot) = (u_n^\circ)_*$ is unique.
\end{prop}

\begin{proof}
Lemma \ref{l:Pn:to:HJe} implies the existence of a viscosity solution $u_n$ constructed from the ODE solution $(\bx, \bb)$. To prove uniqueness, we show that $u_n$ satisfies the conditions of Proposition \ref{p:visc-ode-uniq}.

We start with showing that $(u_n^*)_* = u_{n,*}$ and $(u_{n,*})^* = u_n^*$ on $Q$. This is obvious when $(t,x) \in Q$ is not an annihilation point. If $(\tau, y) \in Q$ is an annihilation point, then Corollary \ref{c:Pn} implies that $u_n |_{ B_\delta (\tau) \times B_\delta (y) }$ takes at most two values in $\e \Z$ for any $\delta > 0$ small enough. If the number of colliding particles at $(\tau, y)$ is odd, then these two values are reached by 
\begin{equation*}
  \lim_{x \searrow y} u_n(\tau, x)
  \quad \text{and} \quad
  \lim_{x \nearrow y} u_n(\tau, x).
\end{equation*}
If instead the number of colliding particles is even, then these two values are reached by 
\begin{equation*}
  \lim_{t \searrow \tau} u_n(t,y)
  \quad \text{and} \quad
  \lim_{t \nearrow \tau} u_n \Big(t, \frac{x_i(t) + x_j(t)}2 \Big),
\end{equation*}
where $i, j$ are the first two elements of the ordered set $I$ of indices of all colliding particles at $(\tau, y)$. 

Next we construct $u_n^\eta$. Consider the initial data
\begin{align*}
x_i^{\eta\circ} = x_i^\circ - b_i^\circ \eta
\quad \text{and} \quad
b_i^{\eta\circ} = b_i^\circ.
\end{align*}
For 
$$
  |\eta| < \min \{ |x_i^\circ - x_j^\circ| : i,j \text{ such that } b_i^\circ b_j^\circ = -1 \} =: \delta
$$ 
we have $(\bx^{\eta\circ}, \bb^{\eta\circ}) \in \cZ_n$. Let $(\bx^{\eta}, \bb^{\eta})$ be the corresponding solution of \eqref{Pn}. By Lemma \ref{l:Pn:to:HJe}, the corresponding step function $u_n^\eta$ is a viscosity solution. Note that $u_n^0 = u_n$. Since $\bx^{\eta} \in C([0,T])$ we have by the choice of $\bx^{\eta\circ}$ that
\begin{align*}
(u_n^{\eta})^*(0, \cdot) 
= (u_n^{\eta\circ})^*
\leq (u_n^{\zeta\circ})_*
= (u_n^{\zeta})_*(0, \cdot) \qquad \text{for } \eta <\zeta.
\end{align*}
Then, by the comparison principle,
\begin{align*}
(u_n^{\eta})^*
\leq (u_n^{\zeta})_* \qquad \text{for } \eta <\zeta.
\end{align*}
With this ordering of the sequence $(u_n^{\eta})_\eta$, it is left to show that $u_n^{\eta} \to u_n$ pointwise in $Q$ as $\eta \to 0$. This convergence statement is a direct consequence of Theorem \ref{t:Pn}\eqref{t:Pn:stab}.
\end{proof}

\section{Interpretation in terms of measure language}
\label{s:kappa}

In this section we re-interpret the convergence result of Theorem~\ref{t} in terms of the empirical measures~$\kappa_n$ given by~\eqref{def:kappa_n}.
This connection builds upon a more general equivalence between different types of convergence. For the case of functions and measures on $\R$ (i.e.\ without time) the following lemma describes these.

First we recall some definitions and relations between them. A function $u : \R \to \R$ has \emph{finite variation} on $\R$ if its \emph{total variation}
\[
  TV(u) := \sup_{N \geq 1} \sup_{y_0 < y_1 < \ldots < y_N} \sum_{i=1}^N \big| u(y_i) - u(y_{i-1}) \big|
\]
is finite. A right-continuous function $u$ with finite variation on $\R$ is \emph{c\`adl\`ag}, i.e.\ right-continuous with left limits, and  differentiable in the distributional sense. The distributional derivative of~$u$ can be characterized as a signed measure $\kappa$ on $\R$ whose total-variation norm $|\kappa|(\R)$ equals $TV(u)$. Thus the concepts of total variation for functions and for measures are coupled, in  that the total variation of a function is the same as the (measure-)total variation of its distributional derivative. 
Finally, a sequence of measures $\{\kappa_n\}_n$ on $\R$ is \emph{tight} if for any $\e > 0$ there exists $R > 0$ such that $|\kappa_n|(\R \setminus [-R, R]) < \e$ for all $n$.

\begin{lem}
\label{l:conv-equiv-R}
	Let $u_n$ and $u$ be c\`adl\`ag functions of finite variation on $\R$, and let the signed measures $\kappa_n$ and $\kappa$ be the corresponding distributional derivatives.  
Assume that the sequence $\{\kappa_n\}_n$ is bounded in total variation and tight.	

	Then the following are equivalent:
\begin{enumerate}
\item \label{l:conv-equiv-R:ct}
$u_n\to u$ continuously, i.e.
\begin{equation}
\label{eqdef:ct-conv-R}
\text{for all $x_n\to x$,} \qquad u_n(x_n)\to u(x);
\end{equation}
\item \label{l:conv-equiv-R:locunif}
$u_n\to u $ locally uniformly on $\R$, 
and the limit $u$ is  continuous on $\R$;
\item \label{l:conv-equiv-R:unif}
$u_n\to u $  uniformly on $\R$, 
and the limit $u$ is  continuous on $\R$;
\item \label{l:conv-equiv-R:kappa}
\begin{enumerate}
	\item \label{lem:i:narrow-conv}$\kappa_n $ converges narrowly to $\kappa$,
	\item \label{lem:i:no-atoms} $\kappa$ has no atoms, and
	\item \label{lem:i:AEC-kappa} there exist a sequence $s_n\xrightarrow{n\to\infty}0$ and a modulus of continuity $\omega$ such that 
	\begin{equation}
	\label{cond:AEC-kappa}
		\text{for all } -\infty< x\leq y<\infty, \qquad 
		  |\kappa_n((x,y])|\leq s_n + \omega(|x-y|).
	\end{equation}
	\end{enumerate}
\end{enumerate}
In all these cases the limit $u$ is uniformly continuous on $\R$.
\end{lem}
\noindent
Note that by choosing $u_n$ and $u$ to be right-continuous, we have the characterization
\begin{equation}
	\label{eq:u-kappa-rel}	
	u_n(x) = \kappa_n\bigl((-\infty,x]\bigr) \qquad \text{and}\qquad
	u(x) = \kappa\bigl((-\infty,x]\bigr) \qquad \text{for all }x\in\R.
	\end{equation}

\begin{rem}
\label{rem:kappa-conv-is-topological}
Since  uniform convergence on the space of bounded functions is generated by a norm topology, conditions~\ref{lem:i:narrow-conv} and~\ref{lem:i:AEC-kappa} together characterize the corresponding topology on any bounded and tight set of measures. 
\end{rem}

Before proving Lemma~\ref{l:conv-equiv-R} we first consider the question of continuity of the limit. First note  that the continuity of the limit $u$ and the absence of atoms in~$\kappa$ are equivalent properties. These are required in the second, third, and fourth equivalent characterizations of Lemma~\ref{l:conv-equiv-R}, but not in the first. This is because  continuous convergence implies continuity of the limit, \emph{even if the sequence itself consists of discontinuous functions}:
\begin{lem}
\label{l:ct-conv-ct}
Let $X$ be a metric space and let $v_n,v:X\to\R$. If $v_n\to v$ continuously, then~$v$ is continuous. 
\end{lem}

\noindent
This Lemma also explains why the limit $u$ in Theorem~\ref{t} is continuous, despite being the limit of discontinuous functions $u_n$: it stems from the fact that Theorem~\ref{t} proves continuous convergence of $u_n$, by establishing equality in the inequality~\eqref{pf:t}.

\begin{proof}[Proof of Lemma~\ref{l:ct-conv-ct}]
Fix $x_m, x\in X$ with $x_m\to x$. For each $m$, use the continuous convergence to choose $n_m > n_{m-1}$ such that $|v_{n_m}(x_m)-v(x_m)|<1/m$. Let us define the sequence $y_n = x_m$ if $n = n_m$ for some $m$ and $y_n = x$ otherwise. Since $y_n \to x$, we have $v_n(y_n) \to v(x)$ by the continuous convergence and hence $v_{n_m}(x_m) \to v(x)$ since $\{v_{n_m}(x_m)\}$ is a subsequence of $\{v_n(y_n)\}$.
Therefore
\begin{align*}
0 &\leq |v(x_m)-v(x)| \leq |v(x_m)-v_{n_m}(x_m)| + |v_{n_m}(x_m) - v(x)| 
\\&< 1/m + |v_{n_m}(x_m) - v(x)| \to 0 \quad \text{as $m \to \infty$.}
\end{align*}
\end{proof}

\begin{proof}[Proof of Lemma~\ref{l:conv-equiv-R}]	
We show $\ref{l:conv-equiv-R:ct} \Longleftrightarrow \ref{l:conv-equiv-R:locunif}$ and $\ref{l:conv-equiv-R:locunif} \Longrightarrow \ref{l:conv-equiv-R:unif}\Longrightarrow \ref{l:conv-equiv-R:kappa}\Longrightarrow \ref{l:conv-equiv-R:locunif}$. 
The final statement follows from the fact that any continuous function of finite variation is uniformly continuous.

The implication $\ref{l:conv-equiv-R:locunif}\Longrightarrow \ref{l:conv-equiv-R:ct}$ follows  directly from the triangle inequality. 
To show $\ref{l:conv-equiv-R:ct} \Longrightarrow \ref{l:conv-equiv-R:locunif}$, first note that Lemma~\ref{l:ct-conv-ct} implies that $u$ is continuous.  Let $K\subset \R$ be a compact set, let   $\{n_m\}$ be any subsequence, and  let  $x_m\in K$ be such that $|u_{n_m}(x_m) - u(x_m)|> \sup_K |u_{n_m} - u|-1/m$. By taking subsequences (not relabeled), we may assume that $x_m \to x$. But then 
\[
\sup_K |u_{n_m} - u| < \frac1m + |u_{n_m}(x_m) - u(x)| + |u(x) - u(x_m)|,
\]
which converges to 0 by the continuity of $u$ and the continuous convergence of $u_n$. We conclude that $\sup_K |u_n- u| \to 0$.

To show $\ref{l:conv-equiv-R:locunif} \Longrightarrow\ref{l:conv-equiv-R:unif}$, fix $\e>0$.  By the  tightness of $\kappa_n$ there exists $R>0$ such that 
\[
\sup_n \sup_{x > R} |u_n(x)-u_n(R)| < \e.
\]
Since $u$ has finite total variation, we can similarly assume that $R$ is such that 
\begin{equation}
\label{est:u-tot-var}	
\sup_{x > R} |u(x)-u(R)| < \e.
\end{equation}
Then, it follows directly from \ref{l:conv-equiv-R:locunif} that 
\[
  \sup_{[-R,R]} |u_n-u| < \e
\] 
for all $n$ large enough. By the triangle inequality
\[
\sup_{x > R} |u_n(x)-u(x)|
\leq \sup_{x > R} |u_n(x)-u_n(R)| + |u_n(R)-u(R)| + \sup_{x > R}|u(R)-u(x)|,
\]
which is smaller than $3 \e$ for all $n$ large enough by the three displays above. Analogously, one can derive that $\sup_{x < -R} |u_n(x)-u(x)| < 3 \e$ for all $n$ large enough.

To show $\ref{l:conv-equiv-R:unif}\Longrightarrow \ref{l:conv-equiv-R:kappa}$, note that the narrow convergence of $\kappa_n$ is a simple extension of e.g.~\cite[Prop.~1.4.9]{Bogachev18} to unbounded domains, and that we already remarked that the continuity of $u$ is equivalent to the absence of atoms in $\kappa$. To prove Property \eqref{lem:i:AEC-kappa}, let $\omega$ be a modulus of continuity of the limit $u$ on $\R$, and set $s_n := 2\sup_\R |u_n-u|$. For any $n$ and any $-\infty< x\leq y <\infty$ we then have
\begin{align*}
|\kappa_n((x,y])| &= |u_n(y)-u_n(x)| \leq |u_n(y)-u(y)| + |u(y)-u(x)| + |u(x)-u_n(x)|\\
&\leq s_n + \omega(|x-y|).
\end{align*}

Finally we prove $\ref{l:conv-equiv-R:kappa}\Longrightarrow \ref{l:conv-equiv-R:locunif}$. The condition~\eqref{cond:AEC-kappa} implies that $u_n$ satisfies the conditions of the generalization~\cite[Prop.~3.3.1]{AmbrosioGigliSavare08} of the classical Arzel\`a-Ascoli theorem. For each $R$, we therefore find a subsequence $u_{n_k}$ that converges uniformly on $[-R,R]$ to a continuous limit. Since the narrow convergence of $\kappa_n$ to $\kappa$ uniquely characterizes  this limit as $u$, we conclude that the whole sequence converges uniformly on $[-R,R]$ to $u$. 
\end{proof}

\bigskip
The lemmas above deal with convergence of functions and measures on $\R$. We now extend the statements to functions and measures on $[0,T]\times \R$.

\begin{lem}
\label{l:conv-equiv-TR}
Let $u_n,u:[0,T]\times 	\R\to \R$ be such that  $u_n(t,\cdot)$ and $u(t,\cdot)$ have finite variation on $\R$  for all $t\in [0,T]$, and let $\{\kappa_n(t)\}_{t\in [0,T]}$ and $\{\kappa(t)\}_{t\in [0,T]}$ be families of signed measures on $\R$. Assume that 
\begin{enumerate}[A.]
	\item the family $\{\kappa_n(t)\}_{n\in \N,t\in[0,T]}$ is bounded in total variation on $\R$ and tight;
	\item for all $t\in [0,T], \ x\in \R$,
\[
u_n(t,x) = \kappa_n(t,(-\infty,x])\qquad\text{and}\qquad
u(t,x) = \kappa(t,(-\infty,x]).
\]
\end{enumerate}


Then the following are equivalent:
\begin{enumerate}
\item \label{l:conv-equiv-TR:ct}
$u_n\to u$ continuously on $[0,T]\times \R$;
\item \label{l:conv-equiv-TR:locunif}
$u_n\to u$ locally uniformly on $[0,T]\times \R$, and $u$ is continuous on $[0,T]\times \R$;
\item \label{l:conv-equiv-TR:unif}
$u_n\to u$  uniformly on $[0,T]\times \R$, and $u$ is continuous on $[0,T]\times \R$;
\item \label{l:conv-equiv-TR:kappa}
For each sequence $t_n\to t$ in $[0,T]$, $\kappa_n(t_n)$ converges to $\kappa(t)$ in the sense of Lemma~\ref{l:conv-equiv-R}. 
The sequence $(s_n)_n$ and the modulus of continuity $\omega$ can be chosen to be uniform in the sequence $t_n\to t$.
\end{enumerate}
The limit function $u$ is uniformly continuous on $[0,T]\times \R$. 
\end{lem}

\begin{proof}
The implication $\ref{l:conv-equiv-TR:unif} \Longrightarrow \ref{l:conv-equiv-TR:locunif}$ is immediate. The implications $\ref{l:conv-equiv-TR:ct}\Longleftrightarrow \ref{l:conv-equiv-TR:locunif}$ are proved in the same way as in Lemma~\ref{l:conv-equiv-R}, and we omit the details. 
The implication $\ref{l:conv-equiv-TR:locunif}\Longrightarrow\ref{l:conv-equiv-TR:unif}$ also follows the same argument, with an additional step to establish the corresponding generalization of~\eqref{est:u-tot-var}, which is the existence of $R>0$ such that 
\[
\sup_{t\in[0,T]} \sup_{x > R} |u(t,x)-u(t,R)| < \e.
\]
We prove the existence of such an $R$ by noting that  for $x > R$, 
\begin{align*}
|u(t,x)-u(t,R)| = \lim_{n\to\infty} |u_n(t,x)-u_n(t,R)|
 = \lim_{n\to\infty} \bigl|\kappa_n\bigl(t,	(R,x]\bigr)\bigr| 
 \leq \liminf_{n\to\infty} |\kappa_n(t)|\bigl([R,\infty)\bigr),
\end{align*}
and this final expression vanishes uniformly in $t$ as $R\to\infty$ by the assumed tightness of $\{\kappa_n(t)\}_{n,t}$. This observation also implies that the limit $u$ is uniformly continuous on $[0,T]\times\R$. 

Finally, we prove $\ref{l:conv-equiv-TR:ct},\ref{l:conv-equiv-TR:unif} \Longleftrightarrow \ref{l:conv-equiv-TR:kappa}$. We first note that the continuous convergence of $u_n$ to $u$ on $[0,T]\times \R$ is equivalent to the property 
\[
\text{for all }t_n\to t \text{ in }[0,T], \quad 
u_n(t_n,\cdot) \to u(t,\cdot) \text{ continuously on }\R.
\]
Then, Lemma \ref{l:conv-equiv-R} implies that $\ref{l:conv-equiv-TR:ct},\ref{l:conv-equiv-TR:unif}$ are equivalent to $\ref{l:conv-equiv-TR:kappa}$. The fact that the parameters $s_n$ and~$\omega$ can be chosen to be independent of the sequence $t_n$ follows from an inspection of the proof of Lemma~\ref{l:conv-equiv-R}: we can take $\omega$ to be a modulus of continuity of $u$ on $[0,T]\times \R$, and we can set $s_n := 2\sup\bigl\{|u_n(t,x)-u(t,x)|: (t,x) \in [0,T]\times \R\bigr\}$.
\end{proof}

\bigskip
Finally we use Lemma~\ref{l:conv-equiv-TR} to translate Theorem~\ref{t} into properties of $\kappa_n$ and $\kappa$:

\begin{cor}
\label{c:conv-kappa}
Let a sequence of initial data $(\bx_n^\circ, \bb_n^\circ) \in \cZ_n$ be such that 
the corresponding measures  $\kappa_n^\circ$ defined by~\eqref{def:kappa_n} converge to  $\kappa^\circ$ in the sense of Lemma~\ref{l:conv-equiv-R}.
Let $(\bx_n, \bb_n)$ be the solution of \eqref{Pn} (Definition \ref{defn:Pn}) with initial data $(\bx^\circ_n, \bb_n^\circ)$. 

Then the measures $\kappa_n$ generated by $(\bx_n, \bb_n)$  converge to $\kappa$ in the sense of Lemma~\ref{l:conv-equiv-TR}. 
%
%
The limit $\kappa$ is characterized by the property that its integral $u(t,x) = \kappa(t, (-\infty,x])$ is the unique viscosity solution of equation~\eqref{HJ:formal} in the sense of Definition \ref{def:HJ:VS} with initial datum $u^\circ(x) = \kappa^\circ((-\infty,x])$.
\end{cor}

\begin{proof}
Theorem~\ref{t} provides continuous convergence of the corresponding functions $u_n$. 
In order to apply  Lemma~\ref{l:conv-equiv-TR}  we only need to establish tightness of the solutions $\{\kappa_n(t)\}_{n,t}$. 
Since by assumption the sequence $\kappa^\circ_n$ converges narrowly, by Prokhorov's theorem it is tight~\cite[Th.~8.6.2]{Bogachev07.II}, and by a well-known characterization of tightness~\cite[Rem.~5.1.5]{AmbrosioGigliSavare08} there exists a  function $\varphi:\R\to[0,\infty)$ with compact sublevel sets such that $\sup_n \int\varphi |\kappa^\circ_n| < \infty$. By passing to the Moreau-Yosida regularization $0\leq \varphi_\lambda\leq \varphi$ we can assume without loss of generality that $\varphi$ is differentiable with Lipschitz continuous derivative (see e.g.~\cite[Th.~XV.4.1.4]{Hiriart-UrrutyLemarechal96}). 
Similar to the calculation in~\eqref{pf:ddt:Mk1} we then calculate
\[
\frac{d}{dt} \int \varphi(x) |\kappa_n(t)|(dx) 
= \frac1{n^2} \sum_{i=1}^n \sum_{j=1}^{i-1} b_ib_j \frac{\varphi'(x_i)-\varphi'(x_j)}{x_i-x_j}
\leq \frac12 \Lip(\varphi'),
\]
implying that the set $\{\kappa_n(t)\}_{n,t}$ is tight. 
\end{proof}

The continuous convergence of $\kappa_n$ to $\kappa$ as measure-valued functions of time leads to continuity of the limit by Lemma~\ref{l:ct-conv-ct}:
\begin{cor}
\label{c:kappa-limit-ct}
The map $t\mapsto \kappa(t)$ from the interval $[0,T]$ into the space of finite signed measures is continuous in the sense of Lemma~\ref{l:conv-equiv-R}. 	
\end{cor}

\begin{proof}
By Lemma~\ref{l:ct-conv-ct} the limit $t\mapsto \kappa(t)$ is sequentially continuous in the sense of part~\ref{l:conv-equiv-R:kappa} of Lemma~\ref{l:conv-equiv-R}; by  Remark~\ref{rem:kappa-conv-is-topological} this convergence  is generated by a metric topology, implying that sequential continuity and continuity are equivalent. 
\end{proof}

\paragraph{Acknowledgements}
PvM gratefully acknowledges support from JSPS KAKENHI Grant Number 20K14358. NP is supported by JSPS KAKENHI Grant Number 18K13440. We gratefully acknowledge support from Kanazawa University for hosting MAP.

\bibliographystyle{alphainitials}

\begin{thebibliography}{ADLGP16}

\bibitem[ADLGP14]{AlicandroDeLucaGarroniPonsiglione14}
R.~Alicandro, L.~De~Luca, A.~Garroni, and M.~Ponsiglione.
\newblock Metastability and dynamics of discrete topological singularities in
  two dimensions: a {$\Gamma$}-convergence approach.
\newblock {\em Archive for Rational Mechanics and Analysis}, 214(1):269--330,
  2014.

\bibitem[ADLGP16]{AlicandroDeLucaGarroniPonsiglione16}
R.~Alicandro, L.~De~Luca, A.~Garroni, and M.~Ponsiglione.
\newblock Dynamics of discrete screw dislocations on glide directions.
\newblock {\em Journal of the Mechanics and Physics of Solids}, 92:87--104,
  2016.

\bibitem[AGS08]{AmbrosioGigliSavare08}
L.~Ambrosio, N.~Gigli, and G.~Savar{\'e}.
\newblock {\em Gradient Flows: In Metric Spaces and in the Space of Probability
  Measures}.
\newblock Birkh\"auser Verlag, New York, 2008.

\bibitem[AMS11]{AmbrosioMaininiSerfaty11}
L.~Ambrosio, E.~Mainini, and S.~Serfaty.
\newblock Gradient flow of the {C}hapman--{R}ubinstein--{S}chatzman model for
  signed vortices.
\newblock In {\em Annales de l'Institut Henri Poincare (C) Non Linear
  Analysis}, volume~28, pages 217--246. Elsevier, 2011.

\bibitem[Ber06]{Berdichevsky06}
V.~L. Berdichevsky.
\newblock On thermodynamics of crystal plasticity.
\newblock {\em Scripta Materialia}, 54(5):711--716, 2006.

\bibitem[BKM10]{BilerKarchMonneau10}
P.~Biler, G.~Karch, and R.~Monneau.
\newblock Nonlinear diffusion of dislocation density and self-similar
  solutions.
\newblock {\em Communications in Mathematical Physics}, 294(1):145--168, 2010.

\bibitem[Bog07]{Bogachev07.II}
V.~I. Bogachev.
\newblock {\em Measure Theory}, volume~2.
\newblock Springer, 2007.

\bibitem[Bog18]{Bogachev18}
V.~I. Bogachev.
\newblock {\em Weak convergence of measures}.
\newblock American Mathematical Society, 2018.

\bibitem[CIL92]{CrandallIshiiLions92}
M.~G. Crandall, H.~Ishii, and P.-L. Lions.
\newblock User's guide to viscosity solutions of second order partial
  differential equations.
\newblock {\em Bull. Amer. Math. Soc. (N.S.)}, 27(1):1--67, 1992.

\bibitem[CL83]{CrandallLions83}
M.~G. Crandall and P.-L. Lions.
\newblock Viscosity solutions of {H}amilton-{J}acobi equations.
\newblock {\em Trans. Amer. Math. Soc.}, 277(1):1--42, 1983.

\bibitem[CL05]{CermelliLeoni06}
P.~Cermelli and G.~Leoni.
\newblock Renormalized energy and forces on dislocations.
\newblock {\em SIAM Journal on Mathematical Analysis}, 37(4):1131--1160, 2005.

\bibitem[CMP15]{ChambolleMoriniPonsiglione15}
A.~Chambolle, M.~Morini, and M.~Ponsiglione.
\newblock Nonlocal curvature flows.
\newblock {\em Arch. Ration. Mech. Anal.}, 218(3):1263--1329, 2015.

\bibitem[DFES20]{DiFrancescoEspositoSchmidtchen20ArXiv}
M.~Di~Francesco, A.~Esposito, and M.~Schmidtchen.
\newblock Many-particle limit for a system of interaction equations driven by
  {N}ewtonian potentials.
\newblock {\em ArXiv: 2008.11106}, 2020.

\bibitem[DFF13]{DiFrancescoFagioli13}
M.~Di~Francesco and S.~Fagioli.
\newblock Measure solutions for non-local interaction pdes with two species.
\newblock {\em Nonlinearity}, 26(10):2777, 2013.

\bibitem[DFF16]{DiFrancescoFagioli16}
M.~Di~Francesco and S.~Fagioli.
\newblock A nonlocal swarm model for predators--prey interactions.
\newblock {\em Mathematical Models and Methods in Applied Sciences},
  26(02):319--355, 2016.

\bibitem[Due16]{Duerinckx16}
M.~Duerinckx.
\newblock Mean-field limits for some {R}iesz interaction gradient flows.
\newblock {\em SIAM Journal on Mathematical Analysis}, 48(3):2269--2300, 2016.

\bibitem[FG07]{FocardiGarroni07}
M.~Focardi and A.~Garroni.
\newblock A {1D} macroscopic phase field model for dislocations and a second
  order {$\Gamma$}-limit.
\newblock {\em Multiscale Modeling \& Simulation}, 6(4):1098--1124, 2007.

\bibitem[FIM09]{ForcadelImbertMonneau09}
N.~Forcadel, C.~Imbert, and R.~Monneau.
\newblock Homogenization of the dislocation dynamics and of some particle
  systems with two-body interactions.
\newblock {\em Discrete and Continuous Dynamical Systems A}, 23(3):785--826,
  2009.

\bibitem[GB99]{GromaBalogh99}
I.~Groma and P.~Balogh.
\newblock Investigation of dislocation pattern formation in a two-dimensional
  self-consistent field approximation.
\newblock {\em Acta Materialia}, 47(13):3647--3654, 1999.

\bibitem[GCZ03]{GromaCsikorZaiser03}
I.~Groma, F.~F. Csikor, and M.~Zaiser.
\newblock Spatial correlations and higher-order gradient terms in a continuum
  description of dislocation dynamics.
\newblock {\em Acta Materialia}, 51(5):1271--1281, 2003.

\bibitem[GGK06]{GromaGyorgyiKocsis06}
I.~Groma, G.~Gy{\"o}rgyi, and B.~Kocsis.
\newblock Debye screening of dislocations.
\newblock {\em Physical Review Letters}, 96(16):165503, 2006.

\bibitem[Gig06]{Giga06}
Y.~Giga.
\newblock {\em Surface evolution equations}, volume~99 of {\em Monographs in
  Mathematics}.
\newblock Birkh\"{a}user Verlag, Basel, 2006.
\newblock A level set approach.

\bibitem[GPPS13]{GeersPeerlingsPeletierScardia13}
M.~G.~D. Geers, R.~H.~J. Peerlings, M.~A. Peletier, and L.~Scardia.
\newblock Asymptotic behaviour of a pile-up of infinite walls of edge
  dislocations.
\newblock {\em Archive for Rational Mechanics and Analysis}, 209:495--539,
  2013.

\bibitem[GvMPS20]{GarroniVanMeursPeletierScardia19}
A.~Garroni, P.~{v}an Meurs, M.~A. Peletier, and L.~Scardia.
\newblock Convergence and non-convergence of many-particle evolutions with
  multiple signs.
\newblock {\em Archive for Rational Mechanics and Analysis}, 235(1):3--49,
  2020.

\bibitem[Hal11]{Hall11}
C.~L. Hall.
\newblock Asymptotic analysis of a pile-up of regular edge dislocation walls.
\newblock {\em Materials Science and Engineering: A}, 530:144--148, 2011.

\bibitem[Hau09]{Hauray09}
M.~Hauray.
\newblock Wasserstein distances for vortices approximation of {E}uler-type
  equations.
\newblock {\em Mathematical Models and Methods in Applied Sciences},
  19(08):1357--1384, 2009.

\bibitem[HB01]{HullBacon01}
D.~Hull and D.~J. Bacon.
\newblock {\em Introduction to Dislocations}.
\newblock Butterworth Heinemann, Oxford, 2001.

\bibitem[HL55]{HeadLouat55}
A.~K. Head and N.~Louat.
\newblock The distribution of dislocations in linear arrays.
\newblock {\em Australian Journal of Physics}, 8(1):1--7, 1955.

\bibitem[HL82]{HirthLothe82}
J.~P. Hirth and J.~Lothe.
\newblock {\em Theory of Dislocations}.
\newblock John Wiley \& Sons, New York, 1982.

\bibitem[HO14]{HudsonOrtner14}
T.~Hudson and C.~Ortner.
\newblock Existence and stability of a screw dislocation under anti-plane
  deformation.
\newblock {\em Archive for Rational Mechanics and Analysis}, 213(3):887--929,
  2014.

\bibitem[HUL96]{Hiriart-UrrutyLemarechal96}
J.-B. Hiriart-Urruty and C.~Lemar{\'e}chal.
\newblock {\em Convex Analysis and Minimization Algorithms}.
\newblock Springer-Verlag, 1996.

\bibitem[HvMP20]{HudsonVanMeursPeletier20ArXiv}
T.~Hudson, P.~van Meurs, and M.~A. Peletier.
\newblock Atomistic origins of continuum dislocation dynamics.
\newblock {\em ArXiv: 2001.06120}, 2020.

\bibitem[IMR08]{ImbertMonneauRouy08}
C.~Imbert, R.~Monneau, and E.~Rouy.
\newblock Homogenization of first order equations with
  (u/$\varepsilon$)-periodic {H}amiltonians {P}art {II}: {A}pplication to
  dislocations dynamics.
\newblock {\em Communications in Partial Differential Equations},
  33(3):479--516, 2008.

\bibitem[IS95]{IshiiSouganidis95}
H.~Ishii and P.~Souganidis.
\newblock Generalized motion of noncompact hypersurfaces with velocity having
  arbitrary growth on the curvature tensor.
\newblock {\em Tohoku Math. J. (2)}, 47(2):227--250, 1995.

\bibitem[JK05]{JakobsenKarlsen05}
E.~R. Jakobsen and K.~H. Karlsen.
\newblock Continuous dependence estimates for viscosity solutions of
  integro-{PDE}s.
\newblock {\em J. Differential Equations}, 212(2):278--318, 2005.

\bibitem[KHG15]{KooimanHuetterGeers15b}
M.~Kooiman, M.~H{\"u}tter, and M.~G.~D. Geers.
\newblock Effective mobility of dislocations from systematic coarse-graining.
\newblock {\em Journal of Statistical Mechanics}, 2015(6):P06005, 2015.

\bibitem[MP12a]{MonneauPatrizi12a}
R.~Monneau and S.~Patrizi.
\newblock Derivation of {O}rowan's law from the {P}eierls--{N}abarro model.
\newblock {\em Communications in Partial Differential Equations},
  37(10):1887--1911, 2012.

\bibitem[MP12b]{MonneauPatrizi12}
R.~Monneau and S.~Patrizi.
\newblock Homogenization of the {P}eierls--{N}abarro model for dislocation
  dynamics.
\newblock {\em Journal of Differential Equations}, 253(7):2064--2105, 2012.

\bibitem[MPS17]{MoraPeletierScardia17}
M.~G. Mora, M.~A. Peletier, and L.~Scardia.
\newblock Convergence of interaction-driven evolutions of dislocations with
  {W}asserstein dissipation and slip-plane confinement.
\newblock {\em SIAM Journal on Mathematical Analysis}, 49(5):4149--4205, 2017.

\bibitem[OS97]{OhnumaSato97}
M.~Ohnuma and K.~Sato.
\newblock Singular degenerate parabolic equations with applications to the
  {$p$}-{L}aplace diffusion equation.
\newblock {\em Comm. Partial Differential Equations}, 22(3-4):381--411, 1997.

\bibitem[PS17]{PetracheSerfaty14}
M.~Petrache and S.~Serfaty.
\newblock Next order asymptotics and renormalized energy for {R}iesz
  interactions.
\newblock {\em Journal of the Institute of Mathematics of Jussieu},
  16(3):501--569, 2017.

\bibitem[PV15]{PatriziValdinoci15}
S.~Patrizi and E.~Valdinoci.
\newblock Crystal dislocations with different orientations and collisions.
\newblock {\em Arch. Ration. Mech. Anal.}, 217(1):231--261, 2015.

\bibitem[PV16]{PatriziValdinoci16}
S.~Patrizi and E.~Valdinoci.
\newblock Relaxation times for atom dislocations in crystals.
\newblock {\em Calc. Var. Partial Differential Equations}, 55(3):Art. 71, 44,
  2016.

\bibitem[Say91]{Sayah91}
A.~Sayah.
\newblock Equqtions d'{H}amilton-{J}acobi du premier ordre avec termes
  int\'{e}gro-diff\'{e}rentiels.
\newblock {\em Communications in Partial Differential Equations},
  16(6-7):1057--1074, 1991.

\bibitem[SBO07]{SmetsBethuelOrlandi07}
D.~Smets, F.~Bethuel, and G.~Orlandi.
\newblock Quantization and motion law for {G}inzburg--{L}andau vortices.
\newblock {\em Archive for Rational Mechanics and Analysis}, 183(2):315--370,
  2007.

\bibitem[Sch96]{Schochet96}
S.~Schochet.
\newblock The point-vortex method for periodic weak solutions of the 2-{D}
  {E}uler equations.
\newblock {\em Communications on Pure and Applied Mathematics}, 49(9):911--965,
  1996.

\bibitem[Ser07]{Serfaty07II}
S.~Serfaty.
\newblock Vortex collisions and energy-dissipation rates in the
  {G}inzburg--{L}andau heat flow. {P}art {II}: {T}he dynamics.
\newblock {\em Journal of the European Mathematical Society}, 9(3):383--426,
  2007.

\bibitem[Sle03]{Slepcev03}
D.~Slep\v{c}ev.
\newblock Approximation schemes for propagation of fronts with nonlocal
  velocities and {N}eumann boundary conditions.
\newblock {\em Nonlinear Anal.}, 52(1):79--115, 2003.

\bibitem[vMM14]{VanMeursMuntean14}
P.~{v}an Meurs and A.~Muntean.
\newblock Upscaling of the dynamics of dislocation walls.
\newblock {\em Advances in Mathematical Sciences and Applications},
  24(2):401--414, 2014.

\bibitem[vMM19]{VanMeursMorandotti19}
P.~van Meurs and M.~Morandotti.
\newblock Discrete-to-continuum limits of particles with an annihilation rule.
\newblock {\em SIAM Journal on Applied Mathematics}, 79(5):1940--1966, 2019.

\end{thebibliography}

\end{document}